\documentclass[11pt,reqno,a4paper]{amsart}
\usepackage{a4wide}
\usepackage{amsmath}
\usepackage{amssymb}
\usepackage{amsthm}
\usepackage{amstext}
\usepackage{geometry}
\usepackage{fancyhdr}
\usepackage{array}
\usepackage{graphicx}
\usepackage{hyperref}
\usepackage{fullpage}
\usepackage{cite}
\usepackage{bm}
\allowdisplaybreaks
\hypersetup{
colorlinks=true,
linkcolor=blue,
urlcolor=red,
citecolor=green,
}

\allowdisplaybreaks[4]

\newtheorem{theorem}{Theorem}[section]

\newtheorem{corollary}[theorem]{Corollary}

\newtheorem{lemma}[theorem]{Lemma}
\newtheorem{proposition}[theorem]{Proposition}
\newtheorem{remark}[theorem]{Remark}

\numberwithin{equation}{section}

\newtheorem{notations}[theorem]{Notations}

\begin{document}

\title{Quantitative analysis of ground states for the fractional logarithmic Schr\"odinger equation}

\author{Xiaoming An}
\address{[Xiaoming An] School of Mathematics and Statistics \& Guizhou University of Finance and Economics, Guiyang, 550025, P. R. China}
\email{xman@mail.gufe.edu.cn}

\author{Shuangjie Peng}
\address{[Shuangjie Peng] School of Mathematics and Statistics, Key Laboratory of Nonlinear Analysis and Applications (Ministry of Education), Central China Normal University, Wuhan, 430079, P. R. China}
\email{sjpeng@mail.ccnu.edu.cn} 

\author{Fulin Zhong$^{\dagger}$}
\address{[Fulin Zhong] School of Mathematics and Statistics, Central China Normal University, Wuhan 430079, P. R. China}
\email{flzhong@mails.ccnu.edu.cn}

\thanks{$^{\dagger}$ Corresponding author: Fulin Zhong}


\begin{abstract} 
Let $N\geq1$ and $0<s<1$. We study positive ground states of the fractional logarithmic Schr\"odinger equation
\begin{equation*}
(-\Delta)^sQ=Q\log Q \quad\text{in }\mathbb{R}^N.
\end{equation*}
We prove that for every $N\geq1$ and $0<s<1$, the positive ground state is unique up to translations and nondegenerate. More precisely, for the linearized operator $L_Q=(-\Delta)^s-1-\log Q$, it holds that
\begin{equation*}
\ker L_Q=\operatorname{span}\{\partial_{x_1}Q,\cdots,\partial_{x_N}Q\}.
\end{equation*}
A main difficulty is that the potential $-1-\log Q$ is unbounded in $\mathbb{R}^N$, which prevents a direct application of the available radial oscillation theory for fractional Schr\"odinger operators with bounded potentials. We overcome this difficulty by a bounded-potential approximation. Using also the fact that the associated quadratic form has Morse index one and an angular decomposition, we obtain the nondegeneracy. Based on the isolation of logarithmic ground states and the uniqueness theory for the fractional power equation, we prove uniqueness by a variational approximation with subcritical power nonlinearities. As an application, we establish sharp fractional logarithmic Sobolev inequalities and characterize all cases of equality.

\vspace{0.25cm}


\vspace{0.25cm}

\end{abstract}

\maketitle
\tableofcontents

\section{Introduction}

Let $N\geq1$ and $0<s<1$. We study the fractional logarithmic Schr\"odinger equation
\begin{equation}\label{eq:log}
(-\Delta)^s u=u\log|u| \quad\text{in }\mathbb{R}^N,
\end{equation}
where $0\log0=0$ and $(-\Delta)^s$ denotes the fractional Laplacian with Fourier symbol $|\xi|^{2s}$. For sufficiently regular functions $u$, the fractional Laplacian has the principal value representation
\begin{equation*}
(-\Delta)^s u(x)=c_{N,s}\mathrm{P.V.}\int_{\mathbb{R}^N}\frac{u(x)-u(y)}{|x-y|^{N+2s}}\mathrm{d}y,
\end{equation*}
where
\begin{equation*}
c_{N,s}=\left(\int_{\mathbb{R}^N}\frac{1-\cos\zeta_1}{|\zeta|^{N+2s}}\mathrm{d}\zeta\right)^{-1}.
\end{equation*}
We refer to \cite{Nezza-BSM-2012} for the basic properties of the fractional Laplacian and fractional Sobolev spaces. 

The fractional Schr\"odinger equation was introduced by Laskin in the framework of fractional quantum mechanics based on L\'evy processes \cite{Laskin-2000,Laskin-2002}. Its nonlinear version can be written as
\begin{equation}\label{aeq1.2}
\mathrm{i}\frac{\partial\psi}{\partial t}=(-\Delta)^s\psi-f(\psi) \quad\text{in }\mathbb{R}^N\times\mathbb{R}.
\end{equation}
For the logarithmic nonlinearity $f(z)=z\log|z|$, the standing-wave ansatz
$
\psi(x,t)=\mathrm{e}^{-\mathrm{i}E t}v(x)
$
gives
\begin{equation*}
(-\Delta)^sv=v\log|v|+E v\quad\text{in }\mathbb{R}^N.
\end{equation*}
Hence, after the rescaling $u=\mathrm{e}^{E}v$, we obtain \eqref{eq:log}. 
For the power nonlinearity $f(z)=|z|^\alpha z$ in \eqref{aeq1.2}, the same standing-wave reduction leads, after normalization, to the power-law fractional Schr\"odinger equation
\begin{equation}\label{aeq1.3}
(-\Delta)^sW+W=|W|^\alpha W
\quad\text{in }\mathbb{R}^N,
\quad
0<\alpha<\alpha_*(s,N),
\end{equation}
where
\begin{equation*}
\alpha_*(s,N)=
\begin{cases}
\frac{4s}{N-2s}, & N>2s,\\
+\infty, & N\leq2s.
\end{cases}
\end{equation*} 

In this paper, we are concerned with the uniqueness and nondegeneracy of ground states of \eqref{eq:log}, for their important application in mathematical physics, such as the stability, blow-up, and long-time analysis of solitary waves. See, for example, \cite{Weinstein-1985,Weinstein-1987,Chang-Gustafson-Nakanishi-Tsai-2007,Kenig-Martel-Robbiano-2011,Frank-Lenzmann-2013,Frank-Lenzmann-Silvestre-2016} and the references therein.

In the classical case $s=1$, uniqueness and nondegeneracy for \eqref{eq:log} and \eqref{aeq1.3} have been extensively studied. For the power-law equation \eqref{aeq1.3}, see \cite{Coffman-ARMA-1972,Kwong-1989,Kwong-Zhang-DIE-1991,Coffman-JDE-1996,Mcleod-Serrin-ARMA-1987} for the uniqueness of positive solutions. It is noteworthy that Tang \cite{Tang-Invention-2026} recently gave a positive answer to a conjecture of Berestycki and Lions \cite{Berestycki-Lions-1983a,Berestycki-Lions-1983b} on the uniqueness of bound states (sign-changing radial solutions) for \eqref{aeq1.3}. For the logarithmic equation \eqref{eq:log}, Serrin and Tang \cite{Serrin-Tang-Indiana-2000} proved uniqueness of positive ground states for $N\geq3$, while Troy \cite{Troy-ARMA-2016} proved uniqueness for $1\leq N\leq9$ by a different comparison argument. Zhang and Zhang \cite{Zhang-Zhang-JFPT-2022} obtained uniqueness of positive radial solutions for $N\geq2$ under general assumptions on the potential. Uniqueness results for \eqref{eq:log} with general logarithmic nonlinearities, as well as uniqueness and nondegeneracy results in the presence of a small constant magnetic field, can be found in \cite{An-Fang-2025,An-Peng-Yang-Zhong-2026}. For further recent results on bound states of \eqref{eq:log}, see \cite{Liu-Sun-Zou-2026}. 

For $0<s<1$, since $(-\Delta)^s$ is nonlocal, ODE tools such as Sturm comparison, Wronskian identities and shooting arguments, which are commonly used in the local case $s=1$, are no longer applicable. Before the work of Frank and Lenzmann \cite{Frank-Lenzmann-2013} and Frank, Lenzmann, and Silvestre \cite{Frank-Lenzmann-Silvestre-2016}, uniqueness of positive ground states for the fractional power-law equation \eqref{aeq1.3} remained open in general, and only a few results were known in several special cases \cite{Amick-Toland-ActaMath-1991,Li-JEMS-2004,Chen-Li-CPAM-2006}. Using the extension technique of Caffarelli and Silvestre \cite{Caffarelli-Silverstre-CPDE-2007}, they established a radial monotonicity formula and proved a uniqueness result for radial solutions of the corresponding linear equation. This implies that the radial eigenvalues are simple. They then used the extension problem and a continuation argument to prove that the second radial eigenfunction changes sign exactly once. This oscillation result is used to prove nondegeneracy. Then, with the nondegeneracy result and the uniqueness in the classical case $s=1$, they used implicit function theorem and a continuation argument in $s$ to establish the uniqueness result for the full range $0<s<1$. 

The fractional logarithmic Schr\"odinger equation has attracted considerable research interest in recent years. D'Avenia, Squassina, and Zenari \cite{DAvenia-Squassina-Zenari-2015} proved the existence of infinitely many weak solutions. Ardila \cite{Ardila-2017} established the existence of ground states and their orbital stability. Li, Peng, and Shuai \cite{Li-Peng-Shuai-2022} studied positive and sign-changing solutions in the presence of variable potentials. Motivated by the power-to-logarithm limit in \cite{Wang-Zhang-2019}, An and Yang \cite{An-Yang-2023} proved that, after a suitable rescaling, ground states of \eqref{aeq1.3} converge as $\alpha\to0^+$ to a ground state of the fractional logarithmic equation. 

To the best of our knowledge, uniqueness and nondegeneracy of positive ground states of \eqref{eq:log} have remained open. In this paper, we establish these two properties for the full range $N\geq1$ and $0<s<1$. 
To state our main results, we first introduce some necessary notations. Following \cite{Nezza-BSM-2012}, the fractional Sobolev space is defined by
\begin{equation*}
H^s(\mathbb{R}^N)=\left\{u\in L^2(\mathbb{R}^N):\iint_{\mathbb{R}^{2N}}\frac{|u(x)-u(y)|^2}{|x-y|^{N+2s}}\mathrm{d}x\mathrm{d}y<\infty\right\}.
\end{equation*}
Equipped with the inner product
\begin{equation*}
\langle u,v\rangle_{H^s(\mathbb{R}^N)}=\frac{c_{N,s}}{2}\iint_{\mathbb{R}^{2N}}\frac{(u(x)-u(y))(v(x)-v(y))}{|x-y|^{N+2s}}\mathrm{d}x\mathrm{d}y+\int_{\mathbb{R}^N}uv\,\mathrm{d}x
\end{equation*}
and the corresponding norm
\begin{equation*}
\|u\|_{H^s(\mathbb{R}^N)}^2=\frac{c_{N,s}}{2}\iint_{\mathbb{R}^{2N}}\frac{|u(x)-u(y)|^2}{|x-y|^{N+2s}}\mathrm{d}x\mathrm{d}y+\|u\|_{L^2(\mathbb{R}^N)}^2,
\end{equation*}
the space $H^s(\mathbb{R}^N)$ is a Hilbert space. Moreover, by \cite[Proposition 3.6]{Nezza-BSM-2012},
\begin{equation*}
\frac{c_{N,s}}{2}\iint_{\mathbb{R}^{2N}}\frac{(u(x)-u(y))(v(x)-v(y))}{|x-y|^{N+2s}}\mathrm{d}x\mathrm{d}y=\left\langle(-\Delta)^{\frac{s}{2}}u,(-\Delta)^{\frac{s}{2}}v\right\rangle_{L^2(\mathbb{R}^N)}
\end{equation*}
for every $u,v\in H^s(\mathbb{R}^N)$. 

For $u\in H^s(\mathbb{R}^N)$, set
\begin{equation*}
A(u)=\int_{\mathbb{R}^N}\left|(-\Delta)^{\frac{s}{2}}u\right|^2\mathrm{d}x,\quad B(u)=\int_{\mathbb{R}^N}|u|^2\mathrm{d}x.
\end{equation*}
In particular,
\begin{equation} \label{eq:gagliardo}
A(u)=\frac{c_{N,s}}{2}\iint_{\mathbb{R}^{2N}}\frac{|u(x)-u(y)|^2}{|x-y|^{N+2s}}\mathrm{d}x\mathrm{d}y.
\end{equation} 
Define
\begin{equation*}
\mathcal{D}_s=\left\{u\in H^s(\mathbb{R}^N):\int_{\mathbb{R}^N}u^2|\log|u||\mathrm{d}x<\infty\right\}.
\end{equation*}
For $u\in\mathcal{D}_s$, set
\begin{equation*}
C_0(u)=\int_{\mathbb{R}^N}|u|^2\log|u|\mathrm{d}x,\quad \mathcal{I}_0(u)=\frac{1}{2}A(u)+\frac{1}{4}B(u)-\frac{1}{2}C_0(u).
\end{equation*}
A function $u\in\mathcal{D}_s$ is a weak solution of \eqref{eq:log} if
\begin{equation*}
\left\langle(-\Delta)^{\frac{s}{2}}u,(-\Delta)^{\frac{s}{2}}\varphi\right\rangle_{L^2(\mathbb{R}^N)}
=\int_{\mathbb{R}^N}u\log|u|\varphi\mathrm{d}x
\end{equation*}
for every $\varphi\in C_c^\infty(\mathbb{R}^N)$. The associated Nehari set is
\begin{equation*}
\mathcal N_0=\left\{u\in\mathcal{D}_s\setminus\{0\}:A(u)=C_0(u)\right\}.
\end{equation*}
A ground state is a weak solution minimizing $\mathcal{I}_0$ on $\mathcal N_0$, see \cite{Ardila-2017,Li-Peng-Shuai-2022}. 

For a positive ground state $Q$, set
\begin{equation*}
V_Q=-1-\log Q,\quad L_Q=(-\Delta)^s+V_Q.
\end{equation*}
The precise definition of $L_Q$ is given in Lemma \ref{lem:LQ-domain}.

Our first main result establishes the uniqueness and nondegeneracy of positive ground states of \eqref{eq:log}.

\begin{theorem}\label{thm:main} 
Let $N\geq1$ and $0<s<1$.
\begin{enumerate}
\item If $Q$ is a positive radial ground state of \eqref{eq:log}, then
\begin{equation*}
	\ker(L_Q)=\operatorname{span}\left\{\partial_{x_1}Q,\cdots,\partial_{x_N}Q\right\}.
\end{equation*}
\item Equation \eqref{eq:log} has exactly one positive ground state up to translations.
\end{enumerate}
\end{theorem}

Sharp logarithmic Sobolev inequalities in the classical case $s=1$ were obtained in \cite{Gross-1975,Carlen-1991,DelPino-Dolbeault-2003}. For fractional derivatives, Cotsiolis and Tavoularis \cite{Cotsiolis-Tavoularis-2005} claimed a sharp logarithmic Sobolev inequality. Chatzakou and Ruzhansky \cite{Chatzakou-Ruzhansky-2024} later pointed out an incompatibility in the exponent choices used in that proof and established a revised inequality with an explicit constant for $0<s<\frac{N}{2}$. As an application of Theorem \ref{thm:main}, our second main result establishes sharp fractional logarithmic Sobolev inequalities for every $N\geq1$ and $0<s<1$, with optimal constants determined by the mass of the unique positive ground state, and characterizes all equality cases.

\begin{theorem}\label{th1.2}
Let $N\geq1$, $0<s<1$, and let $Q$ be the radial positive ground state of \eqref{eq:log}. Then, for every $u\in\mathcal{D}_s\setminus\{0\}$,
\begin{equation*}
\int_{\mathbb{R}^N}|u|^2\log|u|^2\mathrm{d}x\leq2A(u)+B(u)\log\left(\frac{B(u)}{B(Q)}\right).
\end{equation*}
Equality holds if and only if
\begin{equation*}
u(x)=cQ(x-y)
\end{equation*}
for some $c\in\mathbb{R}\setminus\{0\}$ and $y\in\mathbb{R}^N$. Moreover, for every $a>0$,
\begin{equation*}
\int_{\mathbb{R}^N}|u|^2\log|u|^2\mathrm{d}x \leq 2\mathrm{e}^{2s}B(Q)^{-\frac{2s}{N}}a^{2s}A(u)+\left[\log B(u)-N\left(1+\log a\right)\right]B(u).
\end{equation*}
For fixed $a>0$, equality holds if and only if
\begin{equation*}
u(x)=cQ\left(\frac{B(Q)^{\frac{1}{N}}}{\mathrm{e}a}(x-y)\right)
\end{equation*}
for some $c\in\mathbb{R}\setminus\{0\}$ and $y\in\mathbb{R}^N$.
\end{theorem}

We briefly describe the main ideas of the proof. The main difficulty in the nondegeneracy analysis is that the linearized potential
\begin{equation*}
V_Q=-1-\log Q
\end{equation*}
is unbounded, since $V_Q(r)\to+\infty$ as $r\to+\infty$. The oscillation theory in \cite[Theorem 2.3]{Frank-Lenzmann-Silvestre-2016} requires bounded radial potentials and therefore cannot be applied to $L_Q$ to obtain the one-node property of the second radial eigenfunction. To overcome this difficulty, we approximate $V_Q$ by the bounded radial potentials
\begin{equation*}
V_m(r)=\min\{V_Q(r),m\},\quad L_m=(-\Delta)^s+V_m.
\end{equation*}
For all sufficiently large $m$, the oscillation theorem in \cite{Frank-Lenzmann-Silvestre-2016} yields the one-node property of the second radial eigenfunction of $L_m$. We then establish the convergence of the first two radial eigenvalues and the associated eigenfunctions as $m\to\infty$, which allows us to pass this nodal property to a second radial eigenfunction of the original operator $L_Q$.

The fixed-mass variational characterization shows that the quadratic form associated with $L_Q$ has Morse index one. Using this fact, the one-node property of the second radial eigenfunction, and two orthogonality identities obtained from the equation and a scaling argument, we exclude nontrivial radial elements of $\ker L_Q$. For $N\geq2$, the spherical-harmonic decomposition and the strict ordering of the angular sectors give
\begin{equation*}
\ker(L_Q)=\operatorname{span}\left\{\partial_{x_1}Q,\cdots,\partial_{x_N}Q\right\}.
\end{equation*}
When $N=1$, the even sector is the radial sector and the odd sector is treated directly by the quadratic form.

The uniqueness proof is based on the fixed-mass variational characterization and a power approximation. We first prove compactness of radial ground states. The radial nondegeneracy then implies that every radial logarithmic ground state is isolated in $L^2(\mathbb{R}^N)$. For powers $2+\varepsilon$ close to $2$, we minimize the corresponding power functional in a small neighborhood of a fixed logarithmic ground state. The resulting minimizer converges to the prescribed logarithmic ground state, and its linearized quadratic form also has Morse index one. Moreover, after a suitable rescaling, it becomes a positive solution of
\begin{equation*}
(-\Delta)^sW+W=W^{1+\varepsilon} \quad\text{in }\mathbb{R}^N.
\end{equation*}
These properties and the uniqueness theory in \cite{Frank-Lenzmann-Silvestre-2016} then imply that two distinct logarithmic ground states cannot occur. This proves the uniqueness of the positive ground state of \eqref{eq:log} up to translations.

Finally, the fixed-mass characterization, together with the uniqueness result and a scaling argument, yields the sharp fractional logarithmic Sobolev inequalities in Theorem \ref{th1.2} and determines all equality cases.

\medskip
\textbf{Plan of the paper}. The paper is organized as follows. Section \ref{sec:variational-spectral} collects the basic properties and variational characterization of ground states of \eqref{eq:log} and develops the spectral framework for the linearized operator. In Section \ref{sec:nondegeneracy}, we combine the radial nodal structure with the analysis of the nonradial sectors to prove the nondegeneracy of positive ground states of \eqref{eq:log}. Section \ref{sec:compactness-isolation} establishes compactness of radial minimizing sequences and the local isolation of radial ground states. In Section \ref{sec:power-approximation}, we combine a local power approximation with the uniqueness theory for the fractional power equation to prove uniqueness of the logarithmic ground state. Finally, Section \ref{sec:log-sobolev-applications} derives the sharp logarithmic Sobolev inequalities and characterizes all equality cases.

\medskip

\section{Preliminaries}\label{sec:variational-spectral}

In this section, we collect some basic facts that will be used throughout the paper, including the regularity, symmetry, radial monotonicity and decay of positive ground states of \eqref{eq:log}, the fixed-mass variational characterization, several basic properties of the fractional Laplacian, and the self-adjoint realization and radial spectral properties of the linearized operator $L_Q$.

\begin{notations}\label{not2.1}
We use the following notation throughout the paper.
\begin{itemize}
\item For $R>0$, we set
\begin{equation*}
	B_R=\left\{x\in\mathbb{R}^N:|x|<R\right\}.
\end{equation*}

\item For a real-valued function $f$, we write
\begin{equation*}
	f_+=\max\{f,0\},\quad f_-=\max\{-f,0\}.
\end{equation*}

\item For a measurable set $E\subset\mathbb{R}^N$, $|E|$ denotes its Lebesgue measure and $\mathbf{1}_E$ its characteristic function, while $|\mathbb{S}^{N-1}|$ denotes the surface measure of the unit sphere $\mathbb{S}^{N-1}$.

\item For $f,g\in L^2(\mathbb{R}^N)$, we write $f\perp g$ if $\langle f,g\rangle_{L^2(\mathbb{R}^N)}=0$. For $0\neq g\in L^2(\mathbb{R}^N)$, we set
\begin{equation*}
	g^\perp=\left\{f\in L^2(\mathbb{R}^N):\langle f,g\rangle_{L^2(\mathbb{R}^N)}=0\right\}.
\end{equation*}

\item We set
\begin{equation*}
	2_s^*=
	\begin{cases}
		\frac{2N}{N-2s}, & N>2s,\\
		+\infty, & N\leq2s.
	\end{cases}
\end{equation*}
\end{itemize}
Throughout the paper, all functions are real-valued and all function spaces are considered over $\mathbb{R}$.
\end{notations}

\subsection{Basic estimates and properties of ground states}

This subsection collects several estimates for the fractional Dirichlet form and the logarithmic nonlinearity, recalls a fractional interpolation inequality, and presents the regularity, radial symmetry, radial monotonicity, and decay of positive ground states of \eqref{eq:log}. 

\begin{lemma}\label{lem:elementary-log}
For every $\delta>0$, there exists $C_\delta>0$ such that
\begin{equation}\label{eq:log-growth}
t^2\log t \leq C_\delta t^{2+\delta} \quad\text{for every }t\geq0.
\end{equation}
Moreover,
\begin{equation}\label{eq:log-negative}
0\leq-t^2\log t\leq\frac{1}{2\mathrm{e}} \quad\text{for }0\leq t\leq1.
\end{equation} 
\end{lemma}

\begin{proof}
For $t\geq1$, set $f(t)=t^{-\delta}\log t$. Then
\begin{equation*}
f'(t)=t^{-\delta-1}\left(1-\delta\log t\right).
\end{equation*}
Hence $f$ attains its maximum at $t=\mathrm{e}^{\frac{1}{\delta}}$ and 
\begin{equation*}
t^2\log t\leq\frac{1}{\delta\mathrm{e}}t^{2+\delta} \quad\text{for }t\geq1.
\end{equation*}
For $0\leq t\leq1$, we have $t^2\log t\leq0$. This proves \eqref{eq:log-growth}.

Let $g(t)=-t^2\log t$ on $[0,1]$, with $g(0)=0$. Since
\begin{equation*}
g'(t)=-t\left(2\log t+1\right),
\end{equation*}
the function $g$ attains its maximum at $t=\mathrm{e}^{-\frac{1}{2}}$. Thus
\begin{equation*}
\max_{0\leq t\leq1}g(t)=g\left(\mathrm{e}^{-\frac{1}{2}}\right)=\frac{1}{2\mathrm{e}},
\end{equation*}
which proves \eqref{eq:log-negative}. 
\end{proof}

\begin{lemma}\label{lem:fractional-gn}
Let $2<q<2_s^*$ and set
\begin{equation*}
\theta_q=\frac{N(q-2)}{4s}.
\end{equation*}
Then there exists $C=C(N,s,q)>0$ such that
\begin{equation*}
\|u\|_{L^q(\mathbb{R}^N)}^q
\leq C A(u)^{\theta_q}B(u)^{\frac{q}{2}-\theta_q}
\quad\text{for every }u\in H^s(\mathbb{R}^N).
\end{equation*}
\end{lemma}

\begin{proof}
The fractional Gagliardo--Nirenberg inequality
\cite[(3.4)]{Frank-Lenzmann-Silvestre-2016}, with $\alpha=q-2$, gives
\begin{equation*}
\|u\|_{L^q(\mathbb{R}^N)}
\leq C\|(-\Delta)^{\frac{s}{2}}u\|_{L^2(\mathbb{R}^N)}^{\vartheta}
\|u\|_{L^2(\mathbb{R}^N)}^{1-\vartheta},
\quad \vartheta=\frac{N(q-2)}{2sq}.
\end{equation*}
Raising this inequality to the power $q$ and using
\begin{equation*}
\frac{q\vartheta}{2}=\frac{N(q-2)}{4s}=\theta_q
\end{equation*}
gives the conclusion.
\end{proof}

\begin{lemma}\label{lem:modulus}
For every $u\in H^s(\mathbb{R}^N)$, we have $|u|\in H^s(\mathbb{R}^N)$ and
\begin{equation}\label{eq:modulus}
A\left(|u|\right)\leq A(u).
\end{equation}
Equality holds if and only if $u\geq0$ almost everywhere or $u\leq0$ almost everywhere.
\end{lemma}

\begin{proof}
The pointwise inequality
\begin{equation*}
\left||u(x)|-|u(y)|\right|\leq|u(x)-u(y)|
\end{equation*}
and the Gagliardo representation \eqref{eq:gagliardo} give $|u|\in H^s(\mathbb{R}^N)$ and \eqref{eq:modulus}. If equality holds, then
\begin{equation*}
u(x)u(y)\geq0 \quad\text{for almost every }(x,y)\in\mathbb{R}^{2N}.
\end{equation*}
Hence $u$ has a constant sign. The converse is immediate.
\end{proof}

By symmetric decreasing rearrangement, see \cite[Section 3.3]{Lieb-Loss-2001} and \cite[Chapter 15]{Leoni-2017}, we obtain the following radial reduction.

\begin{lemma}\label{lem:rearrangement}
Let $u\in H^s(\mathbb{R}^N)$ and let $u^*$ be the symmetric decreasing rearrangement of $|u|$. Then
\begin{equation*}
B(u^*)=B(u), \quad A(u^*)\leq A(u).
\end{equation*}
If $u\in\mathcal{D}_s$, then $u^*\in\mathcal{D}_s$ and
\begin{equation*}
C_0(u^*)=C_0(u).
\end{equation*}
Moreover, for every $q\in[1,+\infty)$ and $u\in L^q(\mathbb{R}^N)$,
\begin{equation*}
\|u^*\|_{L^q(\mathbb{R}^N)}=\|u\|_{L^q(\mathbb{R}^N)}.
\end{equation*}
\end{lemma}

\begin{proof}
By \cite[Section 3.3]{Lieb-Loss-2001} and \cite[Theorem 15.10]{Leoni-2017}, $u^*$ and $|u|$ are equimeasurable. Hence
\begin{equation*}
B(u^*)=B(u), \quad \|u^*\|_{L^q(\mathbb{R}^N)}=\|u\|_{L^q(\mathbb{R}^N)}
\end{equation*}
for $u\in L^q(\mathbb{R}^N)$. If $u\in\mathcal{D}_s$, the same theorem applied to $t^2|\log t|$ shows that $u^*\in\mathcal{D}_s$ and 
\begin{equation*}
C_0(u^*)=C_0(u).
\end{equation*}
By the fractional P\'olya--Szeg\H{o} inequality \cite[Theorem 9.2-(1)]{Almgren-Lieb-1989} (see also \cite[Theorem A.1]{Frank-Seiringer-2008}) and Lemma \ref{lem:modulus},
\begin{equation*}
A(u^*)\leq A(|u|)\leq A(u).
\end{equation*} 
\end{proof}

Since the linearized potential $V_Q$ is unbounded, we will use cutoff arguments in the spectral analysis of $L_Q$. Therefore, we record the following localization identity, which allows us to pass from distributional identities to the quadratic-form setting.

Hereafter, $H^s_{\mathrm{loc}}(\mathbb{R}^N)$ denotes the space of functions $w$ such that $\chi w\in H^s(\mathbb{R}^N)$ for every $\chi\in C_c^\infty(\mathbb{R}^N)$.

\begin{lemma}\label{lem:cutoff}
Let $\eta\in C_c^\infty(\mathbb{R}^N)$ satisfy $0\leq\eta\leq1$ and set $\eta_R(x)=\eta\left(\frac{x}{R}\right)$. For every $w\in H^s(\mathbb{R}^N)$,
\begin{equation}\label{eq:cutid}
A(\eta_Rw) = \left\langle(-\Delta)^sw,\eta_R^2w\right\rangle + \mathcal{E}_R(w),
\end{equation}
where $\langle\cdot,\cdot\rangle$ denotes the duality pairing between $H^{-s}(\mathbb{R}^N)$ and $H^s(\mathbb{R}^N)$, and
\begin{equation*}
\mathcal{E}_R(w) = \frac{c_{N,s}}{2} \iint_{\mathbb{R}^{2N}} \frac{w(x)w(y)\left(\eta_R(x)-\eta_R(y)\right)^2} {|x-y|^{N+2s}}\mathrm{d}x\mathrm{d}y
\end{equation*}
satisfying
\begin{equation}\label{eq:cuter}
|\mathcal{E}_R(w)| \leq CR^{-2s}\|w\|_{L^2(\mathbb{R}^N)}^2.
\end{equation}
The identity also extends to $w\in L^2(\mathbb{R}^N)\cap H^s_{\mathrm{loc}}(\mathbb{R}^N)$ satisfying $(-\Delta)^sw=F$ in $\mathcal D'(\mathbb{R}^N)$ with $F\in L^2_{\mathrm{loc}}(\mathbb{R}^N)$. More precisely,
\begin{equation*}
A(\eta_Rw)=\int_{\mathbb{R}^N}F\eta_R^2w\mathrm{d}x+\mathcal E_R(w).
\end{equation*}
\end{lemma}

\begin{proof}
Using \eqref{eq:gagliardo}, we compute
\begin{equation*}
\begin{aligned}
	&\left(\eta_R(x)w(x)-\eta_R(y)w(y)\right)^2 - \left(w(x)-w(y)\right) \left(\eta_R(x)^2w(x)-\eta_R(y)^2w(y)\right) \\
	&= w(x)w(y)\left(\eta_R(x)-\eta_R(y)\right)^2.
\end{aligned}
\end{equation*}
Integration gives \eqref{eq:cutid}. Moreover, using Cauchy-Schwarz inequality and Fubini Theorem,
\begin{equation*}
\begin{aligned}
	|\mathcal{E}_R(w)| \leq& \frac{c_{N,s}}{4} \iint_{\mathbb{R}^{2N}} \frac{\left(w(x)^2+w(y)^2\right) \left(\eta_R(x)-\eta_R(y)\right)^2} {|x-y|^{N+2s}}\mathrm{d}x\mathrm{d}y \\
	=& \frac{c_{N,s}}{2} \int_{\mathbb{R}^N}w(x)^2 \left( \int_{\mathbb{R}^N} \frac{\left(\eta_R(x)-\eta_R(y)\right)^2} {|x-y|^{N+2s}}\mathrm{d}y \right)\mathrm{d}x.
\end{aligned}
\end{equation*}
With $y=x+Rz$ and
\begin{equation*}
|\eta(\xi)-\eta(\xi+z)| \leq C\min\{|z|,1\},
\end{equation*}
we obtain
\begin{equation} \label{eq:symmetry}
\begin{aligned}
	\int_{\mathbb{R}^N}\frac{(\eta_R(x)-\eta_R(y))^2}{|x-y|^{N+2s}} \mathrm{d}y
	&=R^{-2s} \int_{\mathbb{R}^N} \frac{\left(\eta\left(\frac{x}{R}\right)-\eta\left(\frac{x}{R}+z\right)\right)^2} {|z|^{N+2s}}\mathrm{d}z\\
	&\leq CR^{-2s}\left(\int_{|z|\leq1}\frac{|z|^2}{|z|^{N+2s}} \mathrm{d}z+\int_{|z|>1}\frac{1}{|z|^{N+2s}} \mathrm{d}z\right)\\
	&\leq CR^{-2s}.
\end{aligned}
\end{equation}
This proves \eqref{eq:cuter}.

Now let $w\in L^2(\mathbb{R}^N)\cap H^s_{\mathrm{loc}}(\mathbb{R}^N)$ and suppose that $(-\Delta)^sw=F$ in $\mathcal{D}'(\mathbb{R}^N)$ with $F\in L^2_{\mathrm{loc}}(\mathbb{R}^N)$. Choose $L>0$ such that $\operatorname{supp}\eta_R\subset B_L$. Let $\theta_k\in C_c^\infty(\mathbb{R}^N)$ satisfy
\begin{equation*}
0\leq\theta_k\leq1,\quad \theta_k=1\ \text{on }B_k,\quad \theta_k=0\ \text{on }\mathbb{R}^N\setminus B_{2k},
\end{equation*}
and set $w_k=\theta_kw$. By the definition of $H^s_{\mathrm{loc}}(\mathbb{R}^N)$, we have $w_k\in H^s(\mathbb{R}^N)$. Moreover, $w_k\to w$ in $L^2(\mathbb{R}^N)$ by the dominated convergence theorem. For $k>2(L+1)$,
\begin{equation*}
\eta_Rw_k=\eta_Rw,\quad \eta_R^2w_k=\eta_R^2w.
\end{equation*}
Applying \eqref{eq:cutid} to $w_k$, we obtain
\begin{equation}\label{eq:cutwk}
A(\eta_Rw)=\langle(-\Delta)^sw_k,\eta_R^2w\rangle+\mathcal E_R(w_k).
\end{equation} 
For $x\in B_{L+1}$, define
\begin{equation*}
G_k(x)=c_{N,s}\int_{\mathbb{R}^N}\frac{(1-\theta_k(y))w(y)}{|x-y|^{N+2s}}\mathrm{d}y.
\end{equation*}
Since $w_k-w=-(1-\theta_k)w$ vanishes on $B_k$, we have
\begin{equation*}
(-\Delta)^s(w_k-w)(x)=G_k(x).
\end{equation*}
Therefore, the equation for $w$ yields
\begin{equation} \label{eq:distributional}
(-\Delta)^sw_k=F+G_k\quad\text{in }\mathcal{D}'(B_{L+1}).
\end{equation}
Moreover, for $x\in B_{L+1}$ and $|y|>k$, we have $|x-y|\geq\frac{|y|}{2}$ when $k>2(L+1)$. Hence, by Cauchy-Schwarz inequality, we have
\begin{equation*}
\begin{aligned}
	\|G_k\|_{L^\infty(B_{L+1})}
	&\leq C\int_{|y|>k}\frac{|w(y)|}{|y|^{N+2s}}\mathrm{d}y\\
	&\leq C\|w\|_{L^2(\mathbb{R}^N)}\left(\int_{|y|>k}|y|^{-2N-4s}\mathrm{d}y\right)^{\frac{1}{2}}\\
	&\leq Ck^{-\frac{N}{2}-2s}\|w\|_{L^2(\mathbb{R}^N)}\to0.
\end{aligned}
\end{equation*}

Since $\eta_R^2w\in H^s(\mathbb{R}^N)$ and $\operatorname{supp}(\eta_R^2w)\subset B_L$, there exist $\phi_j\in C_c^\infty(B_{L+1})$ such that $\phi_j\to\eta_R^2w$ in $H^s(\mathbb{R}^N)$.
For every $j$, using \eqref{eq:distributional},
\begin{equation*}
\langle(-\Delta)^sw_k,\phi_j\rangle=\int_{B_{L+1}}(F+G_k)\phi_j\mathrm{d}x.
\end{equation*}
Since $w_k\in H^s(\mathbb{R}^N)$, the left-hand side is continuous with respect to $H^s$ convergence. Since $F+G_k\in L^2(B_{L+1})$, the right-hand side is continuous with respect to $L^2$ convergence. Letting $j\to+\infty$, we obtain
\begin{equation*}
\langle(-\Delta)^sw_k,\eta_R^2w\rangle
=\int_{\mathbb{R}^N}F\eta_R^2w\mathrm{d}x+\int_{B_{L+1}}G_k\eta_R^2w\mathrm{d}x.
\end{equation*} 
By the definition of $\mathcal E_R$, the symmetry in $x$ and $y$, and the Cauchy-Schwarz inequality,
\begin{equation*}
\begin{aligned}
	|\mathcal E_R(w_k)-\mathcal E_R(w)|
	&\leq \frac{c_{N,s}}{2}\iint_{\mathbb{R}^{2N}}\frac{|w_k(x)-w(x)||w_k(y)+w(y)|(\eta_R(x)-\eta_R(y))^2}{|x-y|^{N+2s}}\mathrm{d}x\mathrm{d}y\\
	&\leq \frac{c_{N,s}}{2}\left(\int_{\mathbb{R}^N}|w_k(x)-w(x)|^2\int_{\mathbb{R}^N}\frac{(\eta_R(x)-\eta_R(y))^2}{|x-y|^{N+2s}}\mathrm{d}y\mathrm{d}x\right)^{\frac{1}{2}}\\
	&\quad\times\left(\int_{\mathbb{R}^N}|w_k(y)+w(y)|^2\int_{\mathbb{R}^N}\frac{(\eta_R(x)-\eta_R(y))^2}{|x-y|^{N+2s}}\mathrm{d}x\mathrm{d}y\right)^{\frac{1}{2}}\\
	&\leq CR^{-2s}\|w_k-w\|_{L^2(\mathbb{R}^N)}\|w_k+w\|_{L^2(\mathbb{R}^N)}\to0,
\end{aligned}
\end{equation*}
where the last inequality follows from \eqref{eq:symmetry}.
Also, since $\eta_R^2w$ has compact support and belongs to $L^2(\mathbb{R}^N)$, we have $\eta_R^2w\in L^1(\mathbb{R}^N)$. Therefore,
\begin{equation*}
\left|\int_{B_{L+1}}G_k\eta_R^2w\mathrm{d}x\right|
\leq\|G_k\|_{L^\infty(B_{L+1})}\|\eta_R^2w\|_{L^1(\mathbb{R}^N)}
\to0.
\end{equation*}
Letting $k\to+\infty$ in \eqref{eq:cutwk}, we conclude that
\begin{equation*}
A(\eta_Rw)=\int_{\mathbb{R}^N}F\eta_R^2w\mathrm{d}x+\mathcal E_R(w).
\end{equation*}
This proves the claimed extension of \eqref{eq:cutid}.
\end{proof}

\begin{proposition}\label{prop:known}
Let $Q$ be a positive ground state of \eqref{eq:log}. After a translation, $Q$ is radial about the origin and
\begin{equation*}
Q(r)>0, \quad Q'(r)<0 \quad\text{for }r>0, \quad Q(r)\to0 \quad\text{as }r\to\infty.
\end{equation*}
Moreover,
\begin{equation}\label{eq:qreg}
Q\in C^\infty(\mathbb{R}^N), \quad Q(x)\leq\frac{C}{1+|x|^{N+2s}}.
\end{equation}
Consequently,
\begin{equation}\label{eq:qint}
Q\in H^{2s}(\mathbb{R}^N)\cap L^1(\mathbb{R}^N), \quad \nabla Q\in L^2(\mathbb{R}^N), \quad x\cdot\nabla Q\in L^1(\mathbb{R}^N), \quad Q\log Q\in L^2(\mathbb{R}^N).
\end{equation}
Finally, $V_Q=-1-\log Q$ is smooth, bounded from below, strictly increasing in $r$ and $V_Q(r)\to+\infty$ as $r\to\infty$.
\end{proposition}

\begin{proof}
Fix $a>1$ and set
$
\widetilde{Q}(x)=aQ\left(2^{\frac{1}{2s}}x\right).
$
Then using $(-\Delta)^sQ=Q\log Q$, we have
\begin{equation*}
(-\Delta)^s\widetilde{Q}
=2\widetilde{Q}\log\left(\frac{\widetilde{Q}}{a}\right)
=\widetilde{Q}\log\widetilde{Q}^2-2\log(a)\widetilde{Q},
\end{equation*}
and hence
\begin{equation*}
(-\Delta)^s\widetilde{Q}+2\log(a)\widetilde{Q}
=\widetilde{Q}\log\widetilde{Q}^2.
\end{equation*}
The estimates in \cite[Lemma 2.4 and Proposition 2.5]{Li-Peng-Shuai-2022} give
\begin{equation*}
\widetilde{Q}\in L^\infty(\mathbb{R}^N)\cap C^{0,\alpha_0}(\mathbb{R}^N)
\end{equation*}
for some $\alpha_0\in(0,1)$. Choose $\sigma\in(0,1)$ and $\eta\in(0,2s)$ such that
$
\frac{2s-\eta}{1-\sigma}>1,
$
and define
\begin{equation*}
F_a(t)=t\log t^2-2\log(a)t, \quad F_a(0)=0, \quad \alpha_{k+1}=2s-\eta+\sigma\alpha_k.
\end{equation*}
Since $F_a\in C_{\mathrm{loc}}^{0,\sigma}([0,+\infty))$, the Schauder estimate
\cite[Theorem 1.1-(b)]{RosOton-Serra-2016} gives
\begin{equation*}
\widetilde{Q}\in C^{0,\alpha_k}(\mathbb{R}^N)
\Longrightarrow
F_a(\widetilde{Q})\in C^{0,\sigma\alpha_k}(\mathbb{R}^N)
\Longrightarrow
\widetilde{Q}\in C^{0,\alpha_{k+1}}(\mathbb{R}^N),
\end{equation*}
where
$
\alpha_{k+1}=2s-\eta+\sigma\alpha_k
$
whenever $\alpha_{k+1}<1$. Since
\begin{equation*}
\alpha_k\to\frac{2s-\eta}{1-\sigma}>1,
\end{equation*}
there exists $k$ such that $\alpha_k< 1$ and
$
2s-\eta+\sigma\alpha_k\geq1.
$
Since $F_a(\widetilde{Q})\in C^{0,\sigma\alpha_k}(\mathbb{R}^N)$, another application of \cite[Theorem 1.1-(b)]{RosOton-Serra-2016} yields
\begin{equation*}
\widetilde{Q}\in C^{1,\beta}(\mathbb{R}^N)
\end{equation*}
for some $\beta>0$. Fix $R>0$. Since $\widetilde{Q}>0$, $F_a$ is smooth on the range of $\widetilde{Q}$ in $B_{2R}$. Repeated application of \cite[Corollary 3.5]{RosOton-Serra-2016} to
\begin{equation*}
(-\Delta)^s\widetilde{Q}=F_a(\widetilde{Q})
\end{equation*}
yields $\widetilde{Q}\in C^\infty(B_R)$.
Since $R>0$ is arbitrary, $Q\in C^\infty(\mathbb{R}^N)$.

Since $\widetilde Q\in C^{0,\alpha_0}(\mathbb{R}^N)\cap L^2(\mathbb{R}^N)$, we have $\widetilde Q(x)\to0$ as $|x|\to+\infty$. The comparison argument in the proof of \cite[Lemma 3.9]{Li-Peng-Shuai-2022}, applied with the constant potential $V_0=2\log a>0$, then gives \eqref{eq:qreg}. The moving-plane argument in \cite[Theorem 4.1]{Chen-Li-2018} gives, after a translation, that $Q=Q(r)$ is nonincreasing in $r$. We prove that the monotonicity is strict. Let
\begin{equation*}
H=\{x=(x_1,x')\in\mathbb{R}^N:x_1>0\},\quad w=-\partial_{x_1}Q.
\end{equation*}
Since $Q$ is radial and nonincreasing,
\begin{equation*}
w(x)=-Q'(|x|)\frac{x_1}{|x|}\geq0\quad\text{in }H.
\end{equation*}
Moreover, $w$ is antisymmetric with respect to $\partial H$. Differentiating $(-\Delta)^sQ=Q\log Q$, we obtain
\begin{equation*}
(-\Delta)^sw=(1+\log Q)w.
\end{equation*}
Suppose that $w(x_0)=0$ for some $x_0\in H$. Since $w\geq0$ in $H$, the point $x_0$ is a local minimum of $w$. Since $w\in C^2_{\mathrm{loc}}(\mathbb{R}^N)\cap L^\infty(\mathbb{R}^N)$ and $x_0$ is a local minimum of $w$, we have $\nabla w(x_0)=0$ and
\begin{equation*}
w(y)-w(x_0)=O(|y-x_0|^2)\quad\text{as }y\to x_0.
\end{equation*}
Since $0<s<1$, the pointwise integral representation of $(-\Delta)^sw$ is therefore valid at $x_0$. For $y=(y_1,y')\in H$, set $y^*=(-y_1,y')$. Using $w(y^*)=-w(y)$, we obtain
\begin{equation*}
(-\Delta)^sw(x_0)=-c_{N,s}\int_Hw(y)\left(\frac{1}{|x_0-y|^{N+2s}}-\frac{1}{|x_0-y^*|^{N+2s}}\right)\mathrm{d}y.
\end{equation*}
For $x_0,y\in H$,
\begin{equation*}
|x_0-y^*|^2-|x_0-y|^2=4(x_0)_1y_1>0,
\end{equation*}
and hence the expression in parentheses is positive. Moreover, $w\not\equiv0$ in $H$, since otherwise $Q'(r)=0$ for every $r>0$, contradicting $Q(r)\to0$ as $r\to+\infty$. Therefore
\begin{equation*}
(-\Delta)^sw(x_0)<0.
\end{equation*}
On the other hand,
\begin{equation*}
(-\Delta)^sw(x_0)=(1+\log Q(x_0))w(x_0)=0,
\end{equation*}
a contradiction. Therefore $w>0$ in $H$ and hence
\begin{equation*}
Q'(r)<0\quad\text{for every }r>0.
\end{equation*}

The decay in \eqref{eq:qreg} gives
\begin{equation*}
\int_{\mathbb{R}^N}Q\mathrm{d}x \leq C+C|\mathbb{S}^{N-1}|\int_1^\infty r^{-1-2s}\mathrm{d}r <\infty.
\end{equation*}
Since $Q'\leq0$ and $r^NQ(r)\to0$,
\begin{equation*}
\int_{\mathbb{R}^N}|x\cdot\nabla Q|\mathrm{d}x = |\mathbb{S}^{N-1}|\int_0^\infty(-Q'(r))r^N\mathrm{d}r = N\int_{\mathbb{R}^N}Q\mathrm{d}x.
\end{equation*}
Moreover,
\begin{equation*}
\int_{\mathbb{R}^N}|\nabla Q|^2\mathrm{d}x \leq \|Q'\|_{L^\infty(0,\infty)}|\mathbb{S}^{N-1}| \int_0^\infty(-Q'(r))r^{N-1}\mathrm{d}r<\infty,
\end{equation*}
where
\begin{equation*}
\int_0^\infty(-Q'(r))r^{N-1}\mathrm{d}r
=
\begin{cases}
	Q(0), & N=1, \\
	(N-1)\displaystyle\int_0^\infty Q(r)r^{N-2}\mathrm{d}r, & N\geq2.
\end{cases}
\end{equation*}
Choose $R>0$ such that $Q\leq1$ on $\mathbb{R}^N\setminus B_R$ and choose $\delta\in(0,1)$ with $2(1-\delta)(N+2s)>N$. Then
\begin{equation*}
Q\log Q\in L^\infty(B_R), \quad |Q\log Q|^2\leq C_\delta Q^{2(1-\delta)} \quad\text{on }\mathbb{R}^N\setminus B_R,
\end{equation*}
so $Q\log Q\in L^2(\mathbb{R}^N)$. Therefore
\begin{equation*}
(-\Delta)^sQ=Q\log Q\in L^2(\mathbb{R}^N) \Longrightarrow Q\in H^{2s}(\mathbb{R}^N).
\end{equation*}
Finally,
\begin{equation}\label{eq:qwgt}
\int_{\mathbb{R}^N}(V_Q)_+Q^2\mathrm{d}x \leq B(Q)+\int_{\mathbb{R}^N}Q^2|\log Q|\mathrm{d}x <\infty.
\end{equation}

Since $Q>0$ and $Q\in C^\infty(\mathbb{R}^N)$, we have
\begin{equation*}
V_Q=-1-\log Q\in C^\infty(\mathbb{R}^N).
\end{equation*} 
Also,
\begin{equation*}
V_Q\geq-1-\log\|Q\|_{L^\infty(\mathbb{R}^N)}, \quad V_Q'(r)=-\frac{Q'(r)}{Q(r)}>0.
\end{equation*}
Since $Q(r)\to0$ as $r\to+\infty$, we also have $V_Q(r)\to+\infty$.
\end{proof}

\subsection{The fixed-mass characterization}
This subsection gives a fixed-mass characterization of ground states of \eqref{eq:log}, by reformulating the variational problem as a minimization problem under a fixed $L^2$-mass constraint.

Let
\begin{equation*}
c_0 = \inf_{u\in\mathcal N_0}\mathcal{I}_0(u).
\end{equation*}
If $Q$ is a ground state of \eqref{eq:log}, then $Q\in H^{2s}(\mathbb{R}^N)$ and $Q\log Q\in L^2(\mathbb{R}^N)$ by \eqref{eq:qint}. Multiplying $(-\Delta)^sQ=Q\log Q$ with $Q$ gives
\begin{equation*}
A(Q)=C_0(Q), \quad \mathcal{I}_0(Q)=\frac{1}{4}B(Q)=c_0.
\end{equation*}
Thus every ground state of \eqref{eq:log} has the same mass
\begin{equation}\label{eq:M}
M=4c_0.
\end{equation}
Define the full fixed-mass sphere
\begin{equation*}
\mathcal{S}_M = \left\{u\in H^s(\mathbb{R}^N):B(u)=M\right\}.
\end{equation*}
When $\mathcal{I}_0$ is considered on $\mathcal{S}_M$, we set
$
\mathcal{I}_0(u)=+\infty $ for $u\notin\mathcal{D}_s.
$

The common mass $M$ allows us to reformulate the ground state problem as a constrained minimization problem.

\begin{lemma}\label{lem:fmass}
For every $u\in\mathcal{S}_M\cap\mathcal{D}_s$, we have
\begin{equation}\label{eq:fmi}
A(u)-C_0(u)\geq0.
\end{equation}
Moreover, the minimizers of $\mathcal{I}_0$ on $\mathcal{S}_M$ are exactly the ground states of \eqref{eq:log}.
\end{lemma}

\begin{proof}
Fix $u\in\mathcal{S}_M\cap\mathcal{D}_s$. For $t>0$,
\begin{equation*}
A(tu)=t^2A(u),\quad C_0(tu)=t^2C_0(u)+t^2M\log t.
\end{equation*}
Hence $tu\in\mathcal{N}_0$ if and only if
\begin{equation*}
A(u)-C_0(u)-M\log t=0,
\end{equation*}
which has the unique solution
\begin{equation}\label{eq:t0}
t(u)=\exp\left(\frac{A(u)-C_0(u)}{M}\right).
\end{equation}
Since $t(u)u\in\mathcal{N}_0$,
\begin{equation*}
\mathcal{I}_0(t(u)u)=\frac{1}{4}t(u)^2M.
\end{equation*}
By the definition of $c_0$ and \eqref{eq:M},
\begin{equation*}
\frac{M}{4}=c_0\leq\frac{1}{4}t(u)^2M.
\end{equation*}
Thus $t(u)\geq1$ and \eqref{eq:t0} gives \eqref{eq:fmi}. 
For $u\in\mathcal{S}_M\cap\mathcal{D}_s$,
\begin{equation*}
\mathcal{I}_0(u)=\frac{M}{4}+\frac{1}{2}\left(A(u)-C_0(u)\right)\geq c_0.
\end{equation*}
For $u\in\mathcal{S}_M\setminus\mathcal{D}_s$, the same inequality follows from the convention $\mathcal{I}_0(u)=+\infty$. Since every ground state has mass $M$ and energy $c_0$, every ground state is a minimizer of $\mathcal{I}_0$ on $\mathcal{S}_M$.

Conversely, let $u$ be a minimizer of $\mathcal{I}_0$ on $\mathcal{S}_M$. Then $\mathcal{I}_0(u)=c_0<+\infty$, so $u\in\mathcal{D}_s$. The preceding identity gives
\begin{equation}\label{eq:min-nehari}
A(u)=C_0(u).
\end{equation}
Fix $\phi\in C_c^\infty(\mathbb{R}^N)$ and, for $|t|$ sufficiently small, define
\begin{equation*}
\alpha(t)=\frac{\sqrt{M}}{\|u+t\phi\|_{L^2(\mathbb{R}^N)}},\quad u_t=\alpha(t)(u+t\phi).
\end{equation*}
Then $B(u_t)=M$. A direct computation gives
\begin{equation*}
\alpha(0)=1,\quad \alpha'(0)=-\frac{\langle u,\phi\rangle_{L^2(\mathbb{R}^N)}}{M},
\end{equation*}
and hence
\begin{equation}\label{eq:first-path}
u_0=u,\quad u_0'=\phi-\frac{\langle u,\phi\rangle_{L^2(\mathbb{R}^N)}}{M}u.
\end{equation}

Set
\begin{equation*}
G(r)=\frac{1}{4}r^2-\frac{1}{2}r^2\log|r|\quad\text{for }r\neq0,\quad G(0)=0.
\end{equation*}
Then $G\in C^1(\mathbb{R})$ and
\begin{equation*}
G'(r)=-r\log|r|,\quad G'(0)=0.
\end{equation*}
Let $K=\operatorname{supp}\phi$ and choose $p\in(2,2_s^*)$. For $|t|$ sufficiently small,
\begin{equation*}
|u+t\phi|^2\left|\log|u+t\phi|\right|\leq C\left(1+|u|^p\right)\quad\text{on }K.
\end{equation*}
Since $u+t\phi=u$ on $\mathbb{R}^N\setminus K$, we have $u+t\phi\in\mathcal{D}_s$. Since $\alpha(t)$ remains bounded above and away from zero, it follows that $u_t\in\mathcal{D}_s$.

Moreover, $\alpha$ and $\alpha'$ are bounded for $|t|$ sufficiently small and
\begin{equation*}
|G'(u_t)u_t'|\leq C\left(1+|u|^p\right)\mathbf{1}_K+Cu^2\left(1+|\log|u||\right)\mathbf{1}_{\mathbb{R}^N\setminus K}.
\end{equation*}
The right-hand side belongs to $L^1(\mathbb{R}^N)$. Hence differentiation under the integral sign is justified.

Since
\begin{equation*}
\mathcal{I}_0(v)=\frac{1}{2}A(v)+\int_{\mathbb{R}^N}G(v)\mathrm{d}x,
\end{equation*}
we obtain from \eqref{eq:first-path}
\begin{equation*}
\left.\frac{\mathrm{d}}{\mathrm{d}t}\frac{1}{2}A(u_t)\right|_{t=0}
=\left\langle(-\Delta)^{\frac{s}{2}}u,(-\Delta)^{\frac{s}{2}}\phi\right\rangle_{L^2(\mathbb{R}^N)}-\frac{\langle u,\phi\rangle_{L^2(\mathbb{R}^N)}}{M}A(u),
\end{equation*}
and
\begin{equation*}
\left.\frac{\mathrm{d}}{\mathrm{d}t}\int_{\mathbb{R}^N}G(u_t)\mathrm{d}x\right|_{t=0}
=-\int_{\mathbb{R}^N}u\log|u|\phi\mathrm{d}x+\frac{\langle u,\phi\rangle_{L^2(\mathbb{R}^N)}}{M}C_0(u).
\end{equation*}
Therefore, by the minimality of $u$ on $\mathcal{S}_M$,
\begin{equation*}
\begin{aligned}
	0=\left.\frac{\mathrm{d}}{\mathrm{d}t}\mathcal{I}_0(u_t)\right|_{t=0}
	=&\left\langle(-\Delta)^{\frac{s}{2}}u,(-\Delta)^{\frac{s}{2}}\phi\right\rangle_{L^2(\mathbb{R}^N)} \\
	&-\int_{\mathbb{R}^N}u\log|u|\phi\mathrm{d}x
	-\frac{\langle u,\phi\rangle_{L^2(\mathbb{R}^N)}}{M}\left(A(u)-C_0(u)\right).
\end{aligned}
\end{equation*}
By \eqref{eq:min-nehari},
\begin{equation*}
\left\langle(-\Delta)^{\frac{s}{2}}u,(-\Delta)^{\frac{s}{2}}\phi\right\rangle_{L^2(\mathbb{R}^N)}
=\int_{\mathbb{R}^N}u\log|u|\phi\mathrm{d}x.
\end{equation*}
Thus $u$ is a weak solution of \eqref{eq:log}. Since $B(u)=M>0$, \eqref{eq:min-nehari} gives $u\in\mathcal{N}_0$. Together with $\mathcal{I}_0(u)=c_0$, this shows that $u$ is a ground state.
\end{proof}

Applying this constrained minimality to the mass-preserving dilation gives the corresponding scaling identity.

\begin{lemma}\label{lem:plog}
Every ground state of \eqref{eq:log} satisfies
\begin{equation*}
A(Q)=\frac{N}{4s}B(Q)=\frac{NM}{4s}.
\end{equation*}
\end{lemma}

\begin{proof}
For $\lambda>0$, set
\begin{equation*}
Q_\lambda(x)=\lambda^{\frac{N}{2}}Q(\lambda x).
\end{equation*}
Then
\begin{equation*}
B(Q_\lambda)=\int_{\mathbb{R}^N}\lambda^NQ(\lambda x)^2 \mathrm{d}x=\int_{\mathbb{R}^N}Q(y)^2 \mathrm{d}y=M,\quad A(Q_\lambda)=\lambda^{2s}A(Q),
\end{equation*}
so $Q_\lambda\in\mathcal{S}_M$. Moreover, with $y=\lambda x$,
\begin{equation*}
\begin{aligned}
	C_0(Q_\lambda)
	&=\int_{\mathbb{R}^N}\lambda^NQ(\lambda x)^2\log\left(\lambda^{\frac{N}{2}}Q(\lambda x)\right) \mathrm{d}x\\
	&=\int_{\mathbb{R}^N}Q(y)^2\left(\frac{N}{2}\log\lambda+\log Q(y)\right)\mathrm{d}y
	=C_0(Q)+\frac{N}{2}M\log\lambda.
\end{aligned}
\end{equation*}
Hence
\begin{equation*}
\mathcal{I}_0(Q_\lambda)=\frac{1}{2}\lambda^{2s}A(Q)+\frac{1}{4}M-\frac{1}{2}C_0(Q)-\frac{N}{4}M\log\lambda.
\end{equation*}
By Lemma \ref{lem:fmass}, $Q$ is a global minimizer of $\mathcal{I}_0$ on $\mathcal{S}_M$. Therefore, $\lambda=1$ minimizes $\mathcal{I}_0(Q_\lambda)$ and
\begin{equation*}
0=\left.\frac{\mathrm{d}}{\mathrm{d}\lambda}\mathcal{I}_0(Q_\lambda)\right|_{\lambda=1}=sA(Q)-\frac{N}{4}M.
\end{equation*}
Thus $A(Q)=\frac{NM}{4s}=\frac{N}{4s}B(Q)$.
\end{proof}

\subsection{The linearized operator and radial spectral theory}

The unbounded potential $V_Q$ requires a quadratic-form realization of the linearized operator $L_Q$. This subsection introduces the corresponding form domain and develops the radial spectral theory of $L_Q$. In particular, we prove that the second radial eigenfunction changes sign exactly once, which will be used in Lemma \ref{lem:rker} to exclude nontrivial radial zero modes.

Define
\begin{equation*}
\mathcal{X}_Q = \left\{ h\in H^s(\mathbb{R}^N): \int_{\mathbb{R}^N}(V_Q)_+h^2 \mathrm{d}x<\infty \right\},
\end{equation*}
equipped with the inner product
\begin{equation*}
\langle h,g\rangle_Q = \langle h,g\rangle_{H^s(\mathbb{R}^N)} + \int_{\mathbb{R}^N}(V_Q)_+hg \mathrm{d}x
\end{equation*}
and the induced norm
\begin{equation*}
\|h\|_Q^2 = \|h\|_{H^s(\mathbb{R}^N)}^2 + \int_{\mathbb{R}^N}(V_Q)_+h^2 \mathrm{d}x.
\end{equation*}
Then $\mathcal{X}_Q$ is a Hilbert space. Indeed, let $\{h_n\}\subset\mathcal{X}_Q$ be Cauchy with respect to $\|\cdot\|_Q$. Since $\{h_n\}$ is Cauchy in $H^s(\mathbb{R}^N)$, there exists $h\in H^s(\mathbb{R}^N)$ such that $h_n\to h$ in $H^s(\mathbb{R}^N)$. Passing to a subsequence, we may also assume that $h_n\to h$ almost everywhere in $\mathbb{R}^N$. For every fixed $n$, Fatou's lemma gives
\begin{equation*}
\int_{\mathbb{R}^N}(V_Q)_+|h_n-h|^2 \mathrm{d}x
\leq \liminf_{m\to\infty}\int_{\mathbb{R}^N}(V_Q)_+|h_n-h_m|^2 \mathrm{d}x.
\end{equation*}
Since $\{h_n\}$ is Cauchy with respect to $\|\cdot\|_Q$,
\begin{equation*}
\int_{\mathbb{R}^N}(V_Q)_+|h_n-h|^2 \mathrm{d}x\to0.
\end{equation*}
In particular, $h\in\mathcal{X}_Q$, and together with $h_n\to h$ in $H^s(\mathbb{R}^N)$ this yields $\|h_n-h\|_Q\to0$. Hence $\mathcal{X}_Q$ is complete.

\begin{lemma}\label{lem:density}
If $h\in\mathcal{X}_Q$, then there exists a sequence $\{h_n\}\subset C_c^\infty(\mathbb{R}^N)$ such that $\|h_n-h\|_Q\to0$.
\end{lemma}

\begin{proof}
Choose $\eta\in C_c^\infty(\mathbb{R}^N)$ such that $0\leq\eta\leq1$, $\eta=1$ on $B_1$ and $\eta=0$ on $\mathbb{R}^N\setminus B_2$. Set $\eta_R(x)=\eta\left(\frac{x}{R}\right)$ and
$h_R=\eta_Rh$. Then
\begin{equation}\label{eq:denw}
\|h_R-h\|_{L^2(\mathbb{R}^N)}^2 + \int_{\mathbb{R}^N}(V_Q)_+|h_R-h|^2\mathrm{d}x \to0
\end{equation}
by dominated convergence theorem. Put $\zeta_R=1-\eta_R$. Since
\begin{equation*}
|\zeta_R(x)h(x)-\zeta_R(y)h(y)|^2
\leq 2\zeta_R(x)^2|h(x)-h(y)|^2+2h(y)^2|\zeta_R(x)-\zeta_R(y)|^2,
\end{equation*}
we obtain
\begin{equation*}
\begin{aligned}
	A(h-h_R)=A(\zeta_Rh)
	&\leq C\iint_{\mathbb{R}^{2N}}\frac{\zeta_R(x)^2|h(x)-h(y)|^2}{|x-y|^{N+2s}} \mathrm{d}x\mathrm{d}y\\
	&\quad+C\int_{\mathbb{R}^N}h(y)^2\left(\int_{\mathbb{R}^N}\frac{|\zeta_R(x)-\zeta_R(y)|^2}{|x-y|^{N+2s}} \mathrm{d}x\right)\mathrm{d}y.
\end{aligned}
\end{equation*}
Since $|\zeta_R(x)-\zeta_R(y)|=|\eta_R(x)-\eta_R(y)|$, the estimate \eqref{eq:symmetry} gives 
\begin{equation*}
A(\zeta_Rh)\leq C\iint_{\mathbb{R}^{2N}}\frac{\zeta_R(x)^2|h(x)-h(y)|^2}{|x-y|^{N+2s}} \mathrm{d}x\mathrm{d}y+CR^{-2s}\|h\|_{L^2(\mathbb{R}^N)}^2.
\end{equation*}
For every $x\in\mathbb{R}^N$, $\zeta_R(x)\to0$ as $R\to\infty$ and
\begin{equation*}
0\leq\frac{\zeta_R(x)^2|h(x)-h(y)|^2}{|x-y|^{N+2s}}\leq\frac{|h(x)-h(y)|^2}{|x-y|^{N+2s}}.
\end{equation*}
Since $h\in H^s(\mathbb{R}^N)$, the dominated convergence theorem yields $A(\zeta_Rh)\to0$. Together with \eqref{eq:denw}, this gives
\begin{equation}\label{eq:denhs}
h_R\to h \quad\text{in }H^s(\mathbb{R}^N).
\end{equation}

For fixed $R$, $(V_Q)_+$ is bounded on $B_{2R+1}$. Choose $\rho\in C_c^\infty(B_1)$ with $\rho\geq0$ and $\int_{\mathbb{R}^N}\rho\mathrm{d}x=1$, set
\begin{equation*}
\rho_\delta(x)=\delta^{-N}\rho\left(\frac{x}{\delta}\right), \quad h_{R,\delta}=\rho_\delta*h_R,
\end{equation*}
and let $0<\delta<1$. Then
$
h_{R,\delta}\in C_c^\infty(B_{2R+1}),
$
and
\begin{equation*}
\|h_{R,\delta}-h_R\|_{H^s(\mathbb{R}^N)} + \left( \int_{\mathbb{R}^N}(V_Q)_+|h_{R,\delta}-h_R|^2\mathrm{d}x \right)^{\frac{1}{2}} \to0
\end{equation*}
as $\delta\to0$. Combining \eqref{eq:denw} and \eqref{eq:denhs} and using a diagonal argument, we obtain a sequence in $C_c^\infty(\mathbb{R}^N)$ converging to $h$ in $\mathcal X_Q$.
\end{proof}

We next use the form domain $\mathcal X_Q$ to define the self-adjoint realization of the linearized operator. Define the symmetric bilinear form
\begin{equation*}
\mathfrak{q}_Q(h,k)=\left\langle(-\Delta)^{\frac{s}{2}}h,(-\Delta)^{\frac{s}{2}}k\right\rangle_{L^2(\mathbb{R}^N)}+\int_{\mathbb{R}^N}V_Qhk\mathrm{d}x,\quad h,k\in\mathcal X_Q.
\end{equation*}
Since $(V_Q)_-\in L^\infty(\mathbb{R}^N)$, the form $\mathfrak q_Q$ is well defined on $\mathcal X_Q$. Indeed,
\begin{equation*}
\int_{\mathbb{R}^N}(V_Q)_+|hk| \mathrm{d}x \leq \left(\int_{\mathbb{R}^N}(V_Q)_+h^2 \mathrm{d}x\right)^{\frac{1}{2}}\left(\int_{\mathbb{R}^N}(V_Q)_+k^2 \mathrm{d}x\right)^{\frac{1}{2}},
\end{equation*}
and $(V_Q)_-\in L^\infty(\mathbb{R}^N)$. Moreover,
\begin{equation}\label{eq:q-lower}
\mathfrak{q}_Q(h,h)\geq-\|(V_Q)_-\|_{L^\infty(\mathbb{R}^N)}\|h\|_{L^2(\mathbb{R}^N)}^2
\end{equation}
for every $h\in\mathcal{X}_Q$. 
Choose $C_Q>\|(V_Q)_-\|_{L^\infty(\mathbb{R}^N)}+1$. Then
\begin{equation*}
\begin{aligned}
\mathfrak{q}_Q(h,h)+C_Q\|h\|_{L^2(\mathbb{R}^N)}^2
&=\|(-\Delta)^{\frac{s}{2}}h\|_{L^2(\mathbb{R}^N)}^2+\int_{\mathbb{R}^N}(V_Q)_+h^2 \mathrm{d}x+\int_{\mathbb{R}^N}(C_Q-(V_Q)_-)h^2 \mathrm{d}x.
\end{aligned}
\end{equation*}
Hence $\mathfrak{q}_Q(h,h)+C_Q\|h\|_{L^2(\mathbb{R}^N)}^2$ is equivalent to $\|h\|_Q^2$. Since $\mathcal{X}_Q$ is complete with respect to $\|\cdot\|_Q$, the form $\mathfrak{q}_Q$ is closed and bounded from below. Since $V_Q$ is locally bounded, $C_c^\infty(\mathbb{R}^N)\subset\mathcal{X}_Q$, so $\mathfrak{q}_Q$ is densely defined in $L^2(\mathbb{R}^N)$.

\begin{lemma}\label{lem:LQ-domain}
The form $\mathfrak q_Q$ determines a unique lower-bounded self-adjoint realization of $L_Q$ on $L^2(\mathbb{R}^N)$, with
\begin{equation*}
\mathcal D(L_Q)=\left\{h\in\mathcal X_Q:(-\Delta)^sh+V_Qh\in L^2(\mathbb{R}^N)\text{ in }\mathcal D'(\mathbb{R}^N)\right\},
\end{equation*}
and $L_Qh=(-\Delta)^sh+V_Qh$ for $h\in\mathcal D(L_Q)$.
\end{lemma}

\begin{proof}
Since $\mathfrak{q}_Q$ is densely defined, closed, symmetric and bounded from below, \cite[Chapter VI, Theorems 2.1 and 2.6]{Kato-1995} (see also \cite[Theorem VIII.15]{Reed-Simon-1978-1}) yields a unique self-adjoint operator $L_Q$ associated with $\mathfrak{q}_Q$. By \cite[Chapter VI, Theorems 2.1-(i) and -(iii)]{Kato-1995}, a function $\tilde{h}\in\mathcal{X}_Q$ belongs to $\mathcal{D}(L_Q)$ if and only if there exists $\tilde{f}\in L^2(\mathbb{R}^N)$ such that
\begin{equation}\label{eq:LQ-form}
\mathfrak{q}_Q(\tilde{h},\phi)=\langle \tilde{f},\phi\rangle_{L^2(\mathbb{R}^N)}
\end{equation}
for every $\phi\in\mathcal{X}_Q$. In this case, $L_Q\tilde{h}=\tilde{f}$. Let $h\in\mathcal{D}(L_Q)$ and set $f=L_Qh$. Since $C_c^\infty(\mathbb{R}^N)\subset\mathcal{X}_Q$, \eqref{eq:LQ-form} gives
\begin{equation*}
\left\langle(-\Delta)^{\frac{s}{2}}h,(-\Delta)^{\frac{s}{2}}\phi\right\rangle_{L^2(\mathbb{R}^N)}+\int_{\mathbb{R}^N}V_Qh\phi \mathrm{d}x=\langle f,\phi\rangle_{L^2(\mathbb{R}^N)}
\end{equation*}
for every $\phi\in C_c^\infty(\mathbb{R}^N)$. Hence $(-\Delta)^sh+V_Qh=f$ in $\mathcal{D}'(\mathbb{R}^N)$, and therefore $(-\Delta)^sh+V_Qh\in L^2(\mathbb{R}^N)$.

Conversely, let $h\in\mathcal{X}_Q$ and suppose that $(-\Delta)^sh+V_Qh=F$ in $\mathcal{D}'(\mathbb{R}^N)$ for some $F\in L^2(\mathbb{R}^N)$. Then
\begin{equation*}
\mathfrak{q}_Q(h,\phi)=\langle F,\phi\rangle_{L^2(\mathbb{R}^N)}
\end{equation*}
for every $\phi\in C_c^\infty(\mathbb{R}^N)$. By Lemma \ref{lem:density}, for every $\phi\in\mathcal{X}_Q$ there exists $\phi_n\in C_c^\infty(\mathbb{R}^N)$ such that $\|\phi_n-\phi\|_Q\to0$. Since
\begin{equation*}
|\mathfrak{q}_Q(h,\phi_n-\phi)|\leq C_h\|\phi_n-\phi\|_Q
\end{equation*}
and
\begin{equation*}
|\langle F,\phi_n-\phi\rangle_{L^2(\mathbb{R}^N)}|\leq \|F\|_{L^2(\mathbb{R}^N)}\|\phi_n-\phi\|_Q,
\end{equation*}
letting $n\to\infty$ gives
\begin{equation*}
\mathfrak{q}_Q(h,\phi)=\langle F,\phi\rangle_{L^2(\mathbb{R}^N)}
\end{equation*}
for every $\phi\in\mathcal{X}_Q$. Hence $h\in\mathcal{D}(L_Q)$ and $L_Qh=F$.
\end{proof}

The condition $V_Q(x)\to+\infty$ as $|x|\to+\infty$ also implies compactness of the resolvent.

\begin{corollary}\label{cor:compact-resolvent}
The operator $L_Q$ has compact resolvent. Consequently, its spectrum consists of isolated eigenvalues of finite multiplicity accumulating only at $+\infty$.
\end{corollary}

\begin{proof}
Recall that $C_Q>\|(V_Q)_-\|_{L^\infty(\mathbb{R}^N)}+1$. Then
\begin{equation*}
V_Q+C_Q\geq1\quad\text{in }\mathbb{R}^N.
\end{equation*}
Let $\{u_n\}\subset\mathcal X_Q$ satisfy
\begin{equation*}
\sup_n\left\{A(u_n)+\int_{\mathbb{R}^N}(V_Q+C_Q)u_n^2\mathrm{d}x\right\}<\infty.
\end{equation*}
For every fixed $R>0$, the sequence $\{u_n\}$ is bounded in $H^s(B_R)$ and hence is precompact in $L^2(B_R)$. Moreover,
\begin{equation*}
\int_{|x|>R}u_n^2\mathrm{d}x \leq \frac{1}{\inf_{r>R}(V_Q(r)+C_Q)}\int_{\mathbb{R}^N}(V_Q+C_Q)u_n^2\mathrm{d}x.
\end{equation*}
Since $V_Q(r)\to+\infty$ as $r\to+\infty$, the right-hand side tends to zero uniformly in $n$ as $R\to+\infty$. Thus $\mathcal X_Q\hookrightarrow L^2(\mathbb{R}^N)$ is compact.

Choose $\gamma>C_Q+1$. Let $\{f_n\}$ be bounded in $L^2(\mathbb{R}^N)$ and set
\begin{equation*}
u_n=(L_Q+\gamma)^{-1}f_n.
\end{equation*}
Testing $(L_Q+\gamma)u_n=f_n$ by $u_n$ gives
\begin{equation*}
A(u_n)+\int_{\mathbb{R}^N}(V_Q+C_Q)u_n^2\mathrm{d}x+(\gamma-C_Q)\|u_n\|_{L^2(\mathbb{R}^N)}^2
=\langle f_n,u_n\rangle_{L^2(\mathbb{R}^N)}.
\end{equation*}
Since $\gamma-C_Q>1$, the Cauchy-Schwarz and Young inequalities yield
\begin{equation*}
A(u_n)+\int_{\mathbb{R}^N}(V_Q+C_Q)u_n^2\mathrm{d}x+\|u_n\|_{L^2(\mathbb{R}^N)}^2\leq C.
\end{equation*}
Hence $\{u_n\}$ is bounded in $\mathcal X_Q$ and therefore precompact in $L^2(\mathbb{R}^N)$. Thus $(L_Q+\gamma)^{-1}$ is compact and $L_Q$ has compact resolvent.
\end{proof}

By Lemma \ref{lem:LQ-domain},
\begin{equation}\label{eq:weak-kernel}
h\in\ker L_Q \iff \mathfrak{q}_Q(h,\phi)=0\quad\text{for every }\phi\in\mathcal{X}_Q.
\end{equation} 
With the self-adjoint realization of $L_Q$ fixed, we can now place the infinitesimal translation modes in its operator kernel.

\begin{lemma}\label{lem:itrans}
For $j=1,\cdots,N$, we have $L_Q\left(\partial_{x_j}Q\right)=0$.
\end{lemma}

\begin{proof}
Set $w=\partial_{x_j}Q$. By Proposition \ref{prop:known},
$
w\in L^2(\mathbb{R}^N)\cap H^s_{\mathrm{loc}}(\mathbb{R}^N).
$
Differentiating
\begin{equation*}
(-\Delta)^sQ=Q\log Q
\end{equation*}
in the distributional sense gives
\begin{equation}\label{eq:trdist}
(-\Delta)^sw=(1+\log Q)w=-V_Qw.
\end{equation}

Choose $\eta\in C_c^\infty(\mathbb{R}^N)$ such that
\begin{equation*}
0\leq\eta\leq1,\quad \eta=1\ \text{on }B_1,\quad \eta=0\ \text{on }\mathbb{R}^N\setminus B_2,
\end{equation*}
and set $\eta_R(x)=\eta(\frac{x}{R})$. Since $w\in H^s_{\mathrm{loc}}(\mathbb{R}^N)$, we see that
$
\eta_R^2w\in H^s(\mathbb{R}^N)
$
and has compact support. Moreover, since $V_Q$ is bounded on $B_{2R}$ and $w\in L^2_{\mathrm{loc}}(\mathbb{R}^N)$, we have $V_Qw\in L^2(B_{2R})$. Hence Lemma \ref{lem:cutoff}, applied to \eqref{eq:trdist}, gives
\begin{equation*}
A(\eta_Rw)+\int_{\mathbb{R}^N}V_Q\eta_R^2w^2\mathrm{d}x=\mathcal E_R(w).
\end{equation*}
Therefore,
\begin{equation*}
\begin{aligned}
	A(\eta_Rw)+\int_{\mathbb{R}^N}(V_Q)_+\eta_R^2w^2\mathrm{d}x
	&=\mathcal E_R(w)+\int_{\mathbb{R}^N}(V_Q)_-\eta_R^2w^2\mathrm{d}x\\
	&\leq\left(CR^{-2s}+\|(V_Q)_-\|_{L^\infty(\mathbb{R}^N)}\right)\|w\|_{L^2(\mathbb{R}^N)}^2.
\end{aligned}
\end{equation*}
Hence, for $R\geq1$, $\{\eta_Rw\}$ is bounded in $H^s(\mathbb{R}^N)$ and
\begin{equation*}
\sup_{R\geq1}\int_{\mathbb{R}^N}(V_Q)_+\eta_R^2w^2 \mathrm{d}x<\infty.
\end{equation*}
Since $\eta_Rw\to w$ in $L^2(\mathbb{R}^N)$, weak lower semicontinuity and Fatou's lemma give
\begin{equation*}
w\in H^s(\mathbb{R}^N),\quad \int_{\mathbb{R}^N}(V_Q)_+w^2 \mathrm{d}x<\infty.
\end{equation*}
Thus $w\in\mathcal{X}_Q$. 
Lemma \ref{lem:LQ-domain} and \eqref{eq:trdist} therefore give
$w\in\mathcal{D}(L_Q)$ and $L_Qw=0$.
\end{proof}

Fix a radial ground state of \eqref{eq:log}. By Proposition \ref{prop:known},
\begin{equation*}
V_Q(r)=-1-\log Q(r)
\end{equation*}
is radial and strictly increasing, is bounded from below, and tends to $+\infty$ as $r\to\infty$.

Define
\begin{equation*}
L^2_{\mathrm{rad}}(\mathbb{R}^N)=\left\{u\in L^2(\mathbb{R}^N):u(Rx)=u(x)\text{ for every }R\in\mathbb{O}(N)\right\},
\end{equation*}
where $\mathbb{O}(N)$ is the orthogonal group of $\mathbb{R}^N$. For a radial function $u$, we write $u(r)$ for its radial profile, where $r=|x|$. Set
\begin{equation*}
H^s_{\mathrm{rad}}(\mathbb{R}^N)=H^s(\mathbb{R}^N)\cap L^2_{\mathrm{rad}}(\mathbb{R}^N),\quad
\mathcal{X}_Q^{\mathrm{rad}}=\mathcal{X}_Q\cap L^2_{\mathrm{rad}}(\mathbb{R}^N).
\end{equation*}
Since $V_Q$ is radial, $L^2_{\mathrm{rad}}(\mathbb{R}^N)$ is a reducing subspace for $L_Q$. By Corollary \ref{cor:compact-resolvent}, the restriction of $L_Q$ to $L^2_{\mathrm{rad}}(\mathbb{R}^N)$ also has compact resolvent. 
For $0\neq u\in\mathcal{X}_Q^{\mathrm{rad}}$, define
\begin{equation*}
\mathcal{R}_Q(u)=\frac{\mathfrak{q}_Q(u,u)}{\|u\|_{L^2(\mathbb{R}^N)}^2}.
\end{equation*}
For $j=1,2$, set
\begin{equation*}
\lambda_j^{\mathrm{rad}}
=\inf_{\substack{E\subset\mathcal{X}_Q^{\mathrm{rad}}\\ \dim E=j}}\sup_{0\neq u\in E}\mathcal{R}_Q(u).
\end{equation*}
By \eqref{eq:q-lower} and the min-max principle, $\lambda_1^{\mathrm{rad}}$ and $\lambda_2^{\mathrm{rad}}$ are finite eigenvalues of the radial restriction of $L_Q$, counted with multiplicity.

The key radial input is the following nodal property.

\begin{proposition}\label{prop:prs}
There exists a normalized radial eigenfunction $\psi\in\mathcal D(L_Q)$ associated with $\lambda_2^{\mathrm{rad}}$. Moreover, after changing the sign of $\psi$ if necessary, there exists $r_*>0$ such that
\begin{equation*}
\psi\geq0\quad\text{almost everywhere in }B_{r_*},\quad
\psi\leq0\quad\text{almost everywhere in }\mathbb{R}^N\setminus B_{r_*},
\end{equation*}
and $\psi$ does not vanish almost everywhere on either region.
\end{proposition} 

Note that $V_Q$ is unbounded, so the radial oscillation theorem in \cite[Theorem 2.3]{Frank-Lenzmann-Silvestre-2016} cannot be applied directly to prove Proposition \ref{prop:prs}. To overcome this difficulty, we first approximate $V_Q$ by bounded truncations and analyze the first two radial eigenvalues and the corresponding eigenfunctions of the truncated operators.

For every $m>V_Q(0)=\inf_{\mathbb{R}^N}V_Q$, the strict monotonicity of $V_Q$ and the fact that $V_Q(r)\to+\infty$ as $r\to+\infty$ yield a unique $R_m>0$ such that $V_Q(R_m)=m$. Set
\begin{equation*}
V_m(r)=\min\{V_Q(r),m\}.
\end{equation*}
Since $V_m\in L^\infty(\mathbb{R}^N)$, the corresponding symmetric bilinear form
\begin{equation*}
\mathfrak{q}_m(u,v)=\left\langle(-\Delta)^{\frac{s}{2}}u,(-\Delta)^{\frac{s}{2}}v\right\rangle_{L^2(\mathbb{R}^N)}+\int_{\mathbb{R}^N}V_muv\mathrm{d}x
\end{equation*}
with domain $H^s(\mathbb{R}^N)$ is closed and bounded from below. Let $L_m$ be the associated self-adjoint operator. For $0\neq u\in H^s_{\mathrm{rad}}(\mathbb{R}^N)$, set
\begin{equation*}
\mathcal{R}_m(u)=\frac{A(u)+\int_{\mathbb{R}^N}V_mu^2\mathrm{d}x}{\|u\|_{L^2(\mathbb{R}^N)}^2},
\end{equation*}
and define (see \cite[Theorem XIII.1]{Reed-Simon-1978-4})
\begin{equation*}
\lambda_{j,m}^{\mathrm{rad}}=\inf_{\substack{E\subset H^s_{\mathrm{rad}}(\mathbb{R}^N)\\ \dim E=j}}\sup_{0\neq u\in E}\mathcal{R}_m(u),
\quad j=1,2.
\end{equation*}

\begin{lemma}\label{lem:prs-truncated}
For all sufficiently large integers $m$, $\lambda_{1,m}^{\mathrm{rad}}$ and $\lambda_{2,m}^{\mathrm{rad}}$ are simple radial eigenvalues of $L_m$ lying below $m$. Moreover, a normalized eigenfunction associated with $\lambda_{2,m}^{\mathrm{rad}}$ changes sign exactly once on $(0,+\infty)$.
\end{lemma}

\begin{proof}
Since $V_m-m=0$ on $\mathbb{R}^N\setminus B_{R_m}$, multiplication by $V_m-m$ is compact from $H^s(\mathbb{R}^N)$ to $L^2(\mathbb{R}^N)$. Indeed, if $z_n\rightharpoonup0$ in $H^s(\mathbb{R}^N)$, then
\begin{equation*}
\|(V_m-m)z_n\|_{L^2(\mathbb{R}^N)} \leq \|V_m-m\|_{L^\infty(\mathbb{R}^N)}\|z_n\|_{L^2(B_{R_m})}\to0
\end{equation*}
by the compact embedding $H^s(B_{R_m})\hookrightarrow L^2(B_{R_m})$. Moreover, Plancherel's identity gives
\begin{equation*}
\left\|((-\Delta)^s+m+1)^{-1}f\right\|_{H^{2s}(\mathbb{R}^N)}^2
=\int_{\mathbb{R}^N}\frac{(1+|\xi|^2)^{2s}}{(|\xi|^{2s}+m+1)^2}|\widehat f(\xi)|^2\mathrm{d}\xi
\leq C_m\|f\|_{L^2(\mathbb{R}^N)}^2.
\end{equation*}
Since $H^{2s}(\mathbb{R}^N)\hookrightarrow H^s(\mathbb{R}^N)$, it follows that
\begin{equation*}
(V_m-m)\left((-\Delta)^s+m+1\right)^{-1}:L^2(\mathbb{R}^N)\to L^2(\mathbb{R}^N)
\end{equation*}
is compact. Since $-1$ belongs to the resolvent set of $(-\Delta)^s+m$, the potential $V_m-m$ is relatively compact with respect to $(-\Delta)^s+m$. Writing
\begin{equation*}
L_m=((-\Delta)^s+m)+(V_m-m),
\end{equation*}
Weyl's theorem \cite[Theorem XIII.14]{Reed-Simon-1978-4} gives
\begin{equation*}
\sigma_{\mathrm{ess}}(L_m)=\sigma_{\mathrm{ess}}\left((-\Delta)^s+m\right)=[m,+\infty).
\end{equation*}

Since $V_m\leq V_{m+1}\leq V_Q$ and $\mathcal{X}_Q^{\mathrm{rad}}\subset H^s_{\mathrm{rad}}(\mathbb{R}^N)$,
\begin{equation}\label{eq:mmmono}
\lambda_{j,m}^{\mathrm{rad}}\leq\lambda_{j,m+1}^{\mathrm{rad}}\leq\lambda_j^{\mathrm{rad}}, \quad j=1,2.
\end{equation}
For all sufficiently large integers $m$,
\begin{equation*}
\lambda_{1,m}^{\mathrm{rad}}\leq\lambda_{2,m}^{\mathrm{rad}}\leq\lambda_2^{\mathrm{rad}}<m.
\end{equation*}
Since $V_m$ is radial, $L_m$ preserves radial functions, and its restriction to $L^2_{\mathrm{rad}}(\mathbb{R}^N)$ is self-adjoint. Applying the same compact-perturbation argument as above to this radial restriction gives
\begin{equation*}
\sigma_{\mathrm{ess}}\left(L_m\big|_{L^2_{\mathrm{rad}}(\mathbb{R}^N)}\right)=[m,+\infty).
\end{equation*}
Thus
\begin{equation*}
\lambda_{j,m}^{\mathrm{rad}}<m=\inf\sigma_{\mathrm{ess}}\left(L_m\big|_{L^2_{\mathrm{rad}}(\mathbb{R}^N)}\right),\quad j=1,2.
\end{equation*}
By the min-max principle, $\lambda_{1,m}^{\mathrm{rad}}$ and $\lambda_{2,m}^{\mathrm{rad}}$ are discrete eigenvalues of the radial restriction of $L_m$, and hence radial eigenvalues of $L_m$.

Since $V_Q$ is smooth, $V_Q'$ is bounded on $[0,R_m]$. Since $V_Q(R_m)=m$ and
\begin{equation*}
W_m(r)=
\begin{cases}
	V_Q(r)-m, & 0\leq r<R_m,\\
	0, & r\geq R_m,
\end{cases}
\end{equation*}
the function $W_m$ is radial, nondecreasing, bounded and globally Lipschitz. Hence
$W_m\in C^{0,\gamma}(\mathbb{R}^N)$ for every $0<\gamma<1$. Choosing
\begin{equation*}
\max\{0,1-2s\}<\gamma<1,
\end{equation*}
we see that $W_m$ satisfies (V1)--(V2) of \cite{Frank-Lenzmann-Silvestre-2016}. In particular, \cite[Corollary 1]{Frank-Lenzmann-Silvestre-2016} gives
\begin{equation*}
\lambda_{1,m}^{\mathrm{rad}}<\lambda_{2,m}^{\mathrm{rad}}<m.
\end{equation*}
Applying \cite[Theorem 2]{Frank-Lenzmann-Silvestre-2016} to
\begin{equation*}
L_m-m=(-\Delta)^s+W_m
\end{equation*}
shows that a normalized radial eigenfunction $\psi_m$ associated with $\lambda_{2,m}^{\mathrm{rad}}$ changes sign exactly once on $(0,+\infty)$.
\end{proof} 

We now show that, as $m\to\infty$, the first two radial min-max levels $\lambda_{1,m}^{\mathrm{rad}}$ and $\lambda_{2,m}^{\mathrm{rad}}$ of $L_m$ converge to the corresponding radial min-max levels $\lambda_1^{\mathrm{rad}}$ and $\lambda_2^{\mathrm{rad}}$ of $L_Q$.

\begin{lemma}\label{lem:prs-econv}
We have
\begin{equation}\label{eq:econv}
\lambda_{j,m}^{\mathrm{rad}}\to\lambda_j^{\mathrm{rad}}
\quad\text{for }j=1,2.
\end{equation}
Moreover, each $\lambda_j^{\mathrm{rad}}$ is attained by a normalized radial function $\phi_j\in\mathcal{D}(L_Q)$ satisfying
\begin{equation}\label{eq:rad-eig}
L_Q\phi_j=\lambda_j^{\mathrm{rad}}\phi_j.
\end{equation}
\end{lemma}

\begin{proof}
For all sufficiently large $m$, choose radial eigenfunctions $\phi_{1,m}$ and $\phi_{2,m}$ satisfying
\begin{equation*}
L_m\phi_{j,m}=\lambda_{j,m}^{\mathrm{rad}}\phi_{j,m},\quad
\langle\phi_{i,m},\phi_{j,m}\rangle_{L^2(\mathbb{R}^N)}=\delta_{ij}.
\end{equation*}
Choose $c>0$ such that $V_Q+c\geq1$ in $\mathbb{R}^N$. Fix a positive integer $k>V_Q(0)=\inf_{\mathbb{R}^N}V_Q$ and let $m\geq k$. Since $V_k\leq V_m$, \eqref{eq:mmmono} gives
\begin{equation*}
\begin{aligned}
	A(\phi_{j,m})+\int_{\mathbb{R}^N}(V_k+c)\phi_{j,m}^2\mathrm{d}x
	&\leq A(\phi_{j,m})+\int_{\mathbb{R}^N}(V_m+c)\phi_{j,m}^2\mathrm{d}x\\
	&=\lambda_{j,m}^{\mathrm{rad}}+c
	\leq\lambda_j^{\mathrm{rad}}+c.
\end{aligned}
\end{equation*}
Fixing $k$, we get that $\{\phi_{j,m}\}$ is bounded in $H^s(\mathbb{R}^N)$. Moreover, since $V_k=k$ on $\mathbb{R}^N\setminus B_{R_k}$,
\begin{equation*}
\int_{|x|\geq R_k}\phi_{j,m}^2\mathrm{d}x
\leq\frac{\lambda_j^{\mathrm{rad}}+c}{k+c}
\quad\text{for }m\geq k.
\end{equation*}
As $k\to\infty$, the right-hand side tends to zero. Hence $\{\phi_{j,m}\}$ is tight in $L^2(\mathbb{R}^N)$. Together with the compact embedding $H^s(B_R)\hookrightarrow L^2(B_R)$, after passing to a subsequence,
\begin{equation}\label{eq:philimit}
\phi_{j,m}\rightharpoonup\phi_j\quad\text{weakly in }H^s(\mathbb{R}^N),\quad
\phi_{j,m}\to\phi_j\quad\text{strongly in }L^2(\mathbb{R}^N),
\end{equation}
for $j=1,2$. Since $L^2_{\mathrm{rad}}(\mathbb{R}^N)$ is closed in $L^2(\mathbb{R}^N)$, each $\phi_j$ is radial. Moreover,
\begin{equation}\label{eq:orthlim}
\langle\phi_i,\phi_j\rangle_{L^2(\mathbb{R}^N)}=\delta_{ij}.
\end{equation}

By \eqref{eq:mmmono}, the limits
\begin{equation*}
\lambda_{j,\infty}^{\mathrm{rad}}:=\lim_{m\to\infty}\lambda_{j,m}^{\mathrm{rad}}\leq\lambda_j^{\mathrm{rad}}
\end{equation*}
exist. Let $\eta\in C_c^\infty(\mathbb{R}^N)$. Since $V_Q$ is bounded on $\operatorname{supp}\eta$, for all sufficiently large $m$, we have $V_m=V_Q$ on $\operatorname{supp}\eta$. Passing to the limit in
\begin{equation*}
\left\langle(-\Delta)^{\frac{s}{2}}\phi_{j,m},(-\Delta)^{\frac{s}{2}}\eta\right\rangle_{L^2(\mathbb{R}^N)}+\int_{\mathbb{R}^N}V_m\phi_{j,m}\eta\mathrm{d}x
=\lambda_{j,m}^{\mathrm{rad}}\langle\phi_{j,m},\eta\rangle_{L^2(\mathbb{R}^N)}
\end{equation*}
gives
\begin{equation}\label{eq:limweak}
\left\langle(-\Delta)^{\frac{s}{2}}\phi_j,(-\Delta)^{\frac{s}{2}}\eta\right\rangle_{L^2(\mathbb{R}^N)}+\int_{\mathbb{R}^N}V_Q\phi_j\eta\mathrm{d}x
=\lambda_{j,\infty}^{\mathrm{rad}}\langle\phi_j,\eta\rangle_{L^2(\mathbb{R}^N)}.
\end{equation}

We next prove that $\phi_j\in\mathcal{X}_Q^{\mathrm{rad}}$. For every fixed $k$, weak lower semicontinuity of $A$ and the strong $L^2$ convergence in \eqref{eq:philimit} give
\begin{equation*}
\begin{aligned}
	A(\phi_j)+\int_{\mathbb{R}^N}(V_k+c)\phi_j^2\mathrm{d}x
	&\leq\liminf_{m\to\infty}\left[A(\phi_{j,m})+\int_{\mathbb{R}^N}(V_k+c)\phi_{j,m}^2\mathrm{d}x\right]\\
	&\leq\lambda_{j,\infty}^{\mathrm{rad}}+c.
\end{aligned}
\end{equation*}
Since $V_k+c$ increases pointwise to $V_Q+c$, the monotone convergence theorem yields
\begin{equation*}
A(\phi_j)+\int_{\mathbb{R}^N}(V_Q+c)\phi_j^2\mathrm{d}x
\leq\lambda_{j,\infty}^{\mathrm{rad}}+c.
\end{equation*}
Thus $\phi_j\in\mathcal{X}_Q^{\mathrm{rad}}$.

By Lemma \ref{lem:density} and the continuity of $\mathfrak{q}_Q$ on $\mathcal{X}_Q$, \eqref{eq:limweak} extends to
\begin{equation}\label{eq:limform}
\mathfrak{q}_Q(\phi_j,\eta)
=\lambda_{j,\infty}^{\mathrm{rad}}\langle\phi_j,\eta\rangle_{L^2(\mathbb{R}^N)}
\end{equation}
for every $\eta\in\mathcal{X}_Q$. Hence
\begin{equation*}
\phi_j\in\mathcal{D}(L_Q),\quad
L_Q\phi_j=\lambda_{j,\infty}^{\mathrm{rad}}\phi_j.
\end{equation*}

By the definition of $\lambda_1^{\mathrm{rad}}$ and \eqref{eq:orthlim},
\begin{equation*}
\lambda_1^{\mathrm{rad}}\leq\mathcal{R}_Q(\phi_1)=\lambda_{1,\infty}^{\mathrm{rad}}.
\end{equation*}
Together with $\lambda_{1,\infty}^{\mathrm{rad}}\leq\lambda_1^{\mathrm{rad}}$, this gives
$
\lambda_{1,\infty}^{\mathrm{rad}}=\lambda_1^{\mathrm{rad}}.
$
For $u=a_1\phi_1+a_2\phi_2$, \eqref{eq:limform} and \eqref{eq:orthlim} give
\begin{equation*}
\mathfrak{q}_Q(u,u)
=\lambda_{1,\infty}^{\mathrm{rad}}a_1^2+\lambda_{2,\infty}^{\mathrm{rad}}a_2^2
\leq \lambda_{2,\infty}^{\mathrm{rad}}(a_1^2+a_2^2).
\end{equation*}
Therefore, by the definition of $\lambda_2^{\mathrm{rad}}$,
\begin{equation*}
\lambda_2^{\mathrm{rad}}
\leq \sup_{0\neq u\in\operatorname{span}\{\phi_1,\phi_2\}}\mathcal{R}_Q(u)
=\lambda_{2,\infty}^{\mathrm{rad}}.
\end{equation*}
Together with $\lambda_{2,\infty}^{\mathrm{rad}}\leq\lambda_2^{\mathrm{rad}}$, this proves \eqref{eq:econv}. The preceding identities also give \eqref{eq:rad-eig}.
\end{proof}

Combining Lemmas \ref{lem:prs-truncated} and \ref{lem:prs-econv}, we now pass the one-node property of the second radial eigenfunctions of $L_m$ to $L_Q$ as $m\to\infty$.

\begin{proof}[Proof of Proposition \ref{prop:prs}]
Let $\psi_m$ be a normalized second radial eigenfunction of $L_m$. By Lemma \ref{lem:prs-truncated}, after changing its sign if necessary, there exists $r_m>0$ such that
\begin{equation*}
\psi_m(r)\geq0\quad\text{for }0\leq r<r_m,\quad
\psi_m(r)\leq0\quad\text{for }r>r_m,
\end{equation*}
and neither restriction vanishes almost everywhere.

Let $\phi_{1,m}$ be a normalized first radial eigenfunction of $L_m$. By Lemma \ref{lem:modulus},
\begin{equation*}
\mathcal{R}_m(|\phi_{1,m}|)\leq \mathcal{R}_m(\phi_{1,m})=\lambda_{1,m}^{\mathrm{rad}}.
\end{equation*}
Since $\lambda_{1,m}^{\mathrm{rad}}$ is the minimum of $\mathcal{R}_m$, equality holds. We may therefore choose $\phi_{1,m}\geq0$.

Since $\psi_m$ is normalized and $\lambda_{2,m}^{\mathrm{rad}}\leq\lambda_2^{\mathrm{rad}}$, the estimates in the proof of Lemma \ref{lem:prs-econv} show that $\{\psi_m\}$ is bounded in $H^s(\mathbb{R}^N)$ and tight in $L^2(\mathbb{R}^N)$. After passing to a subsequence,
\begin{equation*}
\psi_m\to\psi\quad\text{strongly in }L^2(\mathbb{R}^N),
\end{equation*}
where the proof of Lemma \ref{lem:prs-econv} and \eqref{eq:econv} give
\begin{equation*}
\psi\in\mathcal{D}(L_Q),\quad
L_Q\psi=\lambda_2^{\mathrm{rad}}\psi,\quad
\|\psi\|_{L^2(\mathbb{R}^N)}=1.
\end{equation*}

After passing to a further subsequence, the same argument applied to $\phi_{1,m}$ gives 
\begin{equation*}
\phi_{1,m}\to\phi_1\quad\text{strongly in }L^2(\mathbb{R}^N),
\end{equation*}
where
\begin{equation*}
\phi_1\in\mathcal{D}(L_Q),\quad
L_Q\phi_1=\lambda_1^{\mathrm{rad}}\phi_1,\quad
\|\phi_1\|_{L^2(\mathbb{R}^N)}=1.
\end{equation*}
Moreover,
\begin{equation*}
\|(\phi_1)_-\|_{L^2(\mathbb{R}^N)}
\leq \|\phi_1-\phi_{1,m}\|_{L^2(\mathbb{R}^N)}\to0,
\end{equation*}
so $\phi_1\geq0$ almost everywhere. Since $L_m$ is self-adjoint and
$
\lambda_{1,m}^{\mathrm{rad}}<\lambda_{2,m}^{\mathrm{rad}},
$
we have
$
\langle\psi_m,\phi_{1,m}\rangle_{L^2(\mathbb{R}^N)}=0.
$
Hence
\begin{equation}\label{eq:psiorth}
\langle\psi,\phi_1\rangle_{L^2(\mathbb{R}^N)}=0.
\end{equation}

We first show that $\psi$ changes sign. Suppose that $\psi\geq0$ almost everywhere. Since $\phi_1\geq0$, \eqref{eq:psiorth} gives $\phi_1\psi=0$ almost everywhere. Since both functions are nonzero, $\phi_1-\psi$ changes sign and
\begin{equation*}
|\phi_1-\psi|=\phi_1+\psi\quad\text{almost everywhere}.
\end{equation*}
The strict part of Lemma \ref{lem:modulus} gives
\begin{equation*}
A(\phi_1+\psi)<A(\phi_1-\psi).
\end{equation*}
Since $\phi_1\psi=0$ almost everywhere, the potential terms in $\mathfrak{q}_Q(\phi_1+\psi,\phi_1+\psi)$ and $\mathfrak{q}_Q(\phi_1-\psi,\phi_1-\psi)$ are equal. Hence
\begin{equation*}
\mathfrak{q}_Q(\phi_1+\psi,\phi_1+\psi)
<
\mathfrak{q}_Q(\phi_1-\psi,\phi_1-\psi).
\end{equation*}
On the other hand,
\begin{equation*}
\begin{aligned}
	&\mathfrak{q}_Q(\phi_1+\psi,\phi_1+\psi)-\mathfrak{q}_Q(\phi_1-\psi,\phi_1-\psi)\\
	&=4\mathfrak{q}_Q(\phi_1,\psi)
	=4\lambda_1^{\mathrm{rad}}\langle\phi_1,\psi\rangle_{L^2(\mathbb{R}^N)}
	=0,
\end{aligned}
\end{equation*}
a contradiction. The case $\psi\leq0$ is identical after replacing $\psi$ by $-\psi$.

The sequence $\{r_m\}$ is bounded away from zero. Otherwise, along a subsequence,
$
r_m\to0.
$
For every $a>0$, we have $\psi_m\leq0$ on $\{r>a\}$ for all sufficiently large $m$. Hence
\begin{equation*}
\|\psi_+\|_{L^2(\{r>a\})}\leq \|\psi-\psi_m\|_{L^2(\mathbb{R}^N)} \to0.
\end{equation*}
Thus $\psi\leq0$ almost everywhere on $\{r>a\}$. Since $a>0$ is arbitrary, $\psi\leq0$ almost everywhere in $\mathbb{R}^N$, a contradiction.

Similarly, $\{r_m\}$ is bounded. Otherwise, along a subsequence,
$
r_m\to\infty.
$
For every $b>0$, we have $\psi_m\geq0$ on $\{r<b\}$ for all sufficiently large $m$, and hence
\begin{equation*}
\|\psi_-\|_{L^2(\{r<b\})}\leq\|\psi-\psi_m\|_{L^2(\mathbb{R}^N)}\to0.
\end{equation*}
Thus $\psi\geq0$ almost everywhere in $\mathbb{R}^N$, again a contradiction. After passing to a subsequence, $r_m\to r_*$ for some $r_*\in(0,+\infty)$.

Fix $0<a<r_*<b$. For all sufficiently large $m$,
$
a<r_m<b.
$
Thus
\begin{equation*}
\|\psi_-\|_{L^2(\{r<a\})}+\|\psi_+\|_{L^2(\{r>b\})}\leq 2\|\psi-\psi_m\|_{L^2(\mathbb{R}^N)}
\to0.
\end{equation*}
It follows that
\begin{equation*}
\psi\geq0\quad\text{almost everywhere on }\{r<r_*\},\quad
\psi\leq0\quad\text{almost everywhere on }\{r>r_*\}.
\end{equation*}
If either restriction vanished almost everywhere, then $\psi$ would have one sign in $\mathbb{R}^N$, which is impossible. Hence neither restriction vanishes almost everywhere.
\end{proof}

\medskip

\section{Nondegeneracy of positive ground states}\label{sec:nondegeneracy}

Let $Q$ be a positive radial ground state of \eqref{eq:log}. In this section, we prove the nondegeneracy of $Q$. Using the Morse index of the associated quadratic form, the one-node property obtained in Section \ref{sec:variational-spectral} and the scaling identities, we first exclude nontrivial radial zero modes. The nonradial zero modes are then determined by the angular decomposition, which gives the full characterization of $\ker L_Q$.

\subsection{Morse index}

We first recall the definition of the Morse index and then prove that the quadratic form associated with $L_Q$ has Morse index one. This fact will be used in Lemma \ref{lem:rker} to show that any nontrivial radial zero mode would be a second radial eigenfunction.

For a symmetric quadratic form $\mathfrak{q}$ with domain $\mathcal{D}(\mathfrak{q})$, we define its Morse index by
\begin{equation}\label{eq:morse-index}
\operatorname{ind}(\mathfrak{q})
=\sup\left\{
\dim E:
\begin{array}{l}
E\subset\mathcal{D}(\mathfrak{q})\text{ is a linear subspace},\\
\mathfrak{q}(h,h)<0\text{ for every }h\in E\setminus\{0\}
\end{array}
\right\}.
\end{equation}
By Corollary \ref{cor:compact-resolvent}, $L_Q$ has compact resolvent. Hence, by the min-max principle, $\operatorname{ind}(\mathfrak q_Q)$ equals the number of negative eigenvalues of $L_Q$, counted with multiplicity \cite[Theorem XIII.2]{Reed-Simon-1978-4}.

\begin{lemma}\label{lem:LQ-on-Q}
We have $Q\in\mathcal{D}(L_Q)$ and
\begin{equation*}
L_Q Q=-Q \quad\text{in }L^2(\mathbb{R}^N).
\end{equation*}
Consequently,
\begin{equation}\label{eq:qneg}
\mathfrak{q}_Q(Q,Q)=-M<0.
\end{equation}
\end{lemma}

\begin{proof}
Since $Q\in\mathcal{D}_s$ and $(V_Q)_+\leq1+|\log Q|$, we have $Q\in\mathcal{X}_Q$. Moreover, the equation for $Q$ gives
\begin{equation*}
(-\Delta)^sQ+V_Q Q = Q\log Q-(1+\log Q)Q = -Q \quad\text{in }\mathcal{D}'(\mathbb{R}^N).
\end{equation*}
Since $Q\in L^2(\mathbb{R}^N)$, Lemma \ref{lem:LQ-domain} yields
\begin{equation*}
Q\in\mathcal{D}(L_Q),\quad L_Q Q=-Q.
\end{equation*}
Therefore
\begin{equation*}
\mathfrak{q}_Q(Q,Q) = \langle L_Q Q,Q\rangle_{L^2(\mathbb{R}^N)} = -\|Q\|_{L^2(\mathbb{R}^N)}^2 = -M.
\end{equation*}
\end{proof} 

Lemma \ref{lem:LQ-on-Q} gives $\operatorname{ind}(\mathfrak q_Q)\geq1$. The fixed-mass minimality of $Q$ yields the reverse inequality. 

\begin{lemma}\label{lem:morse}
The quadratic form $\mathfrak{q}_Q$ has Morse index one, that is,
\begin{equation*}
\operatorname{ind}(\mathfrak{q}_Q)=1.
\end{equation*}
\end{lemma}

\begin{proof}
By \eqref{eq:qneg},
\begin{equation*}
\mathfrak{q}_Q(Q,Q)=-M<0.
\end{equation*}
Thus $\operatorname{span}\{Q\}\subset\mathcal{X}_Q$ is a negative subspace. The definition \eqref{eq:morse-index} gives
$
\operatorname{ind}(\mathfrak{q}_Q)\geq1.
$

Let $h\in C_c^\infty(\mathbb{R}^N)$ satisfy
\begin{equation}\label{eq:htan}
\int_{\mathbb{R}^N}Qh\mathrm{d}x=0.
\end{equation}
Define
\begin{equation*}
\alpha(t) = \frac{\sqrt M}{\|Q+th\|_{L^2(\mathbb{R}^N)}}, \quad u_t=\alpha(t)(Q+th).
\end{equation*}
Then $B(u_t)=M$. By \eqref{eq:htan},
\begin{equation*}
\|Q+th\|_{L^2(\mathbb{R}^N)}^2 =M+t^2\|h\|_{L^2(\mathbb{R}^N)}^2,
\end{equation*}
so
\begin{equation*}
\alpha(0)=1, \quad \alpha'(0)=0, \quad \alpha''(0)=-\frac{\|h\|_{L^2(\mathbb{R}^N)}^2}{M},
\end{equation*}
and
\begin{equation}\label{eq:upath}
u_0=Q, \quad u_0'=h, \quad u_0''=-\frac{\|h\|_{L^2(\mathbb{R}^N)}^2}{M}Q.
\end{equation}

Put
\begin{equation*}
G(t)=\frac{1}{4}t^2-\frac{1}{2}t^2\log t, \quad t>0.
\end{equation*}
Then
\begin{equation}\label{eq:gder}
G'(t)=-t\log t, \quad G''(t)=-(1+\log t).
\end{equation}
Let $K=\operatorname{supp}h$. Since $Q$ is continuous and positive,
\begin{equation*}
m_K = \min_{x\in K}Q(x)>0.
\end{equation*}
Choose $t_0>0$ such that, for $|t|\leq t_0$,
\begin{equation*}
\frac{1}{2}\leq\alpha(t)\leq2, \quad |t|\|h\|_{L^\infty(\mathbb{R}^N)}\leq\frac{m_K}{2}.
\end{equation*}
Then
\begin{equation*}
Q+th\geq\frac{m_K}{2}\quad\text{on }K,
\end{equation*}
while $u_t=\alpha(t)Q>0$ on $\mathbb{R}^N\setminus K$. Hence $u_t>0$ in $\mathbb{R}^N$.

Since $\alpha$, $\alpha'$ and $\alpha''$ are bounded on $[-t_0,t_0]$, the functions $u_t$, $u_t'$ and $u_t''$ are uniformly bounded on $K$, whereas on $\mathbb{R}^N\setminus K$ they are constant multiples of $Q$. Therefore,
\begin{equation*}
\begin{aligned}
	&\left|G'(u_t)u_t'\right| + \left|G''(u_t)(u_t')^2+G'(u_t)u_t''\right| \\
	&\leq C\mathbf 1_K + CQ^2\left(1+|\log Q|\right)\mathbf 1_{\mathbb{R}^N\setminus K} \in L^1(\mathbb{R}^N),
\end{aligned}
\end{equation*}
uniformly for $|t|\leq t_0$. Hence differentiation under the integral sign is justified.
Since
\begin{equation*}
\mathcal{I}_0(u) = \frac{1}{2}A(u)+\int_{\mathbb{R}^N}G(u)\mathrm{d}x \quad\text{for }u>0,
\end{equation*}
relations \eqref{eq:upath} and \eqref{eq:gder} give
\begin{equation}\label{eq:secp}
\begin{aligned}
	\left.\frac{\mathrm{d}^2}{\mathrm{d}t^2}\mathcal{I}_0(u_t)\right|_{t=0} =& A(h) + \left\langle(-\Delta)^sQ,u_0''\right\rangle_{L^2(\mathbb{R}^N)} \\
	&- \int_{\mathbb{R}^N}(1+\log Q)h^2\mathrm{d}x - \int_{\mathbb{R}^N}Q\log Q u_0''\mathrm{d}x.
\end{aligned}
\end{equation}
Since $Q\log Q\in L^2(\mathbb{R}^N)$, $u_0''$ is a constant multiple of $Q$ and
$
(-\Delta)^sQ=Q\log Q,
$
we have
\begin{equation*}
\left\langle(-\Delta)^sQ,u_0''\right\rangle_{L^2(\mathbb{R}^N)} = \int_{\mathbb{R}^N}Q\log Q u_0''\mathrm{d}x.
\end{equation*}
Therefore, \eqref{eq:secp} reduces to
\begin{equation}\label{eq:secp-form}
\left.\frac{\mathrm{d}^2}{\mathrm{d}t^2}\mathcal{I}_0(u_t)\right|_{t=0} = A(h)-\int_{\mathbb{R}^N}(1+\log Q)h^2\mathrm{d}x = \mathfrak{q}_Q(h,h).
\end{equation}
By Lemma \ref{lem:fmass}, $Q$ minimizes $\mathcal{I}_0$ on $\mathcal{S}_M$. Since $u_t\in\mathcal{S}_M$, \eqref{eq:secp-form} gives
\begin{equation}\label{eq:tnng}
\mathfrak{q}_Q(h,h)\geq0\quad\text{for every }h\in C_c^\infty(\mathbb{R}^N)\cap Q^\perp.
\end{equation}

Now let $h\in\mathcal{X}_Q\cap Q^\perp$. By Lemma \ref{lem:density}, choose $\widetilde h_n\in C_c^\infty(\mathbb{R}^N)$ such that $\|\widetilde h_n-h\|_Q\to0$. Since $Q\neq0$ and $C_c^\infty(\mathbb{R}^N)$ is dense in $L^2(\mathbb{R}^N)$, we may choose $\chi\in C_c^\infty(\mathbb{R}^N)$ such that $\langle Q,\chi\rangle_{L^2(\mathbb{R}^N)}\neq0$. Set
\begin{equation*}
h_n=\widetilde h_n-\frac{\langle Q,\widetilde h_n\rangle_{L^2(\mathbb{R}^N)}}{\langle Q,\chi\rangle_{L^2(\mathbb{R}^N)}}\chi.
\end{equation*}
Then $h_n\in C_c^\infty(\mathbb{R}^N)\cap Q^\perp$. Moreover, since $\widetilde h_n\to h$ in $L^2(\mathbb{R}^N)$ and $h\perp Q$,
\begin{equation*}
\langle Q,\widetilde h_n\rangle_{L^2(\mathbb{R}^N)}\to0,
\end{equation*}
and therefore $\|h_n-h\|_Q\to0$. Passing to the limit in \eqref{eq:tnng}, we obtain
\begin{equation}\label{eq:tnga}
\mathfrak{q}_Q(h,h)\geq0\quad\text{for every }h\in\mathcal{X}_Q\cap Q^\perp.
\end{equation}

If $E\subset\mathcal{X}_Q$ is a negative subspace with $\dim E\geq2$, then
\begin{equation*}
\dim(E\cap Q^\perp)\geq\dim E-1\geq1.
\end{equation*}
Hence there exists $0\neq h\in E\cap Q^\perp$. Since $E$ is negative, $\mathfrak{q}_Q(h,h)<0$, contradicting \eqref{eq:tnga}. Therefore
$
\operatorname{ind}(\mathfrak{q}_Q)=1.
$
\end{proof}

\subsection{The radial kernel}

In this subsection, we show that $L_Q$ has trivial kernel in $L^2_{\mathrm{rad}}(\mathbb{R}^N)$. The main idea is to use the strict radial monotonicity of $Q$ from Proposition \ref{prop:known}, the one-node property of the second radial eigenfunction from Proposition \ref{prop:prs} and the two orthogonality relations \eqref{eq:qorth} and \eqref{eq:qlorth} to derive a contradiction under the assumption that
$
\ker(L_Q)\cap L^2_{\mathrm{rad}}(\mathbb{R}^N)\neq\{0\}.
$
We first establish an $L^\infty$ bound for eigenfunctions of $L_Q$.

\begin{lemma}\label{lem:eiglinfty}
Let $\mu\in\mathbb{R}$ and $v\in\mathcal{D}(L_Q)$ satisfy
\begin{equation*}
L_Qv=\mu v \quad\text{in }L^2(\mathbb{R}^N).
\end{equation*}
Then $v\in L^\infty(\mathbb{R}^N)$.
\end{lemma}

\begin{proof}
Choose $C_\mu>0$ such that
$
V_Q-\mu\geq-C_\mu.
$
For $\beta\geq1$ and $T>0$, define as in \cite[Lemma 2.4]{Li-Peng-Shuai-2022},
\begin{equation*}
\varphi_{\beta,T}(t)
=\begin{cases}
	0, & t\leq0,\\
	t^\beta, & 0<t<T,\\
	\beta T^{\beta-1}(t-T)+T^\beta, & t\geq T.
\end{cases}
\end{equation*}
Then $\varphi_{\beta,T}$ is increasing, convex, Lipschitz and
\begin{equation}\label{eq:mosphi}
t\varphi_{\beta,T}'(t)
\leq
\beta\varphi_{\beta,T}(t)
\quad\text{for almost every }t\in\mathbb{R}.
\end{equation}

Set
\begin{equation*}
\psi_{\beta,T}(t)=\varphi_{\beta,T}(t)\varphi_{\beta,T}'(t)=
\begin{cases}
	0, & t\leq0,\\
	\beta t^{2\beta-1}, & 0<t<T,\\
	\beta^2T^{2\beta-2}(t-T)+\beta T^{2\beta-1}, & t\geq T.
\end{cases}
\end{equation*}
Since $\beta\geq1$, the function $\psi_{\beta,T}$ is Lipschitz on $\mathbb{R}$ and satisfies $\psi_{\beta,T}(0)=0$. Hence
\begin{equation*}
\psi_{\beta,T}(v)\in H^s(\mathbb{R}^N), \quad
|\psi_{\beta,T}(v)|\leq C_{\beta,T}|v|.
\end{equation*}
Since $v\in\mathcal{D}(L_Q)\subset\mathcal{X}_Q$, it follows that
\begin{equation*}
\int_{\mathbb{R}^N}(V_Q)_+|\psi_{\beta,T}(v)|^2\mathrm{d}x
\leq C_{\beta,T}\int_{\mathbb{R}^N}(V_Q)_+v^2\mathrm{d}x<\infty.
\end{equation*}
Therefore, $\psi_{\beta,T}(v)\in\mathcal{X}_Q$ and is an admissible test function.

By the convex truncation estimate in \cite[(2.13)--(2.15)]{Li-Peng-Shuai-2022},
\begin{equation*}
A\left(\varphi_{\beta,T}(v)\right)
\leq\left\langle(-\Delta)^{\frac{s}{2}}v,(-\Delta)^{\frac{s}{2}}\psi_{\beta,T}(v)\right\rangle_{L^2(\mathbb{R}^N)}.
\end{equation*}
Testing $L_Qv=\mu v$ by $\psi_{\beta,T}(v)$ gives
\begin{equation*}
\begin{aligned}
	A\left(\varphi_{\beta,T}(v)\right)&\leq \int_{\mathbb{R}^N}(\mu-V_Q)
	v\psi_{\beta,T}(v)\mathrm{d}x\\
	&\leq C_\mu \int_{\mathbb{R}^N}v\varphi_{\beta,T}(v)\varphi_{\beta,T}'(v)\mathrm{d}x
	\leq C_\mu\beta \int_{\mathbb{R}^N}\varphi_{\beta,T}(v)^2\mathrm{d}x,
\end{aligned}
\end{equation*}
where we used $v\psi_{\beta,T}(v)\geq0$ and \eqref{eq:mosphi}.

Fix $q$ such that
$
2<q<2_s^*
$
and set
$
\chi=\frac{q}{2}>1.
$
By the Sobolev embedding theorem,
\begin{equation*}
\|\varphi_{\beta,T}(v)\|_{L^q(\mathbb{R}^N)}^2
\leq C\left(A\left(\varphi_{\beta,T}(v)\right)+\|\varphi_{\beta,T}(v)\|_{L^2(\mathbb{R}^N)}^2\right)
\leq C\beta \|\varphi_{\beta,T}(v)\|_{L^2(\mathbb{R}^N)}^2,
\end{equation*}
where $C>0$ is independent of $\beta$ and $T$. 
Assume that $v_+\in L^{2\beta}(\mathbb{R}^N)$. Since
$
0\leq\varphi_{\beta,T}(t)\leq t_+^\beta
$
and $\varphi_{\beta,T}(t)$ increases to $t_+^\beta$ as $T\to+\infty$, the monotone convergence theorem yields 
\begin{equation}\label{eq:mositer}
\|v_+\|_{L^{2\beta\chi}(\mathbb{R}^N)}
\leq (C\beta)^{\frac{1}{2\beta}}\|v_+\|_{L^{2\beta}(\mathbb{R}^N)}.
\end{equation} 
Set $\beta_k=\chi^k$ for $k\geq0$. Since $\beta_0=1$ and $v_+\in L^2(\mathbb{R}^N)$, we may apply \eqref{eq:mositer} successively with $\beta=\beta_k$ to obtain
\begin{equation*}
\|v_+\|_{L^{2\beta_{k+1}}(\mathbb{R}^N)}
\leq \|v_+\|_{L^2(\mathbb{R}^N)} \prod_{j=0}^k (C\beta_j)^{\frac{1}{2\beta_j}}.
\end{equation*}
Since $\beta_k\to+\infty$ and
$
\sum_{j=0}^\infty\frac{1+\log\beta_j}{\beta_j}<\infty,
$ 
\begin{equation*}
\|v_+\|_{L^\infty(\mathbb{R}^N)}\leq C\|v_+\|_{L^2(\mathbb{R}^N)}.
\end{equation*}
Applying the same argument to $-v$ gives
$
\|v_-\|_{L^\infty(\mathbb{R}^N)}\leq C\|v_-\|_{L^2(\mathbb{R}^N)}.
$
Therefore,
\begin{equation*}
\|v\|_{L^\infty(\mathbb{R}^N)}\leq C\|v\|_{L^2(\mathbb{R}^N)}.
\end{equation*}
\end{proof}

The following elementary estimate will be used to derive \eqref{eq:qlorth}.

\begin{lemma}\label{lem:scaling-remainder}
For $0<a\leq b$, define
\begin{equation*}
R(a,b) = a\left(1+\log b-\log a\right)-b.
\end{equation*}
Then
\begin{equation}\label{eq:rbd}
|R(a,b)|\leq b-a.
\end{equation}
Moreover,
\begin{equation}\label{eq:rsmall}
\frac{|R(a,b)|}{b-a}\to0 \quad\text{as }\frac{a}{b}\to1,
\end{equation}
where the quotient is defined as zero when $a=b$.
\end{lemma}

\begin{proof}
Set $z=\frac{a}{b}\in(0,1]$. Then
\begin{equation*}
R(a,b) = b\left(z-z\log z-1\right).
\end{equation*}
Since
\begin{equation*}
\frac{\mathrm{d}}{\mathrm{d}z}\left(z-z\log z-1\right)=-\log z\geq0, \quad \left.\left(z-z\log z-1\right)\right|_{z=1}=0,
\end{equation*}
we have $z-z\log z-1\leq0$, and therefore
\begin{equation*}
|R(a,b)| = b\left(1-z+z\log z\right) \leq b(1-z) = b-a,
\end{equation*}
which proves \eqref{eq:rbd}. If $z<1$, then
\begin{equation*}
\frac{|R(a,b)|}{b-a} = 1+\frac{z\log z}{1-z}.
\end{equation*}
Since $\log z=(z-1)+o(|z-1|)$ as $z\to1$, the right-hand side tends to zero. This proves \eqref{eq:rsmall}.
\end{proof}

We now combine the Morse index, the one-node property and the preceding estimates to exclude a radial zero mode.

\begin{lemma}\label{lem:rker}
It holds that
\begin{equation*}
\ker\left(L_Q\right)\cap L^2_{\mathrm{rad}}(\mathbb{R}^N)=\left\{0\right\}.
\end{equation*}
\end{lemma}

\begin{proof}
Suppose by contradiction that
\begin{equation*}
\ker(L_Q)\cap L^2_{\mathrm{rad}}(\mathbb{R}^N)\neq\{0\},
\end{equation*}
and choose $0\neq w\in\ker(L_Q)\cap L^2_{\mathrm{rad}}(\mathbb{R}^N)$. We claim that
\begin{equation}\label{eq:lambda2-zero}
\lambda_2^{\mathrm{rad}}=0.
\end{equation}

By Lemma \ref{lem:morse}, $\operatorname{ind}(\mathfrak{q}_Q)=1$, and therefore $\lambda_2^{\mathrm{rad}}\geq0$. Indeed, if $\lambda_2^{\mathrm{rad}}<0$, then by the definition of $\lambda_2^{\mathrm{rad}}$ there exists a two-dimensional subspace $E\subset\mathcal{X}_Q^{\mathrm{rad}}$ such that
\begin{equation*}
\sup_{0\neq u\in E}\mathcal{R}_Q(u)<0.
\end{equation*}
Hence $\mathfrak{q}_Q(u,u)<0$ for every $u\in E\setminus\{0\}$, contradicting $\operatorname{ind}(\mathfrak{q}_Q)=1$.

On the other hand, Lemma \ref{lem:LQ-on-Q} gives $Q\in\mathcal{D}(L_Q)$ and $L_QQ=-Q$. Since $w\in\ker L_Q$, \eqref{eq:weak-kernel} and the symmetry of $\mathfrak{q}_Q$ give
\begin{equation*}
\mathfrak{q}_Q(Q,w)=\mathfrak{q}_Q(w,Q)=0,\quad
\mathfrak{q}_Q(w,w)=0,
\end{equation*}
while
\begin{equation*}
\mathfrak{q}_Q(Q,Q)=\langle L_QQ,Q\rangle_{L^2(\mathbb{R}^N)}
=-\|Q\|_{L^2(\mathbb{R}^N)}^2=-M.
\end{equation*}
Moreover, $Q$ and $w$ are linearly independent, since $L_QQ=-Q$ whereas $L_Qw=0$. Thus $E=\operatorname{span}\{Q,w\}$ is a two-dimensional subspace of $\mathcal{X}_Q^{\mathrm{rad}}$, and for $u=aQ+bw$,
\begin{equation*}
\mathfrak{q}_Q(u,u)=-Ma^2\leq0.
\end{equation*}
Since equality is attained for every nonzero multiple of $w$,
\begin{equation*}
\sup_{0\neq u\in E}\mathcal{R}_Q(u)=0.
\end{equation*}
The definition of $\lambda_2^{\mathrm{rad}}$ therefore gives $\lambda_2^{\mathrm{rad}}\leq0$. Hence \eqref{eq:lambda2-zero} follows.

By Proposition \ref{prop:prs}, there exist a normalized radial function $v\in\mathcal{D}(L_Q)$ and $r_*>0$ such that
\begin{equation*}
L_Qv=\lambda_2^{\mathrm{rad}}v=0,
\end{equation*}
and, after replacing $v$ by $-v$ if necessary,
\begin{equation}\label{eq:vsign}
v\geq0\quad\text{almost everywhere in }B_{r_*},\quad
v\leq0\quad\text{almost everywhere in }\mathbb{R}^N\setminus B_{r_*},
\end{equation}
with $v\not\equiv0$ on either region. Moreover, Lemma \ref{lem:eiglinfty} gives $v\in L^\infty(\mathbb{R}^N)$.

Since $v\in\ker L_Q$, \eqref{eq:weak-kernel} and the symmetry of $\mathfrak{q}_Q$ yield
\begin{equation*}
0=\mathfrak{q}_Q(v,Q)=\mathfrak{q}_Q(Q,v)
=\langle L_QQ,v\rangle_{L^2(\mathbb{R}^N)}
=-\int_{\mathbb{R}^N}Qv\mathrm{d}x.
\end{equation*}
Hence
\begin{equation}\label{eq:qorth}
\int_{\mathbb{R}^N}Qv\mathrm{d}x=0.
\end{equation} 

For $\lambda\geq1$, set $Q_\lambda(x)=Q(\lambda x)$. Since $Q$ is radially decreasing, $0<Q_\lambda\leq Q$. Hence \eqref{eq:qwgt} gives $Q_\lambda\in\mathcal{X}_Q$. Taking $Q_\lambda$ as a test function in $L_Qv=0$, we obtain
\begin{equation}\label{eq:tvq}
\left\langle(-\Delta)^{\frac{s}{2}}v,(-\Delta)^{\frac{s}{2}}Q_\lambda\right\rangle_{L^2(\mathbb{R}^N)}
=\int_{\mathbb{R}^N}(1+\log Q)vQ_\lambda\mathrm{d}x.
\end{equation}
Since
\begin{equation*}
(-\Delta)^sQ_\lambda=\lambda^{2s}Q_\lambda\log Q_\lambda
\end{equation*}
and $Q_\lambda\log Q_\lambda\in L^2(\mathbb{R}^N)$ by \eqref{eq:qint} and scaling, we may take $v$ as a test function in this equation. Hence
\begin{equation*}
\left\langle(-\Delta)^{\frac{s}{2}}Q_\lambda,(-\Delta)^{\frac{s}{2}}v\right\rangle_{L^2(\mathbb{R}^N)}
=\lambda^{2s}\int_{\mathbb{R}^N}vQ_\lambda\log Q_\lambda\mathrm{d}x.
\end{equation*}
Together with \eqref{eq:tvq}, this gives
\begin{equation}\label{eq:fscale}
\int_{\mathbb{R}^N}vQ_\lambda\left(1+\log Q-\lambda^{2s}\log Q_\lambda\right)\mathrm{d}x=0.
\end{equation}

We next derive a second orthogonality relation. By \eqref{eq:qint}, $x\cdot\nabla Q\in L^1(\mathbb{R}^N)$. For $\lambda>1$,
\begin{equation*}
\frac{Q-Q_\lambda}{\lambda-1}=-\frac{1}{\lambda-1}\int_1^\lambda x\cdot\nabla Q(tx)\mathrm{d}t.
\end{equation*}
Since
\begin{equation*}
x\cdot\nabla Q(tx)\to x\cdot\nabla Q(x)
\quad\text{in }L^1(\mathbb{R}^N)
\end{equation*}
as $t\to1$, it follows that
\begin{equation*}
\frac{Q-Q_\lambda}{\lambda-1}\to-x\cdot\nabla Q
\quad\text{in }L^1(\mathbb{R}^N)
\end{equation*}
as $\lambda\downarrow1$. For $0<a\leq b$, define
\begin{equation*}
R(a,b) = a\left(1+\log b-\log a\right)-b.
\end{equation*}
By \eqref{eq:rbd},
\begin{equation*}
\left|\frac{R(Q_\lambda,Q)}{\lambda-1}\right|
\leq\frac{Q-Q_\lambda}{\lambda-1}.
\end{equation*}
The right-hand side converges in $L^1(\mathbb{R}^N)$ and is therefore uniformly integrable. Moreover, \eqref{eq:rsmall} implies
\begin{equation*}
\frac{R(Q_\lambda,Q)}{\lambda-1}\to0
\quad\text{a.e. in }\mathbb{R}^N.
\end{equation*}
Vitali's theorem and $v\in L^\infty(\mathbb{R}^N)$ yield
\begin{equation}\label{eq:rlim}
\frac{1}{\lambda-1}\int_{\mathbb{R}^N}vR(Q_\lambda,Q)\mathrm{d}x\to0.
\end{equation}

Using \eqref{eq:qorth} and the definition of $R$, we rewrite \eqref{eq:fscale} as
\begin{equation*}
\int_{\mathbb{R}^N}vR(Q_\lambda,Q)\mathrm{d}x-(\lambda^{2s}-1)\int_{\mathbb{R}^N}vQ_\lambda\log Q_\lambda\mathrm{d}x=0.
\end{equation*}
Since
\begin{equation*}
Q_\lambda\log Q_\lambda\to Q\log Q \quad\text{in }L^2(\mathbb{R}^N),
\quad \frac{\lambda^{2s}-1}{\lambda-1}\to2s,
\end{equation*}
dividing by $\lambda-1$ and using \eqref{eq:rlim} gives
\begin{equation}\label{eq:qlorth}
\int_{\mathbb{R}^N}Qv\log Q\mathrm{d}x=0.
\end{equation}

Combining \eqref{eq:qorth} and \eqref{eq:qlorth}, we obtain
\begin{equation*}
0=\int_{\mathbb{R}^N}Qv\left(\log Q-\log Q(r_*)\right)\mathrm{d}x.
\end{equation*}
By \eqref{eq:vsign} and the strict radial monotonicity of $Q$,
\begin{equation*}
Qv\left(\log Q-\log Q(r_*)\right)\geq0
\quad\text{a.e. in }\mathbb{R}^N,
\end{equation*}
and the inequality is strict on a set of positive measure. This is a contradiction.
\end{proof}

\subsection{The nonradial sectors}
In the last subsection, we determine the nonradial zero modes of $L_Q$ by the spherical-harmonic decomposition and the ordering of the angular sectors. In particular, we show that the nonradial part of $\ker L_Q$ is generated by the translation modes $\partial_{x_1}Q,\cdots,\partial_{x_N}Q$.

Assume that $N\geq2$ and write $x=r\omega$, where $r=|x|$ and $\omega\in\mathbb{S}^{N-1}=\{\omega\in\mathbb{R}^N:|\omega|=1\}$. Set
\begin{equation*}
\mathcal{K}=L^2\left((0,\infty),r^{N-1}\mathrm{d}r\right),
\end{equation*}
with inner product and norm
\begin{equation*}
\langle f,g\rangle_{\mathcal{K}}=\int_0^\infty f(r)g(r)r^{N-1}\mathrm{d}r,\quad
\|f\|_{\mathcal{K}}^2=\int_0^\infty |f(r)|^2r^{N-1}\mathrm{d}r.
\end{equation*}
Let $\mathbb{N}_0=\{0,1,2,\cdots\}$. For $\ell\in\mathbb{N}_0$, let $\mathcal{Y}_\ell$ be the eigenspace of the Laplace-Beltrami operator $-\Delta_{\mathbb{S}^{N-1}}$ corresponding to the eigenvalue $\ell(\ell+N-2)$. Set $d_\ell=\dim\mathcal{Y}_\ell$ and let $\{Y_{\ell,k}\}_{k=1}^{d_\ell}$ be an orthonormal basis of $\mathcal{Y}_\ell$ in $L^2(\mathbb{S}^{N-1})$. Define
\begin{equation*}
\mathcal{H}_\ell=\left\{\sum_{k=1}^{d_\ell}f_k(r)Y_{\ell,k}(\omega):f_k\in\mathcal{K}\right\}.
\end{equation*}
Then
\begin{equation*}
L^2(\mathbb{R}^N)=\bigoplus_{\ell=0}^\infty\mathcal{H}_\ell
\end{equation*}
is the orthogonal decomposition into spherical harmonic sectors. We refer to \cite[Section 7.2 and Appendix C.3]{Frank-Lenzmann-Silvestre-2016} for this decomposition.

For $\ell\in\mathbb{N}_0$, consider the differential operator
\begin{equation*}
-\frac{\mathrm{d}^2}{\mathrm{d}r^2}-\frac{N-1}{r}\frac{\mathrm{d}}{\mathrm{d}r}+\frac{\ell(\ell+N-2)}{r^2}
\end{equation*}
initially defined on $C_c^\infty(0,\infty)$. We denote its Friedrichs extension in $\mathcal{K}$ by $-\Delta_\ell$. Since $-\Delta_\ell$ is nonnegative and self-adjoint, its fractional powers are defined by the spectral theorem, see \cite[Section 7.2 and Appendix C.3]{Frank-Lenzmann-Silvestre-2016} and \cite[Chapter VI, Section 5]{Kato-1995}.

For every $\ell\in\mathbb{N}_0$, set
\begin{equation*}
\mathcal{D}(\mathfrak{q}_\ell)=\left\{f\in\mathcal{D}\left((-\Delta_\ell)^{\frac{s}{2}}\right):\int_0^\infty(V_Q)_+(r)|f(r)|^2r^{N-1}\mathrm{d}r<\infty\right\}
\end{equation*}
and
\begin{equation*}
\mathfrak{q}_\ell(f,g)=\left\langle(-\Delta_\ell)^{\frac{s}{2}}f,(-\Delta_\ell)^{\frac{s}{2}}g\right\rangle_{\mathcal{K}}+\int_0^\infty V_Q(r)f(r)g(r)r^{N-1}\mathrm{d}r.
\end{equation*}
We equip $\mathcal{D}(\mathfrak{q}_\ell)$ with the norm
\begin{equation*}
\|f\|_{\ell,Q}^2=\|f\|_{\mathcal{K}}^2+\left\|(-\Delta_\ell)^{\frac{s}{2}}f\right\|_{\mathcal{K}}^2+\int_0^\infty(V_Q)_+(r)|f(r)|^2r^{N-1}\mathrm{d}r.
\end{equation*}
Since $(V_Q)_-\in L^\infty(\mathbb{R}^N)$, $\mathfrak{q}_\ell$ is well defined and symmetric on $\mathcal{D}(\mathfrak{q}_\ell)$. Define
\begin{equation*}
\mathcal{N}_\ell=\left\{f\in\mathcal{D}(\mathfrak{q}_\ell):\mathfrak{q}_\ell(f,g)=0
\quad\text{for every }g\in\mathcal{D}(\mathfrak{q}_\ell)\right\}.
\end{equation*}

\begin{lemma}\label{lem:sector-dec}
Let
\begin{equation*}
h(r,\omega)=\sum_{\ell=0}^\infty\sum_{k=1}^{d_\ell}f_{\ell,k}(r)Y_{\ell,k}(\omega)
\end{equation*}
be the spherical harmonic expansion of $h\in L^2(\mathbb{R}^N)$. Then
\begin{equation*}
h\in\mathcal{X}_Q \iff f_{\ell,k}\in\mathcal{D}(\mathfrak{q}_\ell)\text{ for every }\ell,k,\quad
\sum_{\ell=0}^\infty\sum_{k=1}^{d_\ell}\|f_{\ell,k}\|_{\ell,Q}^2<\infty.
\end{equation*}
If $h,\phi\in\mathcal{X}_Q$ have expansions
\begin{equation*}
h=\sum_{\ell=0}^\infty\sum_{k=1}^{d_\ell}f_{\ell,k}Y_{\ell,k},\quad \phi=\sum_{\ell=0}^\infty\sum_{k=1}^{d_\ell}g_{\ell,k}Y_{\ell,k},
\end{equation*}
then
\begin{equation}\label{eq:sector-form}
\mathfrak{q}_Q(h,\phi)=\sum_{\ell=0}^\infty\sum_{k=1}^{d_\ell}\mathfrak{q}_\ell(f_{\ell,k},g_{\ell,k}).
\end{equation}
Moreover, for $h\in\mathcal{X}_Q$,
\begin{equation}\label{eq:sector-kernel-dec}
h\in\ker L_Q \iff f_{\ell,k}\in\mathcal{N}_\ell \quad\text{for every }\ell,k.
\end{equation}
\end{lemma}

\begin{proof}
By the spherical-harmonic decomposition in \cite[Appendix C.3, equations (C.17)--(C.18)]{Frank-Lenzmann-Silvestre-2016}, applied with zero potential,
\begin{equation*}
\|h\|_{L^2(\mathbb{R}^N)}^2=\sum_{\ell=0}^\infty\sum_{k=1}^{d_\ell}\|f_{\ell,k}\|_{\mathcal{K}}^2,
\end{equation*}
and
\begin{equation*}
h\in H^s(\mathbb{R}^N) \iff f_{\ell,k}\in\mathcal D\left((-\Delta_\ell)^{\frac{s}{2}}\right)\text{ for every }\ell,k,\quad
\sum_{\ell=0}^\infty\sum_{k=1}^{d_\ell}
\left\|(-\Delta_\ell)^{\frac{s}{2}}f_{\ell,k}\right\|_{\mathcal{K}}^2<\infty.
\end{equation*}
Moreover,
\begin{equation}\label{eq:sector-kinetic}
\left\|(-\Delta)^{\frac{s}{2}}h\right\|_{L^2(\mathbb{R}^N)}^2=\sum_{\ell=0}^\infty\sum_{k=1}^{d_\ell}
\left\|(-\Delta_\ell)^{\frac{s}{2}}f_{\ell,k}\right\|_{\mathcal{K}}^2.
\end{equation}

Since $(V_Q)_+$ is radial, Parseval's identity on $\mathbb{S}^{N-1}$ and Tonelli's theorem give
\begin{equation*}
\int_{\mathbb{R}^N}(V_Q)_+(x)|h(x)|^2\mathrm{d}x=\sum_{\ell=0}^\infty\sum_{k=1}^{d_\ell}
\int_0^\infty(V_Q)_+(r)|f_{\ell,k}(r)|^2r^{N-1}\mathrm{d}r.
\end{equation*}
This proves the characterization of $\mathcal{X}_Q$.

Let $h,\phi\in\mathcal{X}_Q$. Orthogonality of the spherical harmonics gives
\begin{equation*}
\left\langle(-\Delta)^{\frac{s}{2}}h,(-\Delta)^{\frac{s}{2}}\phi\right\rangle_{L^2(\mathbb{R}^N)}=\sum_{\ell=0}^\infty\sum_{k=1}^{d_\ell}
\left\langle(-\Delta_\ell)^{\frac{s}{2}}f_{\ell,k},(-\Delta_\ell)^{\frac{s}{2}}g_{\ell,k}\right\rangle_{\mathcal{K}}.
\end{equation*}
Since $V_Q$ is radial,
\begin{equation*}
\int_{\mathbb{R}^N}V_Q(x)h(x)\phi(x)\mathrm{d}x=\sum_{\ell=0}^\infty\sum_{k=1}^{d_\ell}
\int_0^\infty V_Q(r)f_{\ell,k}(r)g_{\ell,k}(r)r^{N-1}\mathrm{d}r.
\end{equation*}
The series on the right-hand side are absolutely convergent by the Cauchy-Schwarz inequality, using $(V_Q)_-\in L^\infty(\mathbb{R}^N)$ for the negative part of the potential term. Hence \eqref{eq:sector-form} follows.

It remains to prove \eqref{eq:sector-kernel-dec}. Suppose first that $h\in\ker L_Q$. Fix $\ell$, $k$ and $g\in\mathcal{D}(\mathfrak{q}_\ell)$, set
\begin{equation*}
\phi(r,\omega)=g(r)Y_{\ell,k}(\omega).
\end{equation*}
By the first assertion, $\phi\in\mathcal{X}_Q$. Hence \eqref{eq:weak-kernel} and \eqref{eq:sector-form} give
\begin{equation*}
\mathfrak{q}_\ell(f_{\ell,k},g)=\mathfrak{q}_Q(h,\phi)=0.
\end{equation*}
Since this holds for every $g\in\mathcal{D}(\mathfrak{q}_\ell)$, we have $f_{\ell,k}\in\mathcal{N}_\ell$.

Conversely, suppose that $f_{\ell,k}\in\mathcal{N}_\ell$ for every $\ell$ and $k$. Then for every $\phi\in\mathcal{X}_Q$, \eqref{eq:sector-form} gives
\begin{equation*}
\mathfrak{q}_Q(h,\phi)=\sum_{\ell=0}^\infty\sum_{k=1}^{d_\ell}
\mathfrak{q}_\ell(f_{\ell,k},g_{\ell,k})=0.
\end{equation*}
Thus $h\in\ker L_Q$ by \eqref{eq:weak-kernel}.
\end{proof}

We finally compare the quadratic forms $\mathfrak{q}_\ell$, $\ell\in\mathbb{N}_0$, associated with different angular sectors.

\begin{lemma}\label{lem:angular}
Let $0\leq\ell<\ell'$. Then
\begin{equation}\label{eq:angular-ordering}
\mathcal{D}(\mathfrak{q}_{\ell'})\subset\mathcal{D}(\mathfrak{q}_\ell),\quad
\mathfrak{q}_\ell(f,f)<\mathfrak{q}_{\ell'}(f,f)
\end{equation}
for every $0\neq f\in\mathcal{D}(\mathfrak{q}_{\ell'})$.
\end{lemma}

\begin{proof}
Applying \cite[Lemma C.7]{Frank-Lenzmann-Silvestre-2016} with zero potential gives
$
(-\Delta_{\ell'})^s>(-\Delta_\ell)^s
$
in the sense of quadratic forms. Thus
\begin{equation*}
\mathcal{D}\left((-\Delta_{\ell'})^{\frac{s}{2}}\right)
\subset \mathcal{D}\left((-\Delta_\ell)^{\frac{s}{2}}\right), \quad
\left\|(-\Delta_\ell)^{\frac{s}{2}}f\right\|_{\mathcal{K}}^2
<\left\|(-\Delta_{\ell'})^{\frac{s}{2}}f\right\|_{\mathcal{K}}^2
\end{equation*}
for every $0\neq f\in\mathcal{D}((-\Delta_{\ell'})^{\frac{s}{2}})$. Since
\begin{equation*}
\int_0^\infty(V_Q)_+(r)|f(r)|^2r^{N-1}\mathrm{d}r<\infty
\end{equation*}
is independent of $\ell$, we obtain
$
\mathcal{D}(\mathfrak{q}_{\ell'})\subset\mathcal{D}(\mathfrak{q}_\ell).
$
The potential term is the same in both sectors. Therefore, for every $0\neq f\in\mathcal{D}(\mathfrak{q}_{\ell'})$,
\begin{equation*}
\mathfrak{q}_{\ell'}(f,f)-\mathfrak{q}_\ell(f,f)=\left\|(-\Delta_{\ell'})^{\frac{s}{2}}f\right\|_{\mathcal{K}}^2-\left\|(-\Delta_\ell)^{\frac{s}{2}}f\right\|_{\mathcal{K}}^2
>0.
\end{equation*}
This proves \eqref{eq:angular-ordering}.
\end{proof}

The angular ordering, together with $\operatorname{ind}(\mathfrak q_Q)=1$, now determines all nonradial zero modes. For $\omega=(\omega_1,\cdots,\omega_N)\in\mathbb{S}^{N-1}$, we denote by $\omega_j$ its $j$-th coordinate, so that $\omega_j=\frac{x_j}{r}$.

\begin{lemma}\label{lem:nonradial-kernel}
If $N\geq2$, then
\begin{equation*}
\mathcal{N}_1=\operatorname{span}\{Q'\},\quad
\mathcal{N}_\ell=\{0\}\quad\text{for }\ell\geq2.
\end{equation*}
Consequently,
\begin{equation*}
\ker L_Q\cap\mathcal{H}_1=\operatorname{span}\left\{\partial_{x_1}Q,\cdots,\partial_{x_N}Q\right\}.
\end{equation*}
\end{lemma}

\begin{proof}
Since $\mathcal{Y}_1=\operatorname{span}\{\omega_1,\cdots,\omega_N\}$, we may write $Y_{1,1}=\sum_{j=1}^N a_j\omega_j$ for some $(a_1,\cdots,a_N)\neq0$. Hence
\begin{equation*}
Q'(r)Y_{1,1}(\omega)=\sum_{j=1}^N a_jQ'(r)\frac{x_j}{r}=\sum_{j=1}^N a_j\partial_{x_j}Q.
\end{equation*}
By Lemma \ref{lem:itrans}, $Q'(r)Y_{1,1}(\omega)\in\ker L_Q$. Lemma \ref{lem:sector-dec} therefore gives $Q'\in\mathcal{N}_1$.

We first show that
\begin{equation}\label{eq:q1-positive}
\mathfrak{q}_1(f,f)\geq0
\quad\text{for every }f\in\mathcal{D}(\mathfrak{q}_1).
\end{equation}
Otherwise, let $0\neq f\in\mathcal{D}(\mathfrak{q}_1)$ satisfy $\mathfrak{q}_1(f,f)<0$. Choose $Y\in\mathcal{Y}_1$ with $\|Y\|_{L^2(\mathbb{S}^{N-1})}=1$ and set $h(r,\omega)=f(r)Y(\omega)$. Since $Q$ is radial and $Y\perp\mathcal{Y}_0$, we have
$
\langle Q,h\rangle_{L^2(\mathbb{R}^N)}=0.
$
By the sector decomposition and $\|Y\|_{L^2(\mathbb{S}^{N-1})}=1$,
\begin{equation*}
\mathfrak{q}_Q(h,h)=\mathfrak{q}_1(f,f)<0.
\end{equation*}
Since $L_QQ=-Q$,
\begin{equation*}
\mathfrak{q}_Q(Q,h)=-\langle Q,h\rangle_{L^2(\mathbb{R}^N)}=0,\quad\mathfrak{q}_Q(Q,Q)=-M<0.
\end{equation*}
Therefore, for every $(a,b)\neq(0,0)$,
\begin{equation*}
\mathfrak{q}_Q(aQ+bh,aQ+bh)=-Ma^2+b^2\mathfrak{q}_1(f,f)<0.
\end{equation*} Thus $\operatorname{span}\{Q,h\}$ is a two-dimensional negative subspace for $\mathfrak{q}_Q$, contradicting Lemma \ref{lem:morse}. This proves \eqref{eq:q1-positive}.

Suppose that $\dim\mathcal{N}_1\geq2$. For every $0\neq f\in\mathcal{N}_1$, the definition of $\mathcal{N}_1$ gives $\mathfrak{q}_1(f,f)=0$. By Lemma \ref{lem:angular},
\begin{equation*}
\mathfrak{q}_0(f,f)<\mathfrak{q}_1(f,f)=0.
\end{equation*}
Let $Y_0=|\mathbb{S}^{N-1}|^{-\frac{1}{2}}$. Then
\begin{equation*}
\mathfrak{q}_Q(fY_0,fY_0)=\mathfrak{q}_0(f,f)<0
\end{equation*}
for every $0\neq f\in\mathcal{N}_1$. Hence
$
\left\{fY_0:f\in\mathcal{N}_1\right\}
$
is a negative subspace for $\mathfrak{q}_Q$ of dimension at least two, again contradicting Lemma \ref{lem:morse}. Therefore, $\dim\mathcal{N}_1=1$. Since $Q'\in\mathcal{N}_1$ and $Q'\not\equiv0$,
\begin{equation*}
\mathcal{N}_1=\operatorname{span}\{Q'\}.
\end{equation*} 
Let $\ell\geq2$. If $0\neq f\in\mathcal{N}_\ell$, then $\mathfrak{q}_\ell(f,f)=0$. By Lemma \ref{lem:angular}, $f\in\mathcal{D}(\mathfrak{q}_1)$ and
\begin{equation*}
\mathfrak{q}_1(f,f)<\mathfrak{q}_\ell(f,f)=0,
\end{equation*}
contradicting \eqref{eq:q1-positive}. Hence $\mathcal{N}_\ell=\{0\}$ for every $\ell\geq2$.

Finally, let $h\in\ker L_Q\cap\mathcal{H}_1$ and write
\begin{equation*}
h(r,\omega)=\sum_{k=1}^{d_1}f_k(r)Y_{1,k}(\omega).
\end{equation*}
By Lemma \ref{lem:sector-dec}, $f_k\in\mathcal{N}_1$ for every $k$, and hence $f_k=c_kQ'$. Since $\mathcal{Y}_1=\operatorname{span}\{\omega_1,\cdots,\omega_N\}$, we obtain
\begin{equation*}
h\in \operatorname{span}\left\{Q'(r)\frac{x_1}{r},\cdots,Q'(r)\frac{x_N}{r}\right\}=\operatorname{span}\left\{\partial_{x_1}Q,\cdots,\partial_{x_N}Q\right\}.
\end{equation*}
The reverse inclusion follows from Lemma \ref{lem:itrans}.
\end{proof}

When $N=1$, the even sector is already covered by the radial result, while the odd sector requires a separate argument. Define
\begin{equation*}
L^2_{\mathrm{even}}(\mathbb{R})=\left\{f\in L^2(\mathbb{R}):f(-x)=f(x)\text{ for almost every }x\right\}
\end{equation*}
and
\begin{equation*}
L^2_{\mathrm{odd}}(\mathbb{R})=\left\{f\in L^2(\mathbb{R}):f(-x)=-f(x)\text{ for almost every }x\right\}.
\end{equation*}
Radial functions are even, so Lemma \ref{lem:rker} gives
\begin{equation*}
\ker L_Q\cap L^2_{\mathrm{even}}(\mathbb{R})=\{0\}.
\end{equation*}

It remains to identify the odd kernel.

\begin{lemma}\label{lem:odd-kernel}
Assume that $N=1$. Then
\begin{equation*}
\ker L_Q\cap L^2_{\mathrm{odd}}(\mathbb{R})=\operatorname{span}\{Q'\}.
\end{equation*}
\end{lemma}

\begin{proof}
Let $h\in\mathcal{X}_Q\cap L^2_{\mathrm{odd}}(\mathbb{R})$ and set $g=h|_{(0,\infty)}$. Since $h(-x)=-h(x)$ and $V_Q$ is even, decomposing $\mathbb{R}^2$ into the four quadrants gives
\begin{equation}\label{eq:oddform}
\frac{1}{2}\mathfrak{q}_Q(h,h)
=\frac{c_{1,s}}{2}\iint_{(0,\infty)^2}\left(\frac{(g(x)-g(y))^2}{|x-y|^{1+2s}}
+\frac{(g(x)+g(y))^2}{(x+y)^{1+2s}}\right)\mathrm{d}x\mathrm{d}y+\int_0^\infty V_Q(x)g(x)^2\mathrm{d}x.
\end{equation}
Define
\begin{equation*}
h^\sharp(x)=\operatorname{sgn}(x)|h(x)|,\quad
\operatorname{sgn}(0)=0.
\end{equation*}
Then $h^\sharp$ is odd and its restriction to $(0,\infty)$ is $|g|$. Since
\begin{equation*}
(g(x)-g(y))^2-\left(|g(x)|-|g(y)|\right)^2=2\left(|g(x)g(y)|-g(x)g(y)\right)
\end{equation*}
and
\begin{equation*}
(g(x)+g(y))^2-\left(|g(x)|+|g(y)|\right)^2=-2\left(|g(x)g(y)|-g(x)g(y)\right),
\end{equation*}
the double-integral term in \eqref{eq:oddform} does not increase when $g$ is replaced by $|g|$. Hence $h^\sharp\in H^s(\mathbb{R})$. Since $|h^\sharp|=|h|$, we also have
\begin{equation*}
\int_{\mathbb{R}}(V_Q)_+|h^\sharp|^2\mathrm{d}x=\int_{\mathbb{R}}(V_Q)_+|h|^2\mathrm{d}x<\infty.
\end{equation*}
Thus $h^\sharp\in\mathcal{X}_Q$, and \eqref{eq:oddform} gives
\begin{equation}\label{eq:odddiam}
\frac{1}{2}\left(\mathfrak{q}_Q(h,h)-\mathfrak{q}_Q(h^\sharp,h^\sharp)\right)=\iint_{(0,\infty)^2}J_s(x,y)\left(|g(x)g(y)|-g(x)g(y)\right)\mathrm{d}x\mathrm{d}y\geq0,
\end{equation}
where
\begin{equation*}
J_s(x,y)=c_{1,s}\left(\frac{1}{|x-y|^{1+2s}}-\frac{1}{(x+y)^{1+2s}}\right)>0
\end{equation*}
for $x,y>0$. Equality in \eqref{eq:odddiam} holds if and only if
\begin{equation}\label{eq:constantsign}
g(x)g(y)\geq0
\quad\text{for almost every }(x,y)\in(0,\infty)^2.
\end{equation}

Let $h\in\mathcal{X}_Q\cap L^2_{\mathrm{odd}}(\mathbb{R})$. Since $L_QQ=-Q$, while $Q$ is even and $h$ is odd,
\begin{equation*}
\mathfrak{q}_Q(Q,h)=-\langle Q,h\rangle_{L^2(\mathbb{R})}=0,\quad\mathfrak{q}_Q(Q,Q)=-M<0.
\end{equation*}
If $\mathfrak{q}_Q(h,h)<0$, then for every $(a,b)\neq(0,0)$,
\begin{equation*}
\mathfrak{q}_Q(aQ+bh,aQ+bh)=-a^2M+b^2\mathfrak{q}_Q(h,h)<0.
\end{equation*}
Thus $\operatorname{span}\{Q,h\}$ is a two-dimensional negative subspace, contradicting Lemma \ref{lem:morse}. Therefore
\begin{equation}\label{eq:odd-nonnegative}
\mathfrak{q}_Q(h,h)\geq0\quad\text{for every }h\in\mathcal{X}_Q\cap L^2_{\mathrm{odd}}(\mathbb{R}).
\end{equation}

Now let $0\neq h\in\ker L_Q\cap L^2_{\mathrm{odd}}(\mathbb{R})$. Since $\mathfrak{q}_Q(h,h)=0$, \eqref{eq:odddiam} gives
\begin{equation*}
\mathfrak{q}_Q(h^\sharp,h^\sharp)\leq0.
\end{equation*}
Since $h^\sharp\in\mathcal{X}_Q\cap L^2_{\mathrm{odd}}(\mathbb{R})$, \eqref{eq:odd-nonnegative} yields $\mathfrak{q}_Q(h^\sharp,h^\sharp)\geq0$. Hence equality holds in \eqref{eq:odddiam}, and \eqref{eq:constantsign} shows that $h$ has a constant sign on $(0,\infty)$.

By Lemma \ref{lem:itrans}, $Q'\in\ker L_Q$, and Proposition \ref{prop:known} gives $Q'(x)<0$ for every $x>0$. Suppose that $h$ and $Q'$ are linearly independent. Set
\begin{equation*}
k=h-\frac{\int_0^\infty hQ'\mathrm{d}x}{\int_0^\infty |Q'|^2\mathrm{d}x}Q'.
\end{equation*}
Then $0\neq k\in\ker L_Q\cap L^2_{\mathrm{odd}}(\mathbb{R})$ and
\begin{equation*}
\int_0^\infty kQ'\mathrm{d}x=0.
\end{equation*}
Applying the conclusion above to $k$, we see that $k$ has a constant sign on $(0,\infty)$. Since $Q'(x)<0$ for every $x>0$ and $k\not\equiv0$,
\begin{equation*}
\int_0^\infty kQ'\mathrm{d}x\neq0,
\end{equation*}
a contradiction. Therefore $h\in\operatorname{span}\{Q'\}$. The reverse inclusion follows from Lemma \ref{lem:itrans}.
\end{proof}

Finally, combining the radial kernel result with the analysis of the nonradial sectors, we prove the nondegeneracy statement in Theorem \ref{thm:main}.

\begin{proof}[Proof of Theorem \ref{thm:main}-(1)]
Assume first that $N\geq2$. Let $h\in\ker L_Q$ and write
\begin{equation*}
h(r,\omega)=\sum_{\ell=0}^\infty\sum_{k=1}^{d_\ell}f_{\ell,k}(r)Y_{\ell,k}(\omega).
\end{equation*}
By Lemma \ref{lem:sector-dec}, $f_{\ell,k}\in\mathcal{N}_\ell$ for every $\ell,k$. The $\ell=0$ component is a radial kernel function and therefore vanishes by Lemma \ref{lem:rker}. By Lemma \ref{lem:nonradial-kernel}, all components with $\ell\geq2$ vanish, while the $\ell=1$ component belongs to
$
\operatorname{span}\left\{\partial_{x_1}Q,\cdots,\partial_{x_N}Q\right\}.
$
Hence
\begin{equation*}
h\in\operatorname{span}\left\{\partial_{x_1}Q,\cdots,\partial_{x_N}Q\right\}.
\end{equation*}
The reverse inclusion follows from Lemma \ref{lem:itrans}. Therefore
\begin{equation*}
\ker L_Q=\operatorname{span}\left\{\partial_{x_1}Q,\cdots,\partial_{x_N}Q\right\}.
\end{equation*}

Assume now that $N=1$. Let $h\in\ker L_Q$ and define
\begin{equation*}
h_{\mathrm{even}}(x)=\frac{h(x)+h(-x)}{2},\quad
h_{\mathrm{odd}}(x)=\frac{h(x)-h(-x)}{2}.
\end{equation*}
Since $V_Q$ is even and the fractional Laplacian is invariant under reflection, both $h_{\mathrm{even}}$ and $h_{\mathrm{odd}}$ belong to $\ker L_Q$. Lemma \ref{lem:rker} gives $h_{\mathrm{even}}=0$, while Lemma \ref{lem:odd-kernel} gives $h_{\mathrm{odd}}\in\operatorname{span}\{Q'\}$. Therefore
\begin{equation*}
\ker L_Q=\operatorname{span}\{Q'\}.
\end{equation*}
\end{proof}

\medskip
\section{Isolation of positive radial ground states}\label{sec:compactness-isolation}

In this section, we study the isolation of positive radial ground states of \eqref{eq:log}. The main idea is to use the compactness of radial minimizing sequences to show that, if two distinct radial ground states approach each other, their normalized difference would converge to a nontrivial radial element of $\ker L_Q$, contradicting the nondegeneracy result from Section \ref{sec:nondegeneracy}.

Fix a positive radial decreasing ground state $Q$ of \eqref{eq:log}, and let $M$ be given by \eqref{eq:M}. For $u\in L^2(\mathbb{R}^N)$, let $u^*$ be the symmetric decreasing rearrangement of $|u|$ and set
\begin{equation*}
d_Q(u)=\|u^*-Q\|_{L^2(\mathbb{R}^N)}.
\end{equation*}
For $u,v\in L^2(\mathbb{R}^N)$, \cite[Theorem 3.4]{Lieb-Loss-2001} gives
\begin{equation*}
\int_{\mathbb{R}^N}u^*v^*\mathrm{d}x\geq\int_{\mathbb{R}^N}|u||v|\mathrm{d}x.
\end{equation*}
Expanding the two squared norms and using Lemma \ref{lem:rearrangement}, we obtain
\begin{equation}\label{eq:rcontr}
\|u^*-v^*\|_{L^2(\mathbb{R}^N)} \leq \||u|-|v|\|_{L^2(\mathbb{R}^N)} \leq \|u-v\|_{L^2(\mathbb{R}^N)}.
\end{equation}

We first record the compactness gained from radial monotonicity.

\begin{lemma}\label{lem:rcomp}
Let $\{v_n\}$ be a sequence of nonnegative, radial and nonincreasing functions bounded in $H^s(\mathbb{R}^N)$. Then after passing to a subsequence, there exists $v\in H^s(\mathbb{R}^N)$ such that
\begin{equation*}
v_n\to v\quad\text{strongly in }L^q(\mathbb{R}^N)
\end{equation*}
for every $2<q<2_s^*$.
\end{lemma}

\begin{proof}
After passing to a subsequence, we may assume that
\begin{equation*}
v_n\rightharpoonup v\quad\text{weakly in }H^s(\mathbb{R}^N), \quad v_n\to v \quad\text{almost everywhere in }\mathbb{R}^N. 
\end{equation*}
By the local compact Sobolev embedding,
\begin{equation*}
v_n\to v\quad\text{strongly in }L^q(B_R)
\end{equation*}
for every fixed $R>0$ and every $2<q<2_s^*$.

Since $v_n$ is radial and nonincreasing,
\begin{equation*}
|B_1|r^Nv_n(r)^2\leq\int_{|x|\leq r}v_n^2\mathrm{d}x\leq C,
\end{equation*}
and hence
$
v_n(r)\leq Cr^{-\frac{N}{2}}.
$
Therefore, for every $q>2$,
\begin{equation*}
\int_{|x|>R}v_n^q\mathrm{d}x\leq\sup_{r>R}v_n(r)^{q-2}\int_{\mathbb{R}^N}v_n^2\mathrm{d}x\leq CR^{-\frac{N(q-2)}{2}}.
\end{equation*}
By Fatou's lemma, the same bound holds for $v$. Thus
\begin{equation*}
\int_{|x|>R}|v_n-v|^q\mathrm{d}x\leq CR^{-\frac{N(q-2)}{2}}
\end{equation*}
uniformly in $n$. Together with the strong convergence on $B_R$, this yields the strong convergence in $L^q(\mathbb{R}^N)$ by first letting $n\to\infty$ and then $R\to+\infty$.
\end{proof}

We now apply this radial compactness to fixed-mass minimizing sequences and rule out loss of mass at infinity.

\begin{lemma}\label{lem:lcomp}
Let $u_n\in\mathcal{S}_M$ satisfy $\mathcal{I}_0(u_n)\to c_0$. After replacing $u_n$ by $u_n^*$ and passing to a subsequence, there exists a positive radial decreasing ground state $U$ of \eqref{eq:log} such that
\begin{equation*} u_n\to U\quad\text{strongly in }L^2(\mathbb{R}^N). \end{equation*}
If every $u_n$ is a ground state of \eqref{eq:log}, then the convergence is strong in $H^s(\mathbb{R}^N)$.
\end{lemma}

\begin{proof}
Since $\mathcal{I}_0(u_n)\to c_0<+\infty$, we have $u_n\in\mathcal{D}_s$ for all sufficiently large $n$. By Lemma \ref{lem:fmass},
\begin{equation}\label{eq:mindif} 
A(u_n)-C_0(u_n)=2\left(\mathcal{I}_0(u_n)-c_0\right)\to0. 
\end{equation}
Choose $\delta>0$ such that $\delta<\frac{4s}{N}$ and $2+\delta<2_s^*$, and set
\begin{equation*} 
\theta=\frac{N\delta}{4s}<1. 
\end{equation*}
By \eqref{eq:log-growth} and Lemma \ref{lem:fractional-gn},
\begin{equation*} 
C_0(u_n)\leq C_\delta\|u_n\|_{L^{2+\delta}(\mathbb{R}^N)}^{2+\delta}\leq CA(u_n)^\theta B(u_n)^{1+\frac{\delta}{2}-\theta}=CA(u_n)^\theta, 
\end{equation*}
where we used $B(u_n)=M$ in the last equality. Combining this estimate with \eqref{eq:mindif} gives
\begin{equation*} 
A(u_n)\leq CA(u_n)^\theta+o(1). 
\end{equation*}
Since $0<\theta<1$, Young's inequality yields $A(u_n)\leq C$. Together with $B(u_n)=M$, this implies that $\{u_n\}$ is bounded in $H^s(\mathbb{R}^N)$.

Lemmas \ref{lem:modulus} and \ref{lem:rearrangement} give
\begin{equation*} 
B(u_n^*)=M,\quad C_0(u_n^*)=C_0(u_n),\quad A(u_n^*)\leq A(u_n), 
\end{equation*}
and therefore
\begin{equation*} c_0\leq\mathcal{I}_0(u_n^*)\leq\mathcal{I}_0(u_n)\to c_0. \end{equation*}
If $u_n$ is a ground state of \eqref{eq:log}, then $A(u_n)=C_0(u_n)$. Hence Lemma \ref{lem:fmass} and the preceding inequalities give
\begin{equation*} 
A(u_n)=C_0(u_n)=C_0(u_n^*)\leq A(u_n^*)\leq A(u_n), 
\end{equation*}
so
\begin{equation}\label{eq:rearr-energy} 
A(u_n^*)=A(u_n). 
\end{equation}
Replacing $u_n$ by $u_n^*$, we may assume that $u_n\geq0$ is radial and nonincreasing. After passing to a subsequence,
\begin{equation}\label{eq:lwk} 
u_n\rightharpoonup U\quad\text{weakly in }H^s(\mathbb{R}^N), 
\end{equation}
and Lemma \ref{lem:rcomp} gives
\begin{equation}\label{eq:lsq} 
u_n\to U\quad\text{strongly in }L^q(\mathbb{R}^N) 
\end{equation}
for every $2<q<2_s^*$. After passing to a further subsequence if necessary, $u_n\to U$ almost everywhere in $\mathbb{R}^N$. In particular, $U\geq0$ is radial and nonincreasing.

We claim that $B(U)=M$. Suppose instead that $B(U)<M$ and set
\begin{equation*}
m_0=\frac{M-B(U)}{2}>0.
\end{equation*}
For every fixed $R>0$, \eqref{eq:lsq} and H\"older's inequality give
\begin{equation*}
\int_{B_R}u_n^2\mathrm{d}x\to\int_{B_R}U^2\mathrm{d}x.
\end{equation*}
Since
\begin{equation*}
\int_{B_R}U^2\mathrm{d}x\leq B(U)=M-2m_0,
\end{equation*}
for all sufficiently large $n$,
\begin{equation}\label{eq:lmass}
\int_{|x|>R}u_n^2\mathrm{d}x\geq m_0.
\end{equation} 
Since $u_n$ is radial and nonincreasing and $B(u_n)=M$,
\begin{equation*}
|B_1|r^Nu_n(r)^2\leq\int_{|x|\leq r}u_n^2\mathrm{d}x\leq M.
\end{equation*}
Hence
\begin{equation*}
u_n(r)\leq C_Mr^{-\frac{N}{2}},\quad C_M=\left(\frac{M}{|B_1|}\right)^{\frac{1}{2}}.
\end{equation*}
Choose $R$ sufficiently large so that $C_MR^{-\frac{N}{2}}<1$. Then, for $|x|>R$,
\begin{equation*}
u_n^2\log u_n\leq u_n^2\log\left(C_MR^{-\frac{N}{2}}\right).
\end{equation*}
Since the logarithm on the right-hand side is negative, \eqref{eq:lmass} yields
\begin{equation*}
\int_{|x|>R}u_n^2\log u_n\mathrm{d}x\leq m_0\log\left(C_MR^{-\frac{N}{2}}\right)
\end{equation*}
for all sufficiently large $n$.

On the other hand, by \eqref{eq:log-growth}, the Sobolev embedding theorem and the boundedness of $\{u_n\}$ in $H^s(\mathbb{R}^N)$,
\begin{equation*}
\int_{B_R}u_n^2\log u_n\mathrm{d}x\leq\int_{\{u_n>1\}}u_n^2\log u_n\mathrm{d}x\leq C_\delta\int_{\mathbb{R}^N}u_n^{2+\delta}\mathrm{d}x\leq C.
\end{equation*}
Therefore
\begin{equation*}
\limsup_{n\to\infty}C_0(u_n)\leq C+m_0\log\left(C_MR^{-\frac{N}{2}}\right).
\end{equation*}
Letting $R\to+\infty$, we obtain
\begin{equation*}
C_0(u_n)\to-\infty.
\end{equation*}
This contradicts \eqref{eq:mindif}, since $C_0(u_n)=A(u_n)+o(1)$ and $A(u_n)\geq0$. Hence
\begin{equation*}
B(U)=M.
\end{equation*}
Since $u_n\rightharpoonup U$ weakly in $L^2(\mathbb{R}^N)$ and $\|u_n\|_{L^2(\mathbb{R}^N)}=\|U\|_{L^2(\mathbb{R}^N)}$, we obtain
\begin{equation*}
u_n\to U\quad\text{strongly in }L^2(\mathbb{R}^N).
\end{equation*}

Set $F(t)=t^2\log t$ for $t>0$ and $F(0)=0$. Fix $p\in(2+\delta,2_s^*)$. On every fixed ball $B_R$,
\begin{equation*}
F(u_n)\to F(U)\quad\text{almost everywhere}.
\end{equation*}
By Lemma \ref{lem:elementary-log},
\begin{equation*}
0\leq[F(u_n)]_-\leq\frac{1}{2\mathrm{e}},\quad [F(u_n)]_+\leq C_\delta u_n^{2+\delta}.
\end{equation*}
Since $\{u_n\}$ is bounded in $L^p(\mathbb{R}^N)$ and $\frac{p}{2+\delta}>1$, the sequence $\{u_n^{2+\delta}\}$ is uniformly integrable on $B_R$. Hence Vitali's theorem gives
\begin{equation}\label{eq:llog}
\int_{B_R}u_n^2\log u_n\mathrm{d}x\to\int_{B_R}U^2\log U\mathrm{d}x.
\end{equation}
Moreover,
\begin{equation*}
\int_{|x|>R}[u_n^2\log u_n]_+\mathrm{d}x\leq C_\delta\int_{|x|>R}u_n^{2+\delta}\mathrm{d}x.
\end{equation*}
By \eqref{eq:lsq} with exponent $2+\delta$, for every fixed $R>0$,
\begin{equation*}
\int_{|x|>R}u_n^{2+\delta}\mathrm{d}x\to\int_{|x|>R}U^{2+\delta}\mathrm{d}x.
\end{equation*}
Then \eqref{eq:llog} yields
\begin{equation*}
\limsup_{n\to\infty}C_0(u_n)\leq\int_{B_R}U^2\log U\mathrm{d}x+C_\delta\int_{|x|>R}U^{2+\delta}\mathrm{d}x.
\end{equation*}
Letting $R\to+\infty$, we obtain
\begin{equation}\label{eq:lsup}
\limsup_{n\to\infty}C_0(u_n)\leq\int_{\mathbb{R}^N}U^2\log U\mathrm{d}x.
\end{equation}
Its positive part is finite since $U\in L^{2+\delta}(\mathbb{R}^N)$. If
\begin{equation*}
\int_{\{U\leq1\}}U^2|\log U|\mathrm{d}x=\infty,
\end{equation*}
then the right-hand side of \eqref{eq:lsup} equals $-\infty$. This contradicts \eqref{eq:mindif}. Hence $U\in\mathcal{D}_s$, and \eqref{eq:lsup} becomes
\begin{equation*}
\limsup_{n\to\infty}C_0(u_n)\leq C_0(U).
\end{equation*}

By the weak lower semicontinuity of $A$ and \eqref{eq:mindif},
\begin{equation*}
A(U)-C_0(U)\leq\liminf_{n\to\infty}A(u_n)-\limsup_{n\to\infty}C_0(u_n)\leq\liminf_{n\to\infty}\left(A(u_n)-C_0(u_n)\right)=0.
\end{equation*}
Since $U\in\mathcal{S}_M\cap\mathcal{D}_s$, Lemma \ref{lem:fmass} gives the reverse inequality. Therefore
\begin{equation*}
A(U)=C_0(U),\quad \mathcal{I}_0(U)=\frac{M}{4}=c_0.
\end{equation*}
Thus $U$ is a global minimizer of $\mathcal{I}_0$ on $\mathcal{S}_M$.

Since $U\in\mathcal{S}_M$ and $\mathcal{I}_0(U)=c_0$, Lemma \ref{lem:fmass} shows that $U$ is a ground state of \eqref{eq:log}. Moreover, $U\geq0$ and $B(U)=M>0$, so $U\not\equiv0$. The positivity argument in \cite[Proof of Theorem 1.1-(i), equation (2.19)]{Li-Peng-Shuai-2022} gives $U>0$ in $\mathbb{R}^N$. Since $U$ is radial about the origin, Proposition \ref{prop:known} gives
\begin{equation*}
U'(r)<0\quad\text{for }r>0.
\end{equation*}

Finally, suppose that the sequence $\{u_n\}$ consists of ground states of \eqref{eq:log}. Since both $u_n$ and $U$ are ground states, Lemma \ref{lem:plog} and \eqref{eq:rearr-energy} give
\begin{equation*}
A(u_n)=A(U)=\frac{NM}{4s}.
\end{equation*}
Together with $B(u_n)=B(U)=M$, this yields
\begin{equation*}
\|u_n\|_{H^s(\mathbb{R}^N)}=\|U\|_{H^s(\mathbb{R}^N)}.
\end{equation*}
Hence \eqref{eq:lwk} and equality of the norms imply
\begin{equation*}
u_n\to U\quad\text{strongly in }H^s(\mathbb{R}^N).
\end{equation*}
\end{proof}

We next prove the isolation of radial ground states of \eqref{eq:log} by combining Lemma \ref{lem:lcomp} with the radial nondegeneracy result in Lemma \ref{lem:rker}.

\begin{lemma}\label{lem:isol}
Every positive radial ground state of \eqref{eq:log} is isolated in $L^2(\mathbb{R}^N)$ among positive radial ground states.
\end{lemma}

\begin{proof}
Suppose by contradiction that there are distinct positive radial ground states $Q_n$ such that
\begin{equation*}
Q_n\to Q\quad\text{strongly in }L^2(\mathbb{R}^N).
\end{equation*}
By Lemma \ref{lem:lcomp}, after passing to a subsequence,
\begin{equation*}
Q_n\to Q\quad\text{strongly in }H^s(\mathbb{R}^N).
\end{equation*} 
Set
\begin{equation*}
m_n=Q_n(0)=\|Q_n\|_{L^\infty(\mathbb{R}^N)},\quad
R_n=\left(\frac{4M}{|B_1|m_n^2}\right)^{\frac{1}{N}}.
\end{equation*}
We claim that
\begin{equation*}
Q_n(R_n)\leq\frac{m_n}{2}.
\end{equation*}
Indeed, since $Q_n$ is radial and nonincreasing, otherwise
\begin{equation*}
M=\int_{\mathbb{R}^N}Q_n^2\mathrm{d}x
\geq\int_{B_{R_n}}Q_n^2\mathrm{d}x
>\frac{m_n^2}{4}|B_1|R_n^N=M,
\end{equation*}
a contradiction. Hence $Q_n(y)\leq\frac{m_n}{2}$ for $|y|\geq R_n$. Evaluating $(-\Delta)^sQ_n=Q_n\log Q_n$ at the maximum point $0$, we obtain
\begin{equation*}
\begin{aligned}
	m_n\log m_n
	&=c_{N,s}\operatorname{P.V.}\int_{\mathbb{R}^N}\frac{m_n-Q_n(y)}{|y|^{N+2s}}\mathrm{d}y\\
	&\geq\frac{c_{N,s}m_n}{2}\int_{|y|\geq R_n}\frac{\mathrm{d}y}{|y|^{N+2s}}
	=Cm_n^{1+\frac{4s}{N}},
\end{aligned}
\end{equation*}
where $C>0$ is independent of $n$. Therefore
$
\log m_n\geq Cm_n^{\frac{4s}{N}}.
$
It follows that
\begin{equation}\label{eq:uniform-linfty}
\sup_n\|Q_n\|_{L^\infty(\mathbb{R}^N)}<\infty.
\end{equation} 
Since $t\log t$ is bounded on every bounded interval of $[0,+\infty)$, \eqref{eq:uniform-linfty} gives
\begin{equation*}
\sup_n\|Q_n\log Q_n\|_{L^\infty(\mathbb{R}^N)}<\infty.
\end{equation*}
Fix $0<\alpha<\min\{2s,1\}$. When $s=\frac{1}{2}$, choose $\varepsilon>0$ such that $\alpha<1-\varepsilon$. Applying the interior estimate \cite[Theorem 1.1-(a)]{RosOton-Serra-2016} to
\begin{equation*}
(-\Delta)^sQ_n=Q_n\log Q_n,
\end{equation*}
we obtain for every $R>0$,
\begin{equation*}
\|Q_n\|_{C^{0,\alpha}(B_R)}
\leq
C_R\left(\|Q_n\|_{L^\infty(\mathbb{R}^N)}+\|Q_n\log Q_n\|_{L^\infty(\mathbb{R}^N)}\right),
\end{equation*}
where $C_R>0$ is independent of $n$. Hence
\begin{equation*}
\sup_n\|Q_n\|_{C^{0,\alpha}(B_R)}<\infty.
\end{equation*}
By the Arzel\`a--Ascoli theorem, every subsequence of $\{Q_n\}$ has a further subsequence converging locally uniformly. Since $Q_n\to Q$ strongly in $L^2(\mathbb{R}^N)$, every such limit is $Q$. Therefore,
\begin{equation}\label{eq:qloc}
Q_n\to Q\quad\text{locally uniformly in }\mathbb{R}^N.
\end{equation}

Set
\begin{equation*}
w_n=\frac{Q_n-Q}{\|Q_n-Q\|_{L^2(\mathbb{R}^N)}}.
\end{equation*}
Then $\|w_n\|_{L^2(\mathbb{R}^N)}=1$ and
\begin{equation}\label{eq:wn}
(-\Delta)^sw_n=a_nw_n,
\end{equation}
where
\begin{equation*}
a_n(x)=
\begin{cases}
	\displaystyle\frac{Q_n(x)\log Q_n(x)-Q(x)\log Q(x)}{Q_n(x)-Q(x)}, & Q_n(x)\neq Q(x),\\[8pt]
	1+\log Q(x), & Q_n(x)=Q(x).
\end{cases}
\end{equation*}
Since $t\mapsto t\log t$ is convex on $(0,+\infty)$,
\begin{equation}\label{eq:acvx}
a_n(x)\leq1+\log\max\{Q_n(x),Q(x)\}.
\end{equation}
The uniform $L^\infty$ bound therefore gives $a_n(x)\leq C$ in $\mathbb{R}^N$.

Since $Q_n$ and $Q$ are radial decreasing and have mass $M$,
\begin{equation*}
|B_1|r^NQ_n(r)^2\leq M,\quad |B_1|r^NQ(r)^2\leq M.
\end{equation*}
Thus
\begin{equation*}
Q_n(r)+Q(r)\leq Cr^{-\frac{N}{2}}\quad\text{for }r\geq1,
\end{equation*}
and \eqref{eq:acvx} yields
\begin{equation}\label{eq:atail}
a_n(r)\leq C-\frac{N}{2}\log r\quad\text{for }r\geq1.
\end{equation}

Since $Q_n,Q\in H^s(\mathbb{R}^N)$, we have $w_n\in H^s(\mathbb{R}^N)$. Moreover, Proposition \ref{prop:known} gives
\begin{equation*}
Q_n\log Q_n, Q\log Q\in L^2(\mathbb{R}^N),
\end{equation*}
and hence
\begin{equation*}
a_nw_n=\frac{Q_n\log Q_n-Q\log Q}{\|Q_n-Q\|_{L^2(\mathbb{R}^N)}}\in L^2(\mathbb{R}^N).
\end{equation*}
Thus, $w_n$ is an admissible test function in \eqref{eq:wn}. Since $a_n\leq C$ and $\|w_n\|_{L^2(\mathbb{R}^N)}=1$, we obtain
\begin{equation}\label{eq:weng}
A(w_n)=\int_{\mathbb{R}^N}a_nw_n^2\mathrm{d}x\leq C.
\end{equation} 
Choose $R_0>1$ so large that \eqref{eq:atail} gives
\begin{equation*}
a_n(r)\leq-\frac{N}{4}\log r\quad\text{for }r>R_0.
\end{equation*}
Splitting the integral in \eqref{eq:weng} over $B_{R_0}$ and its complement, and using $a_n\leq C$ on $B_{R_0}$, we obtain
\begin{equation*}
A(w_n)\leq C-\frac{N}{4}\int_{|x|>R_0}\log|x|w_n^2\mathrm{d}x.
\end{equation*}
Hence
\begin{equation*}
A(w_n)+\frac{N}{4}\int_{|x|>R_0}\log|x|w_n^2\mathrm{d}x\leq C.
\end{equation*}
In particular, for $R>R_0$,
\begin{equation}\label{eq:wtail}
\int_{|x|>R}w_n^2\mathrm{d}x\leq\frac{C}{\log R},
\end{equation}
uniformly in $n$.

By \eqref{eq:weng} and $\|w_n\|_{L^2(\mathbb{R}^N)}=1$, $\{w_n\}$ is bounded in $H^s(\mathbb{R}^N)$. After passing to a subsequence,
\begin{equation}\label{eq:wweak}
w_n\rightharpoonup w\quad\text{weakly in }H^s(\mathbb{R}^N).
\end{equation}
For every fixed $R>0$, the local compact Sobolev embedding theorem gives
\begin{equation*}
w_n\to w\quad\text{strongly in }L^2(B_R).
\end{equation*}
Using \eqref{eq:wtail}, letting first $n\to\infty$ and then $R\to+\infty$ yields
\begin{equation}\label{eq:wstrong}
w_n\to w\quad\text{strongly in }L^2(\mathbb{R}^N).
\end{equation}
Therefore
$
\|w\|_{L^2(\mathbb{R}^N)}=1.
$
Since each $w_n$ is radial, so is $w$.

Let $K\subset\mathbb{R}^N$ be compact. Since $Q>0$, we have
$
\min_KQ>0.
$
By \eqref{eq:qloc}, $Q_n\to Q$ uniformly on $K$. Therefore, using the definition of $a_n$,
\begin{equation} \label{eq:alocal}
a_n\to1+\log Q\quad\text{uniformly on }K.
\end{equation}
For every $\varphi\in C_c^\infty(\mathbb{R}^N)$, the weak form of \eqref{eq:wn} is
\begin{equation*}
\left\langle(-\Delta)^{\frac{s}{2}}w_n,(-\Delta)^{\frac{s}{2}}\varphi\right\rangle_{L^2(\mathbb{R}^N)}
=\int_{\mathbb{R}^N}a_nw_n\varphi\mathrm{d}x.
\end{equation*}
Passing to the limit by \eqref{eq:wweak}, \eqref{eq:wstrong} and \eqref{eq:alocal}, we obtain
\begin{equation}\label{eq:wlin}
(-\Delta)^sw+V_Qw=0\quad\text{in }\mathcal{D}'(\mathbb{R}^N).
\end{equation}

It remains to show that $w\in\mathcal{X}_Q$. Choose $\eta\in C_c^\infty(B_2)$ with $0\leq\eta\leq1$ and $\eta=1$ on $B_1$, and put $\eta_R(x)=\eta(\frac{x}{R})$. Since $V_Q$ is locally bounded, $V_Qw\in L^2_{\mathrm{loc}}(\mathbb{R}^N)$. Applying Lemma \ref{lem:cutoff} to \eqref{eq:wlin} gives
\begin{equation*}
A(\eta_Rw)+\int_{\mathbb{R}^N}(V_Q)_+\eta_R^2w^2\mathrm{d}x
=
\mathcal{E}_R(w)+\int_{\mathbb{R}^N}(V_Q)_-\eta_R^2w^2\mathrm{d}x.
\end{equation*}
Therefore
\begin{equation*}
A(\eta_Rw)+\int_{\mathbb{R}^N}(V_Q)_+\eta_R^2w^2\mathrm{d}x
\leq \left(CR^{-2s}+\|(V_Q)_-\|_{L^\infty(\mathbb{R}^N)}\right)\|w\|_{L^2(\mathbb{R}^N)}^2.
\end{equation*}
Since $A(\eta_Rw)\geq0$, Fatou's lemma gives
\begin{equation*}
\int_{\mathbb{R}^N}(V_Q)_+w^2\mathrm{d}x<\infty.
\end{equation*}
Hence $w\in\mathcal{X}_Q$. By \eqref{eq:wlin} and Lemma \ref{lem:LQ-domain},
\begin{equation*}
w\in\mathcal{D}(L_Q),\quad L_Qw=0.
\end{equation*}
This contradicts Lemma \ref{lem:rker}, since $w$ is radial and
$
\|w\|_{L^2(\mathbb{R}^N)}=1.
$
\end{proof}

Finally, using Lemmas \ref{lem:lcomp} and \ref{lem:isol}, we can get a positive energy gap on the boundary of a small neighborhood of the fixed radial ground state $Q$.

\begin{lemma}\label{lem:egap}
Let $Q$ be a positive radial ground state of \eqref{eq:log}. Then there exists $\rho_*>0$ such that any positive radial ground state $P$ satisfying
\begin{equation*}
\|P-Q\|_{L^2(\mathbb{R}^N)}\leq2\rho_*
\end{equation*}
must coincide with $Q$. Moreover, for every $\rho\in(0,\rho_*)$,
\begin{equation*}
\inf\left\{\mathcal{I}_0(u):u\in\mathcal{S}_M, d_Q(u)=\rho\right\}>c_0.
\end{equation*}
\end{lemma}

\begin{proof}
By Lemma \ref{lem:isol}, there exists $d_*>0$ such that any positive radial ground state $P$ satisfying
$
\|P-Q\|_{L^2(\mathbb{R}^N)}\leq d_*
$
must coincide with $Q$. Choose
$
0<\rho_*<\frac{d_*}{2}.
$
Then the first assertion follows. 

Fix $\rho\in(0,\rho_*)$. Suppose by contradiction that
\begin{equation*}
\inf\left\{\mathcal{I}_0(u):u\in\mathcal{S}_M, d_Q(u)=\rho\right\}=c_0.
\end{equation*}
Then there exists $u_n\in\mathcal{S}_M$ such that $d_Q(u_n)=\rho$ and $\mathcal{I}_0(u_n)\to c_0$.
Since $\mathcal{I}_0(u_n)<+\infty$, we have $u_n\in\mathcal{D}_s$. By Lemma \ref{lem:rearrangement},
\begin{equation*}
u_n^*\in\mathcal{S}_M,\quad \mathcal{I}_0(u_n^*)\leq\mathcal{I}_0(u_n),
\end{equation*}
and, by the definition of $d_Q$,
\begin{equation*}
\|u_n^*-Q\|_{L^2(\mathbb{R}^N)}=d_Q(u_n)=\rho.
\end{equation*}
By Lemma \ref{lem:fmass},
\begin{equation*}
c_0\leq\mathcal{I}_0(u_n^*)\leq\mathcal{I}_0(u_n),
\end{equation*}
and hence
$
\mathcal{I}_0(u_n^*)\to c_0.
$
Replacing $u_n$ by $u_n^*$, we may assume that $u_n$ is nonnegative, radial and nonincreasing, with
\begin{equation}\label{eq:egap-dist}
\|u_n-Q\|_{L^2(\mathbb{R}^N)}=\rho.
\end{equation}

By Lemma \ref{lem:lcomp}, after passing to a subsequence,
\begin{equation*}
u_n\to U\quad\text{strongly in }L^2(\mathbb{R}^N),
\end{equation*}
where $U$ is a positive radial decreasing ground state of \eqref{eq:log}. Passing to the limit in \eqref{eq:egap-dist}, we obtain
\begin{equation*}
\|U-Q\|_{L^2(\mathbb{R}^N)}=\rho<\rho_*<d_*.
\end{equation*}
Hence $U=Q$ by the choice of $d_*$. This contradicts $\rho>0$.
\end{proof}

\medskip
\section{Uniqueness of positive ground states}\label{sec:power-approximation}

In this section, we prove the uniqueness of positive ground states by the isolation result from Section \ref{sec:compactness-isolation} and the uniqueness theory for the fractional power equation \cite{Frank-Lenzmann-2013,Frank-Lenzmann-Silvestre-2016}. Near each logarithmic ground state, we construct a local minimizer for a nearby power problem. Its linearized quadratic form has Morse index one, and after rescaling the minimizer becomes a positive solution of the normalized fractional power equation. The uniqueness result for the power equation then excludes the existence of two distinct logarithmic ground states.

Fix a positive radial decreasing ground state $Q$ of \eqref{eq:log}. For
\begin{equation*}
0<\varepsilon<\min\left\{1,\alpha_*(s,N)\right\},
\end{equation*}
define
\begin{equation*}
\mathcal{I}_\varepsilon(u)=\frac{1}{2}A(u)+\frac{1}{2\varepsilon}B(u)-\frac{1}{\varepsilon(2+\varepsilon)}\int_{\mathbb{R}^N}|u|^{2+\varepsilon}\mathrm{d}x,\quad u\in H^s(\mathbb{R}^N).
\end{equation*}
Since $\varepsilon<\alpha_*(s,N)$, we have $2+\varepsilon<2_s^*$, and hence $\mathcal{I}_\varepsilon$ is well defined on $H^s(\mathbb{R}^N)$. On $\mathcal{S}_M$, using $B(u)=M$, we obtain
\begin{equation}\label{eq:iren}
\mathcal{I}_\varepsilon(u)=\frac{1}{2}A(u)+\frac{M}{2(2+\varepsilon)}-\frac{1}{2+\varepsilon}\int_{\mathbb{R}^N}u^2\frac{|u|^\varepsilon-1}{\varepsilon}\mathrm{d}x.
\end{equation}

\subsection{Convergence of the power functionals}

We first record the pointwise estimates used to pass to the logarithmic nonlinearity.

\begin{lemma}\label{lem:power-log}
For $\varepsilon>0$ and $0<t\leq1$,
\begin{equation}\label{eq:ell-small}
\log t\leq\frac{t^\varepsilon-1}{\varepsilon}\leq0,\quad 0\leq-t^2\frac{t^\varepsilon-1}{\varepsilon}\leq-t^2\log t.
\end{equation}
For every $\delta>0$, there exists $C_\delta>0$ such that, for $t\geq1$ and $0<\varepsilon<\frac{\delta}{2}$,
\begin{equation}\label{eq:ellg}
0\leq t^2\frac{t^\varepsilon-1}{\varepsilon}\leq C_\delta t^{2+\delta}.
\end{equation} 
\end{lemma}

\begin{proof}
For $0<t\leq1$, the inequality $\mathrm{e}^x\geq1+x$ gives
\begin{equation*}
\log t\leq\frac{\mathrm{e}^{\varepsilon\log t}-1}{\varepsilon}=\frac{t^\varepsilon-1}{\varepsilon}\leq0,
\end{equation*}
where the last inequality follows from $t^\varepsilon\leq1$. Multiplying by $-t^2$ yields \eqref{eq:ell-small}.

For $t\geq1$, the fundamental theorem of calculus gives
\begin{equation}\label{eq:ell-int}
\frac{t^\varepsilon-1}{\varepsilon}=\log t\int_0^1t^{\theta\varepsilon}\mathrm{d}\theta.
\end{equation}
If $0<\varepsilon<\frac{\delta}{2}$, then
\begin{equation*}
0\leq t^2\frac{t^\varepsilon-1}{\varepsilon}\leq t^{2+\frac{\delta}{2}}\log t\leq C_\delta t^{2+\delta},
\end{equation*}
because $t^{-\frac{\delta}{2}}\log t$ is bounded on $[1,+\infty)$. This proves \eqref{eq:ellg}. 
\end{proof}

The following lemma gives the compactness and lower semicontinuity needed for the power approximation.

\begin{lemma}\label{lem:plimit}
Let $0<\varepsilon_n<\min\left\{1,\alpha_*(s,N)\right\}$ with $\varepsilon_n\to0$, and let $u_n\in\mathcal{S}_M$ satisfy
\begin{equation*}
\sup_n\mathcal{I}_{\varepsilon_n}(u_n)<+\infty.
\end{equation*}
After replacing $u_n$ by $u_n^*$ and passing to a subsequence, there exists $u\in\mathcal{S}_M$ such that
\begin{equation*}
u_n\to u\quad\text{strongly in }L^2(\mathbb{R}^N)
\end{equation*}
and
\begin{equation}\label{eq:ldir}
\mathcal{I}_0(u)\leq\liminf_{n\to\infty}\mathcal{I}_{\varepsilon_n}(u_n).
\end{equation}
Moreover, if $v\in\mathcal{D}_s\cap L^{2+\delta}(\mathbb{R}^N)$ for some $\delta>0$, then
\begin{equation}\label{eq:fconv}
\mathcal{I}_\varepsilon(v)\to\mathcal{I}_0(v)\quad\text{as }\varepsilon\to0.
\end{equation}
\end{lemma}

\begin{proof}
For $\varepsilon>0$, set
\begin{equation*}
F_\varepsilon(t)=t^2\frac{t^\varepsilon-1}{\varepsilon}\quad\text{for }t>0,\quad F_\varepsilon(0)=0,
\end{equation*}
and let $F_0(t)=t^2\log t$ for $t>0$, with $F_0(0)=0$.

Choose $\delta_0>0$ such that $\delta_0<\frac{4s}{N}$ and $2+\delta_0<2_s^*$, and set $\theta_0=\frac{N\delta_0}{4s}<1$. For all sufficiently large $n$, $\varepsilon_n<\frac{\delta_0}{2}$. If $v\in\mathcal{S}_M$, then \eqref{eq:ell-small}, \eqref{eq:ellg}, and Lemma \ref{lem:fractional-gn} give
\begin{equation*}
\int_{\mathbb{R}^N}F_{\varepsilon_n}(|v|)\mathrm{d}x\leq C_{\delta_0}\|v\|_{L^{2+\delta_0}(\mathbb{R}^N)}^{2+\delta_0}\leq CA(v)^{\theta_0}.
\end{equation*}
Hence
\begin{equation*}
\mathcal{I}_{\varepsilon_n}(v)\geq\frac{1}{2}A(v)-CA(v)^{\theta_0}.
\end{equation*}
Since $\theta_0<1$, Young's inequality gives 
\begin{equation*}
\mathcal{I}_{\varepsilon_n}(v)\geq\frac{1}{4}A(v)-C.
\end{equation*}
Applying this to $v=u_n$ and using $\sup_n\mathcal{I}_{\varepsilon_n}(u_n)<+\infty$, we obtain
\begin{equation*}
A(u_n)\leq C
\end{equation*}
for all sufficiently large $n$. Since $B(u_n)=M$, the sequence $\{u_n\}$ is bounded in $H^s(\mathbb{R}^N)$.

By Lemma \ref{lem:rearrangement},
\begin{equation*}
u_n^*\in\mathcal{S}_M,\quad \mathcal{I}_{\varepsilon_n}(u_n^*)\leq\mathcal{I}_{\varepsilon_n}(u_n).
\end{equation*}
Replacing $u_n$ by $u_n^*$ and passing to a subsequence, we may assume that $u_n$ is nonnegative, radial, nonincreasing and
\begin{equation*}
u_n\rightharpoonup u\quad\text{weakly in }H^s(\mathbb{R}^N).
\end{equation*}
By Lemma \ref{lem:rcomp},
\begin{equation}\label{eq:psq}
u_n\to u\quad\text{strongly in }L^q(\mathbb{R}^N)\quad\text{for every }2<q<2_s^*.
\end{equation}
After passing to a further subsequence, we may also assume that $u_n\to u$ almost everywhere in $\mathbb{R}^N$.

We prove that $B(u)=M$. Suppose that $B(u)<M$ and set
\begin{equation*}
m_0=\frac{M-B(u)}{2}>0.
\end{equation*}
For every fixed $R>0$, \eqref{eq:psq} and H\"older's inequality give
\begin{equation*}
\int_{B_R}u_n^2\mathrm{d}x\to\int_{B_R}u^2\mathrm{d}x.
\end{equation*}
Since
\begin{equation*}
\int_{B_R}u^2\mathrm{d}x\leq B(u)=M-2m_0,
\end{equation*}
for all sufficiently large $n$,
\begin{equation}\label{eq:pmass}
\int_{|x|>R}u_n^2\mathrm{d}x\geq m_0.
\end{equation}

Since $u_n$ is radial and nonincreasing,
\begin{equation*}
|B_1|r^Nu_n(r)^2\leq\int_{|x|\leq r}u_n^2\mathrm{d}x\leq M,
\end{equation*}
and hence
\begin{equation*}
u_n(r)\leq C_Mr^{-\frac{N}{2}},\quad C_M=\left(\frac{M}{|B_1|}\right)^{\frac{1}{2}}.
\end{equation*}
Fix $R>0$ so large that $b_R:=C_MR^{-\frac{N}{2}}<1$. Since $t\mapsto\frac{t^{\varepsilon_n}-1}{\varepsilon_n}$ is increasing and $\frac{b_R^{\varepsilon_n}-1}{\varepsilon_n}<0$, \eqref{eq:pmass} gives
\begin{equation*}
\int_{|x|>R}F_{\varepsilon_n}(u_n)\mathrm{d}x
\leq\frac{b_R^{\varepsilon_n}-1}{\varepsilon_n}\int_{|x|>R}u_n^2\mathrm{d}x
\leq m_0\frac{b_R^{\varepsilon_n}-1}{\varepsilon_n}.
\end{equation*}
Since $\frac{b_R^{\varepsilon_n}-1}{\varepsilon_n}\to\log b_R<0$,
\begin{equation*}
\limsup_{n\to\infty}\int_{|x|>R}F_{\varepsilon_n}(u_n)\mathrm{d}x\leq\frac{m_0}{2}\log b_R.
\end{equation*}
On the other hand, $[F_{\varepsilon_n}(u_n)]_+=0$ on $\{u_n\leq1\}$. By \eqref{eq:ellg} and Lemma \ref{lem:fractional-gn},
\begin{equation} \label{eq:ppos}
\int_{\mathbb{R}^N}[F_{\varepsilon_n}(u_n)]_+\mathrm{d}x
\leq C_{\delta_0}\int_{\mathbb{R}^N}u_n^{2+\delta_0}\mathrm{d}x
\leq C A(u_n)^{\theta_0}B(u_n)^{1+\frac{\delta_0}{2}-\theta_0}
\leq C,
\end{equation}
where we used $B(u_n)=M$ and the boundedness of $A(u_n)$. Hence
\begin{equation*}
\begin{aligned}
	\int_{\mathbb{R}^N}F_{\varepsilon_n}(u_n)\mathrm{d}x
	&=\int_{B_R}F_{\varepsilon_n}(u_n)\mathrm{d}x+\int_{|x|>R}F_{\varepsilon_n}(u_n)\mathrm{d}x
	\leq C+\int_{|x|>R}F_{\varepsilon_n}(u_n)\mathrm{d}x.
\end{aligned}
\end{equation*}
It follows that
\begin{equation*}
\limsup_{n\to\infty}\int_{\mathbb{R}^N}F_{\varepsilon_n}(u_n)\mathrm{d}x\leq C+\frac{m_0}{2}\log b_R.
\end{equation*}
Letting $R\to+\infty$, we obtain
\begin{equation*}
\int_{\mathbb{R}^N}F_{\varepsilon_n}(u_n)\mathrm{d}x\to-\infty.
\end{equation*}
By \eqref{eq:iren} and $A(u_n)\geq0$,
\begin{equation*}
\mathcal{I}_{\varepsilon_n}(u_n)\geq\frac{M}{2(2+\varepsilon_n)}-\frac{1}{2+\varepsilon_n}\int_{\mathbb{R}^N}F_{\varepsilon_n}(u_n)\mathrm{d}x\to+\infty,
\end{equation*}
This contradicts $\sup_n\mathcal{I}_{\varepsilon_n}(u_n)<+\infty$. Hence
\begin{equation*}
B(u)=M.
\end{equation*}
Since $u_n\rightharpoonup u$ weakly in $L^2(\mathbb{R}^N)$ and $\|u_n\|_{L^2(\mathbb{R}^N)}=\|u\|_{L^2(\mathbb{R}^N)}$, we obtain
\begin{equation*} 
u_n\to u\quad\text{strongly in }L^2(\mathbb{R}^N).
\end{equation*}

We next pass to the nonlinear term. If $t_n\to t>0$, then \eqref{eq:ell-int} gives
\begin{equation*}
F_{\varepsilon_n}(t_n)=t_n^2\log t_n\int_0^1t_n^{\theta\varepsilon_n}\mathrm{d}\theta\to t^2\log t=F_0(t).
\end{equation*}
If $t_n\to0$, then \eqref{eq:ell-small} gives $|F_{\varepsilon_n}(t_n)|\leq-t_n^2\log t_n\to0$. Thus
\begin{equation*}
F_{\varepsilon_n}(u_n)\to F_0(u)\quad\text{almost everywhere in }\mathbb{R}^N.
\end{equation*}

Choose $q$ with $2+\delta_0<q<2_s^*$. By \eqref{eq:ell-small} and \eqref{eq:log-negative},
\begin{equation*}
0\leq[F_{\varepsilon_n}(u_n)]_- \leq [u_n^2\log u_n]_- \leq\frac{1}{2\mathrm{e}},
\end{equation*}
Moreover, \eqref{eq:ellg} gives
\begin{equation*}
[F_{\varepsilon_n}(u_n)]_+\leq C_{\delta_0}u_n^{2+\delta_0}.
\end{equation*}
Since $\{u_n\}$ is bounded in $L^q(\mathbb{R}^N)$ and $q>2+\delta_0$, the sequence $\{u_n^{2+\delta_0}\}$ is uniformly integrable on $B_R$. Hence Vitali's theorem gives
\begin{equation*}
\int_{B_R}F_{\varepsilon_n}(u_n)\mathrm{d}x\to\int_{B_R}u^2\log u\mathrm{d}x.
\end{equation*}
Moreover,
\begin{equation*}
\int_{|x|>R}[F_{\varepsilon_n}(u_n)]_+\mathrm{d}x\leq C_{\delta_0}\int_{|x|>R}u_n^{2+\delta_0}\mathrm{d}x.
\end{equation*}
The strong convergence in $L^{2+\delta_0}(\mathbb{R}^N)$ from \eqref{eq:psq} gives
\begin{equation*}
\lim_{R\to+\infty}\limsup_{n\to\infty}\int_{|x|>R}u_n^{2+\delta_0}\mathrm{d}x=0.
\end{equation*}
Since
\begin{equation*}
\int_{\mathbb{R}^N}F_{\varepsilon_n}(u_n)\mathrm{d}x\leq \int_{B_R}F_{\varepsilon_n}(u_n)\mathrm{d}x+\int_{|x|>R}[F_{\varepsilon_n}(u_n)]_+\mathrm{d}x,
\end{equation*}
letting $n\to\infty$ gives
\begin{equation*}
\limsup_{n\to\infty}\int_{\mathbb{R}^N}F_{\varepsilon_n}(u_n)\mathrm{d}x \leq \int_{B_R}u^2\log u\mathrm{d}x+o_R(1),
\end{equation*}
where $o_R(1)\to0$ as $R\to+\infty$. Moreover,
\begin{equation*}
[u^2\log u]_+\leq C_{\delta_0}u^{2+\delta_0}\in L^1(\mathbb{R}^N).
\end{equation*}
Hence $\int_{\mathbb{R}^N}u^2\log u\mathrm{d}x$ is well defined in $[-\infty,+\infty)$ and
\begin{equation*}
\int_{B_R}u^2\log u\mathrm{d}x\to\int_{\mathbb{R}^N}u^2\log u\mathrm{d}x
\end{equation*}
as $R\to+\infty$. Therefore
\begin{equation}\label{eq:plsup}
\limsup_{n\to\infty}\int_{\mathbb{R}^N}F_{\varepsilon_n}(u_n)\mathrm{d}x
\leq
\int_{\mathbb{R}^N}u^2\log u\mathrm{d}x.
\end{equation}
If $\int_{\{u\leq1\}}u^2|\log u|\mathrm{d}x=+\infty$, then the right-hand side of \eqref{eq:plsup} equals $-\infty$. Hence
\begin{equation*}
\int_{\mathbb{R}^N}F_{\varepsilon_n}(u_n)\mathrm{d}x\to-\infty,
\end{equation*}
which contradicts $\sup_n\mathcal{I}_{\varepsilon_n}(u_n)<+\infty$. Thus $u\in\mathcal{D}_s$.

Denote
\begin{equation*}
J_n=\int_{\mathbb{R}^N}F_{\varepsilon_n}(u_n)\mathrm{d}x.
\end{equation*}
By \eqref{eq:ppos},
\begin{equation*}
J_n\leq\int_{\mathbb{R}^N}[F_{\varepsilon_n}(u_n)]_+\mathrm{d}x\leq C.
\end{equation*}
On the other hand, \eqref{eq:iren} gives
\begin{equation*}
J_n=(2+\varepsilon_n)\left(\frac{1}{2}A(u_n)-\mathcal{I}_{\varepsilon_n}(u_n)\right)+\frac{M}{2}.
\end{equation*}
Since $A(u_n)\geq0$ and $\sup_n\mathcal{I}_{\varepsilon_n}(u_n)<+\infty$, we obtain $J_n\geq-C$. Thus $|J_n|\leq C$ and
\begin{equation*}
\frac{J_n}{2+\varepsilon_n}=\frac{J_n}{2}+o(1).
\end{equation*}
Consequently, \eqref{eq:iren} can be written as
\begin{equation*}
\mathcal{I}_{\varepsilon_n}(u_n)=\frac{1}{2}A(u_n)+\frac{M}{4}-\frac{1}{2}J_n+o(1).
\end{equation*}
Using the weak lower semicontinuity of $A$ and \eqref{eq:plsup}, we obtain
\begin{equation*}
\begin{aligned}
	\liminf_{n\to\infty}\mathcal{I}_{\varepsilon_n}(u_n)
	&\geq\frac{1}{2}A(u)+\frac{M}{4}-\frac{1}{2}\limsup_{n\to\infty}J_n\\
	&\geq\frac{1}{2}A(u)+\frac{M}{4}-\frac{1}{2}\int_{\mathbb{R}^N}u^2\log u\mathrm{d}x
	=\mathcal{I}_0(u).
\end{aligned}
\end{equation*}
This proves \eqref{eq:ldir}.

Finally, let $v\in\mathcal{D}_s\cap L^{2+\delta}(\mathbb{R}^N)$ for some $\delta>0$. For $0<\varepsilon<\frac{\delta}{2}$, \eqref{eq:ell-small} and \eqref{eq:ellg} give
\begin{equation*}
|F_\varepsilon(|v|)|\leq-v^2\log|v|\quad\text{on }\{|v|\leq1\},\quad 0\leq F_\varepsilon(|v|)\leq C_\delta|v|^{2+\delta}\quad\text{on }\{|v|>1\}.
\end{equation*}
Since $v\in\mathcal{D}_s$, the function $-v^2\log|v|$ is integrable on $\{|v|\leq1\}$, while $v\in L^{2+\delta}(\mathbb{R}^N)$ gives $|v|^{2+\delta}\in L^1(\mathbb{R}^N)$. Therefore, using $F_\varepsilon(|v|)\to v^2\log|v|$ pointwise, the dominated convergence theorem yields
\begin{equation*}
\int_{\mathbb{R}^N}F_\varepsilon(|v|)\mathrm{d}x\to\int_{\mathbb{R}^N}v^2\log|v|\mathrm{d}x.
\end{equation*}
Moreover, by the definition of $F_\varepsilon$,
\begin{equation*}
\mathcal{I}_\varepsilon(v)=\frac{1}{2}A(v)+\frac{B(v)}{2(2+\varepsilon)}-\frac{1}{2+\varepsilon}\int_{\mathbb{R}^N}F_\varepsilon(|v|)\mathrm{d}x.
\end{equation*}
Hence
\begin{equation*}
\mathcal{I}_\varepsilon(v)\to\frac{1}{2}A(v)+\frac{1}{4}B(v)-\frac{1}{2}\int_{\mathbb{R}^N}v^2\log|v|\mathrm{d}x=\mathcal{I}_0(v),
\end{equation*}
which proves \eqref{eq:fconv}.
\end{proof}

\subsection{A power local minimizer near each logarithmic ground state}

We next use the energy gap from Lemma \ref{lem:egap} to construct a power local minimizer near a fixed logarithmic ground state $Q$ of \eqref{eq:log}.

Fix $0<\rho<\min\{\rho_*,\sqrt{2M}\}$ and set
\begin{equation*}
\eta_\rho=\min\left\{1,\inf\left\{\mathcal{I}_0(u):u\in\mathcal{S}_M,\ d_Q(u)=\rho\right\}-c_0\right\}>0,
\end{equation*}
where the positivity follows from Lemma \ref{lem:egap}. Define
\begin{equation*}
\mathcal{B}_\rho(Q)=\left\{u\in\mathcal{S}_M:d_Q(u)\leq\rho\right\}.
\end{equation*}
By the Sobolev embedding, $Q\in L^{2+\delta}(\mathbb{R}^N)$ for some $\delta>0$. Since $Q\in\mathcal{D}_s$, Lemma \ref{lem:plimit} gives
\begin{equation}\label{eq:iq}
\mathcal{I}_\varepsilon(Q)\to\mathcal{I}_0(Q)=c_0
\end{equation}
as $\varepsilon\to0$.

We have the following energy gap for $\mathcal{I}_\varepsilon$ on $\partial\mathcal{B}_\rho(Q)$.

\begin{lemma}\label{lem:pgap}
For all sufficiently small $\varepsilon>0$,
\begin{equation*}
\inf\left\{\mathcal{I}_\varepsilon(u):u\in\mathcal{S}_M, d_Q(u)=\rho\right\}\geq c_0+\frac{\eta_\rho}{2}.
\end{equation*}
\end{lemma}

\begin{proof}
Suppose that there exist $\varepsilon_n\to0$ and $u_n\in\mathcal{S}_M$ such that
\begin{equation*}
d_Q(u_n)=\rho,\quad \mathcal{I}_{\varepsilon_n}(u_n)<c_0+\frac{\eta_\rho}{2}.
\end{equation*}
By Lemma \ref{lem:rearrangement}, replacing $u_n$ by $u_n^*$ if necessary, we may assume that
\begin{equation*}
u_n\geq0,\quad u_n=u_n^*,\quad d_Q(u_n)=\rho.
\end{equation*}
Lemma \ref{lem:plimit} gives, after passing to a subsequence, some $u\in\mathcal{S}_M$ such that
\begin{equation*}
u_n\to u\quad\text{strongly in }L^2(\mathbb{R}^N),\quad \mathcal{I}_0(u)\leq c_0+\frac{\eta_\rho}{2}.
\end{equation*}
Passing to a further subsequence, $u_n\to u$ almost everywhere. Hence $u$ is nonnegative, radial and nonincreasing. Therefore
\begin{equation*}
d_Q(u)=\|u-Q\|_{L^2(\mathbb{R}^N)}=\lim_{n\to\infty}\|u_n-Q\|_{L^2(\mathbb{R}^N)}=\rho.
\end{equation*}
This contradicts the definition of $\eta_\rho$.
\end{proof}

With the boundary energy gap established, we now minimize $\mathcal{I}_\varepsilon$ in a small neighborhood of the fixed logarithmic ground state $Q$.  

\begin{lemma}\label{lem:patt}
There exists $\varepsilon_1>0$ such that, for every $0<\varepsilon<\varepsilon_1$, the value
\begin{equation*}
m_\varepsilon=\inf_{u\in\mathcal{B}_\rho(Q)}\mathcal{I}_\varepsilon(u)
\end{equation*}
is attained by some nonnegative, radial and nonincreasing function $Q_{\varepsilon}\in\mathcal{B}_\rho(Q)$.
\end{lemma}

\begin{proof}
The estimate established in the proof of Lemma \ref{lem:plimit}, together with Young's inequality, gives $\bar{\varepsilon}<\min\{1,\alpha_*(s,N)\}$ and $C>0$ such that
\begin{equation}\label{eq:pcoer}
\mathcal{I}_\varepsilon(v)\geq\frac{1}{4}A(v)-C\quad\text{for every }v\in\mathcal{S}_M
\end{equation}
if $0<\varepsilon<\bar{\varepsilon}$. By \eqref{eq:iq}, reducing $\bar{\varepsilon}$ if necessary, we may also assume
\begin{equation*}
\mathcal{I}_\varepsilon(Q)\leq c_0+1\quad\text{for }0<\varepsilon<\bar{\varepsilon}.
\end{equation*}
It follows from \eqref{eq:pcoer} that whenever $v\in\mathcal{S}_M$ satisfies $\mathcal{I}_\varepsilon(v)\leq\mathcal{I}_\varepsilon(Q)+1$,
\begin{equation*}
A(v)\leq4(c_0+C+2)=:C_A.
\end{equation*} 
Set $\kappa=M-\frac{\rho^2}{2}>0$. Choose $R>0$ such that
\begin{equation*}
\|Q\|_{L^2(\mathbb{R}^N\setminus B_R)}\leq\frac{\kappa}{2\sqrt{M}}.
\end{equation*}
There exists $c_*>0$ such that
\begin{equation*}
|B_R|^{-\frac{\alpha}{2}}\left(\frac{\kappa^2}{4M}\right)^{1+\frac{\alpha}{2}}\geq c_*\quad\text{for every }0\leq\alpha\leq1.
\end{equation*}
We now choose $0<\varepsilon_1<\bar{\varepsilon}$ so that $\varepsilon_1C_A<c_*$.

Fix $0<\varepsilon<\varepsilon_1$. By \eqref{eq:pcoer}, $m_\varepsilon>-\infty$, while $Q\in\mathcal{B}_\rho(Q)$ gives $m_\varepsilon\leq\mathcal{I}_\varepsilon(Q)$. Let $\{u_n\}\subset\mathcal{B}_\rho(Q)$ be a minimizing sequence such that
\begin{equation*}
\mathcal{I}_\varepsilon(u_n)\leq\mathcal{I}_\varepsilon(Q)+1.
\end{equation*}
By Lemma \ref{lem:rearrangement}, replacing $u_n$ by $u_n^*$ if necessary, we may assume
\begin{equation*}
u_n\geq0,\quad u_n=u_n^*,\quad d_Q(u_n)\leq\rho.
\end{equation*}
By the choice of $C_A$, we have $A(u_n)\leq C_A$. Since $B(u_n)=M$, the sequence $\{u_n\}$ is bounded in $H^s(\mathbb{R}^N)$. Passing to a subsequence,
\begin{equation*}
u_n\rightharpoonup u\quad\text{weakly in }H^s(\mathbb{R}^N).
\end{equation*}
Since $\varepsilon<\alpha_*(s,N)$, we have $2+\varepsilon<2_s^*$. Lemma \ref{lem:rcomp} therefore yields
\begin{equation}\label{eq:felp}
u_n\to u\quad\text{strongly in }L^{2+\varepsilon}(\mathbb{R}^N).
\end{equation}
Passing to a further subsequence, $u_n\to u$ almost everywhere. Thus $u$ is nonnegative, radial and nonincreasing. 
Since $u_n=u_n^*$ and $B(u_n)=B(Q)=M$,
\begin{equation*}
d_Q(u_n)^2=\|u_n-Q\|_{L^2(\mathbb{R}^N)}^2=2M-2\int_{\mathbb{R}^N}u_nQ\mathrm{d}x.
\end{equation*}
The bound $d_Q(u_n)\leq\rho$ gives
\begin{equation*}
\int_{\mathbb{R}^N}u_nQ\mathrm{d}x\geq M-\frac{\rho^2}{2}=\kappa.
\end{equation*}
Since $u_n\rightharpoonup u$ weakly in $L^2(\mathbb{R}^N)$,
\begin{equation}\label{eq:ovlp}
\int_{\mathbb{R}^N}uQ\mathrm{d}x=\lim_{n\to\infty}\int_{\mathbb{R}^N}u_nQ\mathrm{d}x\geq\kappa.
\end{equation}
In particular, $u\neq0$. By weak lower semicontinuity of $B$ and $A$,
\begin{equation*}
B(u)\leq M,\quad A(u)\leq C_A.
\end{equation*} 
We claim that $B(u)=M$. 

Suppose that $B(u)<M$ and set
$
t=\left(\frac{M}{B(u)}\right)^{\frac{1}{2}}>1.
$
Then $tu\in\mathcal{S}_M$ and $(tu)^*=tu$. Moreover,
\begin{equation*}
2M-2\int_{\mathbb{R}^N}uQ\mathrm{d}x
=\lim_{n\to\infty}d_Q(u_n)^2\leq\rho^2.
\end{equation*}
Since \eqref{eq:ovlp} gives $\int_{\mathbb{R}^N}uQ\mathrm{d}x>0$ and $t>1$,
\begin{equation*}
d_Q(tu)^2=2M-2t\int_{\mathbb{R}^N}uQ\mathrm{d}x\leq2M-2\int_{\mathbb{R}^N}uQ\mathrm{d}x\leq\rho^2.
\end{equation*}
Thus $tu\in\mathcal{B}_\rho(Q)$. By \eqref{eq:ovlp}, the choice of $R$ and $B(u)\leq M$,
\begin{equation*}
\begin{aligned}
	\int_{B_R}uQ\mathrm{d}x
	&\geq\kappa-\|u\|_{L^2(\mathbb{R}^N\setminus B_R)}\|Q\|_{L^2(\mathbb{R}^N\setminus B_R)}
	\geq\frac{\kappa}{2}.
\end{aligned}
\end{equation*}
The Cauchy-Schwarz inequality then gives
\begin{equation*}
\int_{B_R}u^2\mathrm{d}x\geq\frac{\left(\int_{B_R}uQ\mathrm{d}x\right)^2}{\int_{B_R}Q^2\mathrm{d}x}\geq\frac{\kappa^2}{4M}.
\end{equation*}
Hence H\"older's inequality and the definition of $c_*$ yield
\begin{equation*}
\int_{\mathbb{R}^N}u^{2+\varepsilon}\mathrm{d}x
\geq |B_R|^{-\frac{\varepsilon}{2}}\left(\frac{\kappa^2}{4M}\right)^{1+\frac{\varepsilon}{2}}
\geq c_*.
\end{equation*}
Since $\varepsilon<\varepsilon_1$, $A(u)\leq C_A$ and $\varepsilon_1C_A<c_*$, we obtain
\begin{equation}\label{eq:pgt}
\int_{\mathbb{R}^N}u^{2+\varepsilon}\mathrm{d}x>\varepsilon A(u).
\end{equation} 
On the other hand, since $\{u_n\}$ is a minimizing sequence, weak lower semicontinuity and \eqref{eq:felp} give
\begin{equation}\label{eq:pliminf}
m_\varepsilon\geq\frac{1}{2}A(u)+\frac{M}{2\varepsilon}-\frac{1}{\varepsilon(2+\varepsilon)}\int_{\mathbb{R}^N}u^{2+\varepsilon}\mathrm{d}x.
\end{equation}
For $r\geq1$, set
\begin{equation*}
g(r)=\frac{r^2}{2}A(u)-\frac{r^{2+\varepsilon}}{\varepsilon(2+\varepsilon)}\int_{\mathbb{R}^N}u^{2+\varepsilon}\mathrm{d}x.
\end{equation*}
By \eqref{eq:pgt},
\begin{equation*}
\begin{aligned}
	g'(r)
	&=rA(u)-\frac{r^{1+\varepsilon}}{\varepsilon}\int_{\mathbb{R}^N}u^{2+\varepsilon}\mathrm{d}x
	\leq r\left(A(u)-\frac{1}{\varepsilon}\int_{\mathbb{R}^N}u^{2+\varepsilon}\mathrm{d}x\right)<0
\end{aligned}
\end{equation*}
for every $r\geq1$. Since $t>1$,
\begin{equation*}
\mathcal{I}_\varepsilon(tu)=\frac{M}{2\varepsilon}+g(t)
<\frac{M}{2\varepsilon}+g(1)\leq m_\varepsilon,
\end{equation*}
where the last inequality follows from \eqref{eq:pliminf}. This contradicts $tu\in\mathcal{B}_\rho(Q)$. Hence $B(u)=M$.

Since $u_n\rightharpoonup u$ weakly in $L^2(\mathbb{R}^N)$ and $B(u_n)=B(u)=M$, we conclude that
\begin{equation*}
u_n\to u\quad\text{strongly in }L^2(\mathbb{R}^N).
\end{equation*}
Therefore
\begin{equation*}
d_Q(u)=\|u-Q\|_{L^2(\mathbb{R}^N)}
=\lim_{n\to\infty}\|u_n-Q\|_{L^2(\mathbb{R}^N)}
\leq\rho,
\end{equation*}
so $u\in\mathcal{B}_\rho(Q)$. Finally, weak lower semicontinuity and \eqref{eq:felp} yield
\begin{equation*}
m_\varepsilon\leq\mathcal{I}_\varepsilon(u)\leq\liminf_{n\to\infty}\mathcal{I}_\varepsilon(u_n)=m_\varepsilon.
\end{equation*}
Thus $u$ attains $m_\varepsilon$. We denote this minimizer by $Q_{\varepsilon}$. By construction, $Q_{\varepsilon}$ is nonnegative, radial and nonincreasing.
\end{proof}

We next show that the minimizer obtained in Lemma \ref{lem:patt} lies in the interior of $\mathcal B_\rho(Q)$ and converges to $Q$ as $\varepsilon\to0$.

\begin{proposition}\label{prop:plmin}
There exists $\varepsilon_0>0$ such that, for every $0<\varepsilon<\varepsilon_0$, a minimizer $Q_{\varepsilon}$ chosen as in Lemma \ref{lem:patt} satisfies
$
d_Q(Q_{\varepsilon})<\rho.
$
In particular, $Q_{\varepsilon}$ is nonnegative, radial, nonincreasing and is a local minimizer of $\mathcal{I}_\varepsilon$ on $\mathcal{S}_M$. Moreover,
\begin{equation}\label{eq:qconv}
Q_{\varepsilon}\to Q\quad\text{strongly in }L^2(\mathbb{R}^N)
\end{equation}
as $\varepsilon\to0$.
\end{proposition}

\begin{proof}
Choose $\varepsilon_0>0$ sufficiently small so that Lemmas \ref{lem:pgap} and \ref{lem:patt} apply for every $0<\varepsilon<\varepsilon_0$. By \eqref{eq:iq}, reducing $\varepsilon_0$ if necessary, we may also assume
\begin{equation*}
\mathcal{I}_\varepsilon(Q)<c_0+\frac{\eta_\rho}{2}
\end{equation*}
for every $0<\varepsilon<\varepsilon_0$. Since $Q_{\varepsilon}$ minimizes $\mathcal{I}_\varepsilon$ on $\mathcal{B}_\rho(Q)$ and $Q\in\mathcal{B}_\rho(Q)$,
\begin{equation*}
\mathcal{I}_\varepsilon(Q_{\varepsilon})
\leq\mathcal{I}_\varepsilon(Q)
<c_0+\frac{\eta_\rho}{2}.
\end{equation*}
The boundary estimate in Lemma \ref{lem:pgap} therefore implies
$
d_Q(Q_{\varepsilon})<\rho.
$

We next show that $Q_{\varepsilon}$ is a local minimizer on $\mathcal{S}_M$. Set
\begin{equation*}
\delta_{\varepsilon}=\rho-d_Q(Q_{\varepsilon})>0.
\end{equation*}
Let $v\in\mathcal{S}_M$ satisfy
$
\|v-Q_{\varepsilon}\|_{L^2(\mathbb{R}^N)}<\delta_{\varepsilon}.
$
Since $Q_{\varepsilon}=Q_{\varepsilon}^*$, \eqref{eq:rcontr} and the triangle inequality give
\begin{equation*}
\begin{aligned}
	d_Q(v)
	&=\|v^*-Q\|_{L^2(\mathbb{R}^N)}\\
	&\leq\|v^*-Q_{\varepsilon}\|_{L^2(\mathbb{R}^N)}
	+\|Q_{\varepsilon}-Q\|_{L^2(\mathbb{R}^N)}\\
	&\leq\|v-Q_{\varepsilon}\|_{L^2(\mathbb{R}^N)}
	+d_Q(Q_{\varepsilon})
	<\rho.
\end{aligned}
\end{equation*}
Thus $v\in\mathcal{B}_\rho(Q)$, and hence
\begin{equation*}
\mathcal{I}_\varepsilon(Q_{\varepsilon})
\leq\mathcal{I}_\varepsilon(v).
\end{equation*}
Therefore, $Q_{\varepsilon}$ is an $L^2(\mathbb{R}^N)$-local minimizer of $\mathcal{I}_\varepsilon$ on $\mathcal{S}_M$. 

Let $\varepsilon_n\to0$ with $0<\varepsilon_n<\varepsilon_0$. By minimality and \eqref{eq:iq},
\begin{equation*}
\mathcal{I}_{\varepsilon_n}(Q_{\varepsilon_n})
\leq\mathcal{I}_{\varepsilon_n}(Q)\to c_0.
\end{equation*}
Lemma \ref{lem:plimit} gives, after passing to a subsequence, some $U\in\mathcal{S}_M$ such that
\begin{equation*}
Q_{\varepsilon_n}\to U\quad\text{strongly in }L^2(\mathbb{R}^N),\quad
\mathcal{I}_0(U)\leq c_0.
\end{equation*}
Passing to a further subsequence, we may also assume that $Q_{\varepsilon_n}\to U$ almost everywhere. Since each $Q_{\varepsilon_n}$ is nonnegative, radial and nonincreasing, the same properties hold for $U$. Moreover,
\begin{equation*}
d_Q(U)=\|U-Q\|_{L^2(\mathbb{R}^N)}=\lim_{n\to\infty}\|Q_{\varepsilon_n}-Q\|_{L^2(\mathbb{R}^N)}\leq\rho.
\end{equation*} 
Since $\mathcal{I}_0(U)\leq c_0<\infty$, we have $U\in\mathcal{D}_s$. Hence Lemma \ref{lem:fmass} yields
$
A(U)\geq C_0(U).
$
Since $B(U)=M$ and $c_0=\frac{M}{4}$,
\begin{equation*}
0\geq\mathcal{I}_0(U)-c_0=\frac{1}{2}\left(A(U)-C_0(U)\right)\geq0.
\end{equation*}
Therefore
\begin{equation*}
A(U)=C_0(U),\quad \mathcal{I}_0(U)=c_0.
\end{equation*}
Thus $U\in\mathcal{N}_0$ and $U$ is a radial ground state of \eqref{eq:log}. The positivity argument in \cite[Proof of Theorem 1.1-(i) equation (2.19)]{Li-Peng-Shuai-2022} gives $U>0$ in $\mathbb{R}^N$. Since
$
d_Q(U)\leq\rho<\rho_*,
$
Lemma \ref{lem:egap} gives $U=Q$.

We have proved that every sequence $\varepsilon_n\to0$ has a subsequence such that
\begin{equation*}
Q_{\varepsilon_n}\to Q\quad\text{strongly in }L^2(\mathbb{R}^N).
\end{equation*}
If \eqref{eq:qconv} were false, there would exist $\sigma>0$ and a sequence $\varepsilon_n\to0$ such that
\begin{equation*}
\|Q_{\varepsilon_n}-Q\|_{L^2(\mathbb{R}^N)}\geq\sigma.
\end{equation*}
The preceding argument would produce a subsequence converging strongly in $L^2(\mathbb{R}^N)$ to $Q$, which is a contradiction. Therefore, \eqref{eq:qconv} holds.
\end{proof}

Since $2+\varepsilon<2_s^*$, the functional $\mathcal{I}_\varepsilon$ is of class $C^2$ on $H^s(\mathbb{R}^N)$. 
By Proposition \ref{prop:plmin}, $Q_{\varepsilon}$ is a local minimizer of $\mathcal{I}_\varepsilon$ on $\mathcal{S}_M$. Hence there exists a Lagrange multiplier $\lambda_{\varepsilon}\in\mathbb{R}$ such that 
\begin{equation*}
(-\Delta)^sQ_{\varepsilon}+\frac{1}{\varepsilon}Q_{\varepsilon}-\frac{1}{\varepsilon}Q_{\varepsilon}^{1+\varepsilon}=\lambda_{\varepsilon}Q_{\varepsilon}.
\end{equation*}
Set $\omega_{\varepsilon}=\frac{1}{\varepsilon}-\lambda_{\varepsilon}$. Then
\begin{equation}\label{eq:peq}
(-\Delta)^sQ_{\varepsilon}+\omega_{\varepsilon}Q_{\varepsilon}=\frac{1}{\varepsilon}Q_{\varepsilon}^{1+\varepsilon}.
\end{equation}

We now show that $\omega_\varepsilon>0$, that $Q_\varepsilon$ is positive and strictly decreasing, and that its linearized quadratic form has Morse index one.

\begin{lemma}\label{lem:pmorse}
There exists $\varepsilon_2>0$ such that, for every $0<\varepsilon<\varepsilon_2$,
\begin{equation*}
\omega_{\varepsilon}>0,\quad Q_{\varepsilon}>0\quad\text{in }\mathbb{R}^N,\quad Q_{\varepsilon}'(r)<0\quad\text{for }r>0.
\end{equation*}
Moreover, the quadratic form
\begin{equation*}
\mathfrak{q}_{\varepsilon}(h,h)=A(h)+\omega_{\varepsilon}B(h)-\frac{1+\varepsilon}{\varepsilon}\int_{\mathbb{R}^N}Q_{\varepsilon}^{\varepsilon}h^2\mathrm{d}x
\end{equation*}
is well defined on $H^s(\mathbb{R}^N)$ and has Morse index one.
\end{lemma}

\begin{proof}
For $h\in H^s(\mathbb{R}^N)$, H\"older's inequality and the Sobolev embedding give
\begin{equation*}
\int_{\mathbb{R}^N}Q_{\varepsilon}^{\varepsilon}h^2\mathrm{d}x
\leq\|Q_{\varepsilon}\|_{L^{2+\varepsilon}(\mathbb{R}^N)}^\varepsilon\|h\|_{L^{2+\varepsilon}(\mathbb{R}^N)}^2<\infty.
\end{equation*}
Thus $\mathfrak{q}_{\varepsilon}$ is well defined on $H^s(\mathbb{R}^N)$. 
Write
\begin{equation*}
u=Q_{\varepsilon},\quad \omega=\omega_{\varepsilon},\quad P=\int_{\mathbb{R}^N}u^{2+\varepsilon}\mathrm{d}x.
\end{equation*}
Testing \eqref{eq:peq} with $u$ gives
\begin{equation}\label{eq:ptest}
A(u)+\omega M=\frac{1}{\varepsilon}P.
\end{equation}

For $\tau>0$, set $u_\tau(x)=\tau^{\frac{N}{2}}u(\tau x)$. Then $B(u_\tau)=M$ and
\begin{equation*}
A(u_\tau)=\tau^{2s}A(u),\quad \int_{\mathbb{R}^N}u_\tau^{2+\varepsilon}\mathrm{d}x=\tau^{\frac{N\varepsilon}{2}}P.
\end{equation*}
Since $u$ is a local minimizer of $\mathcal{I}_\varepsilon$ on $\mathcal{S}_M$, the function $\tau\mapsto\mathcal{I}_\varepsilon(u_\tau)$ has a local minimum at $\tau=1$. Since
\begin{equation*}
\mathcal{I}_\varepsilon(u_\tau)=\frac{\tau^{2s}}{2}A(u)+\frac{M}{2\varepsilon}-\frac{\tau^{\frac{N\varepsilon}{2}}}{\varepsilon(2+\varepsilon)}P,
\end{equation*}
differentiating at $\tau=1$ gives 
\begin{equation}\label{eq:pscale}
P=\frac{2s(2+\varepsilon)}{N}A(u).
\end{equation}
Combining \eqref{eq:ptest} and \eqref{eq:pscale}, we obtain
\begin{equation}\label{eq:pomega}
N\varepsilon\omega M=\left(4s-\varepsilon(N-2s)\right)A(u).
\end{equation}
If $N\leq2s$, then $4s-\varepsilon(N-2s)>0$. If $N>2s$, the same inequality follows from $\varepsilon<\alpha_*(s,N)=\frac{4s}{N-2s}$. Since $B(u)=M>0$, we have $u\not\equiv0$ and therefore $P>0$. By \eqref{eq:pscale}, $A(u)>0$. Hence \eqref{eq:pomega} gives $\omega>0$.

We next prove that the Morse index is one. Let $h\in H^s(\mathbb{R}^N)$ satisfy
\begin{equation}\label{eq:phtan}
\int_{\mathbb{R}^N}uh\mathrm{d}x=0.
\end{equation}
Define
\begin{equation*}
\alpha(t)=\frac{\sqrt{M}}{\|u+th\|_{L^2(\mathbb{R}^N)}},\quad u_t=\alpha(t)(u+th).
\end{equation*}
Then $B(u_t)=M$. By \eqref{eq:phtan},
\begin{equation*}
\|u+th\|_{L^2(\mathbb{R}^N)}^2=M+t^2\|h\|_{L^2(\mathbb{R}^N)}^2,
\end{equation*}
so
\begin{equation*}
\alpha(0)=1,\quad \alpha'(0)=0,\quad \alpha''(0)=-\frac{\|h\|_{L^2(\mathbb{R}^N)}^2}{M},
\end{equation*}
and
\begin{equation}\label{eq:ppath}
u_0=u,\quad u_0'=h,\quad u_0''=-\frac{\|h\|_{L^2(\mathbb{R}^N)}^2}{M}u.
\end{equation} 

Since $\mathcal{I}_\varepsilon$ is of class $C^2$ on $H^s(\mathbb{R}^N)$, the chain rule and \eqref{eq:ppath} give
\begin{equation}\label{eq:psec}
\begin{aligned}
	\left.\frac{\mathrm{d}^2}{\mathrm{d}t^2}\mathcal{I}_\varepsilon(u_t)\right|_{t=0}
	&=A(h)+\frac{1}{\varepsilon}B(h)-\frac{1+\varepsilon}{\varepsilon}\int_{\mathbb{R}^N}u^\varepsilon h^2\mathrm{d}x-\lambda_{\varepsilon}B(h)\\
	&=A(h)+\omega B(h)-\frac{1+\varepsilon}{\varepsilon}\int_{\mathbb{R}^N}u^\varepsilon h^2\mathrm{d}x\\
	&=\mathfrak{q}_{\varepsilon}(h,h).
\end{aligned}
\end{equation} 
Since $u$ is a local minimizer of $\mathcal{I}_\varepsilon$ on $\mathcal{S}_M$ and $u_t\in\mathcal{S}_M$, \eqref{eq:psec} gives
\begin{equation}\label{eq:ptan}
\mathfrak{q}_{\varepsilon}(h,h)\geq0\quad\text{for every }h\in H^s(\mathbb{R}^N)\text{ satisfying }h\perp u.
\end{equation} 
On the other hand, \eqref{eq:ptest} gives
\begin{equation*}
\mathfrak{q}_{\varepsilon}(u,u)=A(u)+\omega M-\frac{1+\varepsilon}{\varepsilon}P=-P<0.
\end{equation*}
Hence $\operatorname{span}\{u\}$ is a negative subspace, and therefore
\begin{equation*}
\operatorname{ind}(\mathfrak{q}_{\varepsilon})\geq1.
\end{equation*} 
If $E\subset H^s(\mathbb{R}^N)$ is a negative subspace with $\dim E\geq2$, then
\begin{equation*}
\dim(E\cap u^\perp)\geq\dim E-1\geq1.
\end{equation*}
Hence there exists $0\neq h\in E\cap u^\perp$. Since $E$ is negative, $\mathfrak{q}_{\varepsilon}(h,h)<0$, contradicting \eqref{eq:ptan}. Therefore,
\begin{equation*}
\operatorname{ind}(\mathfrak{q}_{\varepsilon})=1.
\end{equation*}

Since $\omega>0$, define
\begin{equation*}
W(x)=(\varepsilon\omega)^{-\frac{1}{\varepsilon}}u\left(\omega^{-\frac{1}{2s}}x\right).
\end{equation*}
By the scaling property of the fractional Laplacian, for every $b>0$,
\begin{equation*}
(-\Delta)^s[u(b \cdot)](x)=b^{2s}\left[(-\Delta)^su\right](bx).
\end{equation*}
Taking $b=\omega^{-\frac{1}{2s}}$, we obtain
\begin{equation*}
(-\Delta)^sW(x)
=(\varepsilon\omega)^{-\frac{1}{\varepsilon}}\omega^{-1}
\left[(-\Delta)^su\right]\left(\omega^{-\frac{1}{2s}}x\right).
\end{equation*}
Using \eqref{eq:peq},
we obtain
\begin{equation*}
\begin{aligned}
	(-\Delta)^sW(x)&=\frac{(\varepsilon\omega)^{-\frac{1}{\varepsilon}}}{\varepsilon\omega}u\left(\omega^{-\frac{1}{2s}}x\right)^{1+\varepsilon}-W(x)
	=W(x)^{1+\varepsilon}-W(x).
\end{aligned}
\end{equation*}
Therefore
\begin{equation*}
(-\Delta)^sW+W=W^{1+\varepsilon}\quad\text{in }\mathbb{R}^N.
\end{equation*}
Moreover, $W$ is nonnegative, nontrivial, radial and nonincreasing. By \cite[Proposition 1.1-(ii)]{Frank-Lenzmann-2013} when $N=1$ and \cite[Proposition 3.1-(ii)]{Frank-Lenzmann-Silvestre-2016} when $N>1$, we have
\begin{equation*}
W>0\quad\text{in }\mathbb{R}^N,\quad W'(r)<0\quad\text{for }r>0.
\end{equation*}
Hence the same conclusions hold for $u=Q_{\varepsilon}$.
\end{proof} 

We recall the uniqueness result for the normalized power equation in the form needed below.

\begin{lemma}\label{lem:puniq}
Let $N\geq1$, $0<s<1$ and $0<\alpha<\alpha_*(s,N)$. Assume that $U\in H^s(\mathbb{R}^N)$ is a positive weak solution of
\begin{equation}\label{eq:pgen}
(-\Delta)^sU+U=U^{1+\alpha}\quad\text{in }\mathbb{R}^N.
\end{equation}
Define
\begin{equation*}
\mathfrak{q}_U(h,h)=A(h)+B(h)-(1+\alpha)\int_{\mathbb{R}^N}U^\alpha h^2\mathrm{d}x,\quad h\in H^s(\mathbb{R}^N).
\end{equation*}
If $\operatorname{ind}(\mathfrak{q}_U)=1$, then $U$ coincides, up to translation, with the unique positive power ground state of \eqref{eq:pgen}. 
\end{lemma}

\begin{proof}
By H\"older's inequality and the Sobolev embedding,
\begin{equation*}
\int_{\mathbb{R}^N}U^\alpha h^2\mathrm{d}x
\leq\|U\|_{L^{2+\alpha}(\mathbb{R}^N)}^\alpha
\|h\|_{L^{2+\alpha}(\mathbb{R}^N)}^2<\infty,
\end{equation*}
so $\mathfrak{q}_U$ is well defined on $H^s(\mathbb{R}^N)$.
Since $U$ is a positive solution of \eqref{eq:pgen}, standard regularity and decay estimates give
$U\in L^\infty(\mathbb{R}^N)$ and $U(x)\to0$ as $|x|\to\infty$,
see \cite[Proposition 3.1]{Frank-Lenzmann-Silvestre-2016} for $N\geq1$.
Hence $U^\alpha$ is a relatively compact perturbation of
$(-\Delta)^s+1$ (see the proof of Lemma \ref{lem:prs-truncated}). By Weyl's theorem,
\begin{equation*}
\sigma_{\mathrm{ess}}(L_{+,U})=[1,+\infty),
\end{equation*}
where
\begin{equation*}
L_{+,U}=(-\Delta)^s+1-(1+\alpha)U^\alpha.
\end{equation*}
Denote by $N_-(L_{+,U})$ the total number of negative eigenvalues of $L_{+,U}$ counted with multiplicity. Then, since $\operatorname{ind}(\mathfrak{q}_U)=1$, recalling the definition of the Morse index in \eqref{eq:morse-index}, we have
\begin{equation*}
N_-(L_{+,U})=\operatorname{ind}(\mathfrak{q}_U)=1.
\end{equation*} 
Thus $U$ is a ground state of \eqref{eq:pgen} in the sense of
\cite[Definition 3.1]{Frank-Lenzmann-Silvestre-2016} for $N\geq1$ and of
\cite{Frank-Lenzmann-2013} for $N=1$.
The uniqueness up to translations then follows from
\cite[Theorem 4]{Frank-Lenzmann-Silvestre-2016} when $N\geq1$.
\end{proof}

Finally, combining Proposition \ref{prop:plmin}, Lemmas \ref{lem:pmorse} and \ref{lem:puniq}, we prove the uniqueness statement in Theorem \ref{thm:main}.

\begin{proof}[Proof of Theorem \ref{thm:main}-(2)]
Assume that $Q_1\neq Q_2$ are positive radial ground states of \eqref{eq:log}, and set
\begin{equation*}
d=\|Q_1-Q_2\|_{L^2(\mathbb{R}^N)}>0.
\end{equation*} 
For $i=1,2$, let $\rho_{*,i}>0$ be given by Lemma \ref{lem:egap} for $Q_i$. Choose
\begin{equation*}
0<\rho_i<\min\left\{\rho_{*,i},\sqrt{2M},\frac{d}{4}\right\},\quad i=1,2.
\end{equation*}
Then the sets
\begin{equation*}
\mathcal{B}_{\rho_i}(Q_i)=\left\{u\in\mathcal{S}_M:d_{Q_i}(u)\leq\rho_i\right\}
\end{equation*}
are disjoint. Indeed, if $u$ belonged to their intersection, then
\begin{equation*}
\|Q_1-Q_2\|_{L^2(\mathbb{R}^N)}\leq\|Q_1-u^*\|_{L^2(\mathbb{R}^N)}+\|u^*-Q_2\|_{L^2(\mathbb{R}^N)}\leq\rho_1+\rho_2,
\end{equation*}
which contradicts the choice of $\rho_1$ and $\rho_2$.

For each $i=1,2$, Proposition \ref{prop:plmin} and Lemma \ref{lem:pmorse}, applied with $(Q,\rho)=(Q_i,\rho_i)$, yield $\bar{\varepsilon}_i>0$ such that their conclusions hold for $0<\varepsilon<\bar{\varepsilon}_i$. Fix
\begin{equation*}
0<\varepsilon<\min\left\{\bar{\varepsilon}_1,\bar{\varepsilon}_2,\alpha_*(s,N),\frac{4s}{N}\right\}.
\end{equation*}
Then, for $i=1,2$, there exists a positive radial local minimizer $Q_{\varepsilon,i}\in\mathcal B_{\rho_i}(Q_i)$ satisfying
\begin{equation*}
(-\Delta)^sQ_{\varepsilon,i}+\omega_{\varepsilon,i}Q_{\varepsilon,i}=\frac{1}{\varepsilon}Q_{\varepsilon,i}^{1+\varepsilon} \quad\text{in }\mathbb{R}^N,
\end{equation*}
with
\begin{equation*}
\|Q_{\varepsilon,i}\|_{L^2(\mathbb{R}^N)}^2=M,\quad \omega_{\varepsilon,i}>0,
\end{equation*}
and the corresponding linearized quadratic form has Morse index one. 

Let $u$ be either $Q_{\varepsilon,1}$ or $Q_{\varepsilon,2}$, and denote the corresponding frequency by $\omega$. Set
\begin{equation*}
a=(\varepsilon\omega)^{\frac{1}{\varepsilon}},\quad b=\omega^{\frac{1}{2s}},\quad W(x)=a^{-1}u(b^{-1}x).
\end{equation*}
Then 
\begin{equation*}
(-\Delta)^sW+W=W^{1+\varepsilon}\quad\text{in }\mathbb{R}^N.
\end{equation*} 
Let
\begin{equation*}
\mathfrak{q}_u(\phi,\phi)=A(\phi)+\omega B(\phi)-\frac{1+\varepsilon}{\varepsilon}\int_{\mathbb{R}^N}u^\varepsilon\phi^2\mathrm{d}x,
\end{equation*}
and
\begin{equation*}
\mathfrak{q}_W(\psi,\psi)=A(\psi)+B(\psi)-(1+\varepsilon)\int_{\mathbb{R}^N}W^\varepsilon\psi^2\mathrm{d}x.
\end{equation*}
For $\psi\in H^s(\mathbb{R}^N)$, set $\phi(x)=\psi(bx)$. Since $b^{2s}=\omega$ and $a^\varepsilon=\varepsilon\omega$, a change of variables gives
\begin{equation}\label{eq:power-scaling}
\mathfrak{q}_u(\phi,\phi)=\omega b^{-N}\mathfrak{q}_W(\psi,\psi).
\end{equation}
Hence the Morse index is preserved under this scaling. More precisely, if $E\subset H^s(\mathbb{R}^N)$ is a finite-dimensional subspace, then
\begin{equation*}
\dim\{\psi(b\cdot):\psi\in E\}=\dim E,
\end{equation*}
and \eqref{eq:power-scaling} shows that negative subspaces correspond to each other. Therefore
\begin{equation*}
\operatorname{ind}(\mathfrak{q}_W)=\operatorname{ind}(\mathfrak{q}_u)=1.
\end{equation*}

By Lemma \ref{lem:puniq}, the normalized functions associated with
$Q_{\varepsilon,1}$ and $Q_{\varepsilon,2}$ are translations of the unique
positive ground state $W_\varepsilon$ of
\begin{equation*}
(-\Delta)^sW+W=W^{1+\varepsilon}\quad\text{in }\mathbb{R}^N.
\end{equation*}
We choose the representative $W_\varepsilon$ to be radial and strictly decreasing.
Hence, for $i=1,2$, there exists $x_{\varepsilon,i}\in\mathbb{R}^N$ such that
\begin{equation*}
Q_{\varepsilon,i}(x)=(\varepsilon\omega_{\varepsilon,i})^{\frac{1}{\varepsilon}}W_\varepsilon\left(\omega_{\varepsilon,i}^{\frac{1}{2s}}(x-x_{\varepsilon,i})\right).
\end{equation*}
Since $Q_{\varepsilon,i}$ is radial about the origin, we have $x_{\varepsilon,i}=0$. 
Taking the $L^2(\mathbb{R}^N)$ norm yields for $i=1,2$ that
\begin{equation}\label{eq:mfreq}
M=\|Q_{\varepsilon,i}\|_{L^2(\mathbb{R}^N)}^2=\varepsilon^{\frac{2}{\varepsilon}}\omega_{\varepsilon,i}^{\frac{2}{\varepsilon}-\frac{N}{2s}}\|W_\varepsilon\|_{L^2(\mathbb{R}^N)}^2.
\end{equation}
Since
$
\frac{2}{\varepsilon}-\frac{N}{2s}>0,
$
formula \eqref{eq:mfreq} gives $\omega_{\varepsilon,1}=\omega_{\varepsilon,2}$. Hence
\begin{equation*}
Q_{\varepsilon,1}=Q_{\varepsilon,2}.
\end{equation*}
This common function belongs to both $\mathcal{B}_{\rho_1}(Q_1)$ and $\mathcal{B}_{\rho_2}(Q_2)$, contradicting their disjointness. Therefore, the positive radial ground state of \eqref{eq:log} is unique. Proposition \ref{prop:known} then gives uniqueness of positive ground states up to translations.
\end{proof}

\medskip
\section{Sharp logarithmic Sobolev inequalities}\label{sec:log-sobolev-applications}

Cotsiolis and Tavoularis \cite[Theorem 2.1]{Cotsiolis-Tavoularis-2005} claimed a sharp fractional logarithmic Sobolev inequality. Chatzakou and Ruzhansky \cite{Chatzakou-Ruzhansky-2024} subsequently pointed out an incompatibility in the exponent choices used in its proof and established a revised inequality with an explicit constant for $0<s<\frac{N}{2}$. In this section, for every $N\geq1$ and $0<s<1$, we use the fixed-mass characterization and Theorem \ref{thm:main} to derive sharp fractional logarithmic Sobolev inequalities on $\mathbb{R}^N$ and characterize all equality cases. 

Let $Q$ be the positive radial ground state of \eqref{eq:log} and set
\begin{equation*}
M=\|Q\|_{L^2(\mathbb{R}^N)}^2.
\end{equation*}
Define
\begin{equation*}
\kappa_{N,s}=2\mathrm{e}^{2s}M^{-\frac{2s}{N}}.
\end{equation*}

The following proposition gives the fixed-mass formulation, its scaling
version, and the equivalent scale-invariant form.

\begin{proposition}\label{prop:sharp-log-sobolev}
Let $N\geq1$ and $0<s<1$.
The following equivalent forms of the sharp logarithmic Sobolev inequality hold for every $u\in\mathcal{D}_s\setminus\{0\}$.

\begin{enumerate}
\item It holds that
\begin{equation}\label{eq:mass-log-sob}
	\int_{\mathbb{R}^N}|u|^2\log|u|^2\mathrm{d}x \leq 2A(u)+B(u)\log\left(\frac{B(u)}{M}\right).
\end{equation}
Equality holds if and only if
\begin{equation}\label{eq:mass-extremals}
	u(x)=cQ(x-y)
\end{equation}
for some $c\in\mathbb{R}\setminus\left\{0\right\}$ and $y\in\mathbb{R}^N$.

\item For every $a>0$,
\begin{equation}\label{eq:param-log-sob}
	\begin{aligned}
		\int_{\mathbb{R}^N}|u|^2\log|u|^2\mathrm{d}x \leq& \kappa_{N,s}a^{2s}A(u)+ \left[ \log B(u)-N\left(1+\log a\right) \right]B(u).
	\end{aligned}
\end{equation}
For fixed $a>0$, equality holds if and only if
\begin{equation}\label{eq:param-extremals}
	u(x)=cQ\left( \frac{M^{\frac{1}{N}}}{\mathrm{e} a}(x-y) \right)
\end{equation}
for some $c\in\mathbb{R}\setminus\left\{0\right\}$ and $y\in\mathbb{R}^N$.

\item It holds that
\begin{equation}\label{eq:scale-log-sob}
	\begin{aligned}
		\int_{\mathbb{R}^N}|u|^2\log|u|^2\mathrm{d}x \leq& B(u)\log\left(\frac{B(u)}{M}\right)+ \frac{NB(u)}{2s} \log\left( \frac{4s\mathrm{e}}{N}\frac{A(u)}{B(u)} \right).
	\end{aligned}
\end{equation}
Equality holds if and only if
\begin{equation}\label{eq:scale-extremals}
	u(x)=cQ\left(\lambda(x-y)\right)
\end{equation}
for some $c\in\mathbb{R}\setminus\left\{0\right\}$, $\lambda>0$ and $y\in\mathbb{R}^N$.
\end{enumerate}
\end{proposition}

\begin{proof}
\emph{Proof of (1).}
Set
\begin{equation*}
v=\left(\frac{M}{B(u)}\right)^{\frac{1}{2}}u.
\end{equation*}
Then $B(v)=M$ and Lemma \ref{lem:fmass} gives $A(v)\geq C_0(v)$. Since
\begin{equation*}
A(v)=\frac{M}{B(u)}A(u)
\end{equation*}
and
\begin{equation*}
2C_0(v)=\int_{\mathbb{R}^N}|v|^2\log|v|^2\mathrm{d}x = \frac{M}{B(u)} \int_{\mathbb{R}^N}|u|^2\log|u|^2\mathrm{d}x + M\log\left(\frac{M}{B(u)}\right),
\end{equation*}
we obtain \eqref{eq:mass-log-sob}. 
Equality is equivalent to
\begin{equation*}
A(v)=C_0(v), \quad \mathcal{I}_0(v)=\frac{M}{4}=c_0.
\end{equation*}
Moreover,
\begin{equation*}
A(|v|)\leq A(v)=C_0(v)=C_0(|v|)\leq A(|v|),
\end{equation*}
so equality holds in Lemma \ref{lem:modulus}. Hence $v$ has a constant sign, and Theorem \ref{thm:main}-(2) gives
\begin{equation*}
v(x)=\pm Q(x-y), \quad y\in\mathbb{R}^N.
\end{equation*}
This is equivalent to \eqref{eq:mass-extremals}. The converse follows by substitution.

\emph{Proof that (1) implies (2).}
For $\rho>0$, set
$
u_\rho(x)=\rho^{\frac{N}{2}}u(\rho x).
$
Then
\begin{equation*}
B(u_\rho)=B(u), \quad A(u_\rho)=\rho^{2s}A(u),
\end{equation*}
and
\begin{equation*}
\int_{\mathbb{R}^N}|u_\rho|^2\log|u_\rho|^2\mathrm{d}x = \int_{\mathbb{R}^N}|u|^2\log|u|^2\mathrm{d}x +NB(u)\log\rho.
\end{equation*}
Applying (1) to $u_\rho$ gives
\begin{equation}\label{eq:dilated-log}
\begin{aligned}
	\int_{\mathbb{R}^N}|u|^2\log|u|^2\mathrm{d}x \leq& 2\rho^{2s}A(u) +B(u)\log\left(\frac{B(u)}{M}\right)-NB(u)\log\rho.
\end{aligned}
\end{equation}
Taking $\rho=\mathrm{e} aM^{-\frac{1}{N}}$ in \eqref{eq:dilated-log}, we have 
\begin{equation*}
\begin{aligned}
	\int_{\mathbb{R}^N}|u|^2\log|u|^2\mathrm{d}x
	\leq&\kappa_{N,s}a^{2s}A(u)+B(u)\log\left(\frac{B(u)}M\right)\\
	&-NB(u)(1+\log a)+B(u)\log M\\
	=&\kappa_{N,s}a^{2s}A(u)+\left[\log B(u)-N(1+\log a)\right]B(u),
\end{aligned}
\end{equation*}
which gives \eqref{eq:param-log-sob}. Equality in \eqref{eq:dilated-log} is equivalent to
\begin{equation*}
u_\rho(x)=cQ(x-y),
\end{equation*}
which gives \eqref{eq:param-extremals}.

\emph{Proof that (2) implies (3).}
Since $u\neq0$, we have $A(u)>0$. Define
\begin{equation*}
F(a)=\kappa_{N,s}a^{2s}A(u)-NB(u)\log a,\quad
F'(a)=\frac{1}{a}\left(2s\kappa_{N,s}A(u)a^{2s}-NB(u)\right).
\end{equation*}
Hence $F$ has a unique minimizer $a_*>0$, determined by
\begin{equation}\label{eq:optimal-a}
a_*^{2s}=\frac{NB(u)}{2s\kappa_{N,s}A(u)}.
\end{equation} 
Substituting \eqref{eq:optimal-a} and using
$\kappa_{N,s}=2\mathrm{e}^{2s}M^{-\frac{2s}{N}}$, we obtain
\begin{equation*}
\begin{aligned}
	&\kappa_{N,s}a_*^{2s}A(u)+B(u)\log B(u)-NB(u)-NB(u)\log a_*\\
	&=\frac{NB(u)}{2s}+B(u)\log B(u)-NB(u)-\frac{NB(u)}{2s}\log\left(\frac{NB(u)}{2s\kappa_{N,s}A(u)}\right)\\
	&=\frac{NB(u)}{2s}+B(u)\log B(u)-NB(u)+\frac{NB(u)}{2s}\log\left(\frac{4sA(u)}{NB(u)}\right)+NB(u)-B(u)\log M\\
	&=B(u)\log\frac{B(u)}{M}
	+\frac{NB(u)}{2s}\log\left(\frac{4s\mathrm{e}}{N}\frac{A(u)}{B(u)}\right).
\end{aligned}
\end{equation*}
Substituting $a=a_*$ into \eqref{eq:param-log-sob} gives \eqref{eq:scale-log-sob}. Equality in \eqref{eq:scale-log-sob} holds if and only if equality holds in \eqref{eq:param-log-sob} for $a=a_*$. By the equality characterization in (2), this is equivalent to
\begin{equation*}
u(x)=cQ\left(\frac{M^{\frac{1}{N}}}{\mathrm{e}a_*}(x-y)\right)
\end{equation*}
for some $c\in\mathbb{R}\setminus\{0\}$ and $y\in\mathbb{R}^N$. Thus every equality case of \eqref{eq:scale-log-sob} has the form
\begin{equation*}
u(x)=cQ\left(\lambda(x-y)\right)
\end{equation*}
for some $c\in\mathbb{R}\setminus\{0\}$, $\lambda>0$ and $y\in\mathbb{R}^N$.
Conversely, let
$
u(x)=cQ\left(\lambda(x-y)\right).
$
By Lemma \ref{lem:plog},
\begin{equation*}
\frac{A(u)}{B(u)}=\lambda^{2s}\frac{N}{4s}.
\end{equation*}
Hence \eqref{eq:optimal-a} gives
$
a_*=\mathrm{e}^{-1}M^{\frac{1}{N}}\lambda^{-1},
$
and therefore
\begin{equation*}
\frac{M^{\frac{1}{N}}}{\mathrm{e}a_*}=\lambda.
\end{equation*}
Thus $u$ is an equality case of \eqref{eq:param-log-sob} for the minimizing parameter $a_*$, and consequently equality holds in \eqref{eq:scale-log-sob}. 

\emph{Proof that (3) implies (1).}
Set
\begin{equation*}
t=\frac{4sA(u)}{NB(u)}>0.
\end{equation*}
Since $1+\log t\leq t$, with equality if and only if $t=1$,
\begin{equation*}
\frac{NB(u)}{2s} \log\left(\frac{4s\mathrm{e}}{N}\frac{A(u)}{B(u)}\right) = \frac{NB(u)}{2s}(1+\log t) \leq 2A(u).
\end{equation*}
Substitution in \eqref{eq:scale-log-sob} gives \eqref{eq:mass-log-sob}. Equality in \eqref{eq:mass-log-sob} holds if and only if equality holds in \eqref{eq:scale-log-sob} and $t=1$. By the equality characterization in (3),
\begin{equation*}
u(x)=cQ\left(\lambda(x-y)\right)
\end{equation*}
for some $c\in\mathbb{R}\setminus\{0\}$, $\lambda>0$ and $y\in\mathbb{R}^N$. By Lemma \ref{lem:plog},
\begin{equation*}
t=\frac{4sA(u)}{NB(u)}=\lambda^{2s}.
\end{equation*}
Hence $t=1$ if and only if $\lambda=1$. Therefore equality in \eqref{eq:mass-log-sob} holds if and only if
\begin{equation*}
u(x)=cQ(x-y)
\end{equation*}
for some $c\in\mathbb{R}\setminus\{0\}$ and $y\in\mathbb{R}^N$.
\end{proof}

Specializing the fixed-mass and scale-invariant inequalities to unit mass gives the following normalized forms.

\begin{corollary}
Let $N\geq1$, $0<s<1$, and let $u\in\mathcal{D}_s$ satisfy $B(u)=1$. Then
\begin{equation}\label{eq:unit-log-sob}
\int_{\mathbb{R}^N}|u|^2\log|u|\mathrm{d}x \leq A(u)-\frac{1}{2}\log M .
\end{equation}
Equality holds if and only if
\begin{equation*}
u(x)=\pm M^{-\frac{1}{2}}Q(x-y)
\end{equation*}
for some $y\in\mathbb{R}^N$. Moreover,
\begin{equation}\label{eq:unit-scale-log}
\int_{\mathbb{R}^N}|u|^2\log|u|\mathrm{d}x \leq \frac{N}{4s} \log\left( \frac{4s\mathrm{e}}{N}A(u) \right) - \frac{1}{2}\log M.
\end{equation}
Equality holds if and only if
\begin{equation*}
u(x)= \pm M^{-\frac{1}{2}}\lambda^{\frac{N}{2}} Q\left(\lambda(x-y)\right)
\end{equation*}
for some $\lambda>0$ and $y\in\mathbb{R}^N$.
\end{corollary}

\begin{proof}
Parts (1) and (3) of Proposition \ref{prop:sharp-log-sobolev}, with $B(u)=1$, give
\begin{equation*}
\int_{\mathbb{R}^N}|u|^2\log|u|^2\mathrm{d}x \leq 2A(u)-\log M
\end{equation*}
and
\begin{equation*}
\int_{\mathbb{R}^N}|u|^2\log|u|^2\mathrm{d}x \leq -\log M + \frac{N}{2s} \log\left( \frac{4s\mathrm{e}}{N}A(u) \right).
\end{equation*}
Dividing by $2$ gives \eqref{eq:unit-log-sob} and \eqref{eq:unit-scale-log}. In the first case,
\begin{equation*}
1=B(cQ(\cdot-y))=c^2M,
\end{equation*}
which gives $c=\pm M^{-\frac{1}{2}}$. In the second case,
\begin{equation*}
1=B(cQ(\lambda(\cdot-y)))=c^2\lambda^{-N}M,
\end{equation*}
which gives $c=\pm M^{-\frac{1}{2}}\lambda^{\frac{N}{2}}$. These identities give the
stated equality families.
\end{proof}

For comparison, we recover the classical Gaussian case when $s=1$.

\begin{remark}
For $s=1$, the positive ground state of \eqref{eq:log} and its mass are
\begin{equation*}
Q_1(x)=\exp\left(\frac{N}{2}-\frac{|x|^2}{4}\right), \quad M_1=\mathrm{e}^N(2\pi)^{\frac{N}{2}}.
\end{equation*}
In fact, direct computation gives
\begin{equation*}
-\Delta Q_1=\left(\frac{N}{2}-\frac{|x|^2}{4}\right)Q_1=Q_1\log Q_1,
\end{equation*}
and
\begin{equation*}
\int_{\mathbb{R}^N}Q_1^2\mathrm{d}x=\mathrm{e}^N\int_{\mathbb{R}^N}\mathrm{e}^{-\frac{|x|^2}{2}}\mathrm{d}x=\mathrm{e}^N(2\pi)^{\frac{N}{2}}.
\end{equation*} 
The resulting equality cases agree with the classical
Gaussian extremals in \cite[Theorem 1]{DelPino-Dolbeault-2003} and \cite[Theorem 4]{Carlen-1991}. Then
\eqref{eq:param-log-sob} becomes
\begin{equation*}
\begin{aligned}
	\int_{\mathbb{R}^N}|u|^2\log|u|^2\mathrm{d}x \leq& \frac{a^2}{\pi}\|\nabla u\|_{L^2(\mathbb{R}^N)}^2 
	+ \left[ \log\|u\|_{L^2(\mathbb{R}^N)}^2 -N\left(1+\log a\right) \right] \|u\|_{L^2(\mathbb{R}^N)}^2.
\end{aligned}
\end{equation*}
For $\|u\|_{L^2(\mathbb{R}^N)}=1$, \eqref{eq:unit-log-sob} and \eqref{eq:unit-scale-log} become
\begin{equation*}
\int_{\mathbb{R}^N}|u|^2\log|u|\mathrm{d}x+\frac{N}{2}\left(1+\frac12\log(2\pi)\right)\leq\|\nabla u\|_{L^2(\mathbb{R}^N)}^2
\end{equation*}
and
\begin{equation*}
\int_{\mathbb{R}^N}|u|^2\log|u|\mathrm{d}x\leq\frac{N}{4}\log\left(\frac{2}{\pi N\mathrm e}\|\nabla u\|_{L^2(\mathbb{R}^N)}^2\right).
\end{equation*} 
The equality cases of the first inequality under the normalization
$\|u\|_{L^2(\mathbb{R}^N)}=1$ are
\begin{equation*}
u(x)=\pm(2\pi)^{-\frac{N}{4}}\exp\left(-\frac{|x-y|^2}{4}\right),\quad y\in\mathbb{R}^N.
\end{equation*}
The equality cases of the scale-invariant inequality under the same normalization are
\begin{equation*}
u(x)=\pm(\pi\sigma)^{-\frac{N}{4}}\exp\left(-\frac{|x-y|^2}{2\sigma}\right),\quad \sigma>0,\quad y\in\mathbb{R}^N.
\end{equation*}
\end{remark}

\medskip
\subsection*{Acknowledgements} 
Shuangjie Peng was supported by National Key R\&D Program (No. 2023YFA1010002).
Xiaoming An was supported by National Natural Science Foundation of China (No. 12561033).

\vskip 0.2truein

\noindent {\bf Data Availability Statement\ } Data sharing not applicable to this article as no
datasets were generated or analysed during the current study.

\vskip 0.2truein
\noindent {\bf Declarations}

\vskip 0.1truein
\noindent {\bf Conflict of interest\ } The authors have no conflict of interest to declare that are
relevant to the content of this article.


\vskip 0.35truein

\medskip	















\begin{thebibliography}{99}


\bibitem{Almgren-Lieb-1989}
F. J. Almgren Jr. and E. H. Lieb, 
Symmetric decreasing rearrangement is sometimes continuous, 
J. Amer. Math. Soc. {\bf 2} (1989), no. 4, 683--773.

\bibitem{Amick-Toland-ActaMath-1991}
C. J. Amick and J. F. Toland, 
Uniqueness and related analytic properties for the Benjamin-Ono equation---a nonlinear Neumann problem in the plane, 
Acta Math. {\bf 167} (1991), no. 1-2, 107--126.

\bibitem{An-Fang-2025}
X. An and Y. Fang,
Existence and uniqueness of positive ground state solutions of general logarithmic Schr\"odinger equations,
J. Differential Equations {\bf 422} (2025), 57--77.

\bibitem{An-Peng-Yang-Zhong-2026}
X. An, S. Peng, X. Yang, and F. Zhong,
Qualitative analysis for ground state solutions of logarithmic Schr\"odinger equations under a small constant magnetic field in $\mathbb{R}^N$,
J. Funct. Anal. {\bf 290} (2026), no. 7, Paper No. 111333, 51 pp.

\bibitem{An-Yang-2023}
X. An and X. Yang,
Convergence from power-law to logarithm-law in nonlinear fractional Schr\"odinger equations,
J. Math. Phys. {\bf 64} (2023), no. 1, Paper No. 011506, 13 pp.

\bibitem{Ardila-2017}
A. H. Ardila,
Existence and stability of standing waves for nonlinear fractional Schr\"odinger equation with logarithmic nonlinearity,
Nonlinear Anal. {\bf 155} (2017), 52--64.

\bibitem{Berestycki-Lions-1983a}
H. Berestycki and P.-L. Lions, 
Nonlinear scalar field equations. I. Existence of a ground state, 
Arch. Rational Mech. Anal. {\bf 82} (1983), no. 4, 313--345.

\bibitem{Berestycki-Lions-1983b}
H. Berestycki and P.-L. Lions, 
Nonlinear scalar field equations. II. Existence of infinitely many solutions, 
Arch. Rational Mech. Anal. {\bf 82} (1983), no. 4, 347--375.

\bibitem{Caffarelli-Silverstre-CPDE-2007}
L. \'A. Caffarelli and L. E. Silvestre, 
An extension problem related to the fractional Laplacian, 
Comm. Partial Differential Equations {\bf 32} (2007), no. 7-9, 1245--1260.

\bibitem{Carlen-1991}
E. A. Carlen,
Superadditivity of Fisher's information and logarithmic Sobolev inequalities,
J. Funct. Anal. {\bf 101} (1991), no. 1, 194--211.

\bibitem{Chang-Gustafson-Nakanishi-Tsai-2007}
S.-M. Chang, S. Gustafson, K. Nakanishi and T.-P. Tsai,
Spectra of linearized operators for NLS solitary waves, 
SIAM J. Math. Anal. {\bf 39} (2007/08), no. 4, 1070--1111.

\bibitem{Chatzakou-Ruzhansky-2024}
M. Chatzakou and M. V. Ruzhansky, 
Revised logarithmic Sobolev inequalities of fractional order, 
Bull. Sci. Math. {\bf 197} (2024), Paper No. 103530, 9 pp.

\bibitem{Chen-Li-CPAM-2006}
W. Chen, C. Li and B. Ou, 
Classification of solutions for an integral equation, 
Comm. Pure Appl. Math. {\bf 59} (2006), no. 3, 330--343.

\bibitem{Chen-Li-2018}
W. Chen and C. Li, 
Maximum principles for the fractional $p$-Laplacian and symmetry of solutions, 
Adv. Math. {\bf 335} (2018), 735--758.

\bibitem{Coffman-ARMA-1972}
C. V. Coffman, 
Uniqueness of the ground state solution for $\Delta u-u+u\sp{3}=0$\ and a variational characterization of other solutions, 
Arch. Rational Mech. Anal. {\bf 46} (1972), 81--95.

\bibitem{Coffman-JDE-1996}
C. V. Coffman, 
Uniqueness of the positive radial solution on an annulus of the Dirichlet problem for $\Delta u-u+u^3=0$, 
J. Differential Equations {\bf 128} (1996), no. 2, 379--386.

\bibitem{Cotsiolis-Tavoularis-2005}
A. Cotsiolis and N. K. Tavoularis,
On logarithmic Sobolev inequalities for higher order fractional derivatives,
C. R. Math. Acad. Sci. Paris {\bf 340} (2005), no. 3, 205--208.

\bibitem{DAvenia-Squassina-Zenari-2015}
P. d'Avenia, M. Squassina and M. Zenari,
Fractional logarithmic Schr\"odinger equations,
Math. Methods Appl. Sci. {\bf 38} (2015), no. 18, 5207--5216.

\bibitem{DelPino-Dolbeault-2003}
M. A. del Pino and J. Dolbeault,
The optimal Euclidean $L^p$-Sobolev logarithmic inequality,
J. Funct. Anal. {\bf 197} (2003), no. 1, 151--161.

\bibitem{Nezza-BSM-2012}
E. Di Nezza, G. Palatucci and E. Valdinoci, 
Hitchhiker's guide to the fractional Sobolev spaces, 
Bull. Sci. Math. {\bf 136} (2012), no. 5, 521--573. 

\bibitem{Frank-Lenzmann-2013}
R. L. Frank and E. Lenzmann,
Uniqueness of non-linear ground states for fractional Laplacians in $\mathbb{R}$,
Acta Math. {\bf 210} (2013), no. 2, 261--318.

\bibitem{Frank-Lenzmann-Silvestre-2016}
R. L. Frank, E. Lenzmann and L. E. Silvestre,
Uniqueness of radial solutions for the fractional Laplacian,
Comm. Pure Appl. Math. {\bf 69} (2016), no. 9, 1671--1726.

\bibitem{Frank-Seiringer-2008}
R. L. Frank and R. Seiringer,
Non-linear ground state representations and sharp Hardy inequalities,
J. Funct. Anal. {\bf 255} (2008), no. 12, 3407--3430.

\bibitem{Gross-1975}
L. Gross,
Logarithmic Sobolev inequalities,
Amer. J. Math. {\bf 97} (1975), no. 4, 1061--1083.

\bibitem{Kato-1995}
T. Kato, 
{\it Perturbation theory for linear operators}, 
reprint of the 1980 edition, 
Classics in Mathematics, Springer, Berlin, 1995.

\bibitem{Kenig-Martel-Robbiano-2011}
C. E. Kenig, Y. Martel and L. Robbiano, 
Local well-posedness and blow-up in the energy space for a class of $L^2$ critical dispersion generalized Benjamin-Ono equations, 
Ann. Inst. H. Poincar\'e{} C Anal. Non Lin\'eaire {\bf 28} (2011), no. 6, 853--887.

\bibitem{Kwong-1989}
M. K. Kwong, 
Uniqueness of positive solutions of $\Delta u-u+u^p=0$ in ${\bf R}^n$, 
Arch. Rational Mech. Anal. {\bf 105} (1989), no. 3, 243--266.

\bibitem{Kwong-Zhang-DIE-1991}
M. K. Kwong and L. Q. Zhang, 
Uniqueness of the positive solution of $\Delta u+f(u)=0$ in an annulus, 
Differential Integral Equations {\bf 4} (1991), no. 3, 583--599.

\bibitem{Leoni-2017}
G. Leoni, 
{\it A first course in Sobolev spaces}, second edition, 
Graduate Studies in Mathematics, 181, Amer. Math. Soc., Providence, RI, 2017.

\bibitem{Li-Peng-Shuai-2022}
Q. Li, S. J. Peng and W. Shuai,
On fractional logarithmic Schr\"odinger equations,
Adv. Nonlinear Stud. {\bf 22} (2022), no. 1, 41--66.

\bibitem{Liu-Sun-Zou-2026}
T. Liu, X. Sun, and W. Zou, 
Uniqueness of bound states to the logarithmic Schr\"odinger equation,
\emph{arXiv preprint} arXiv:2606.19077 (2026).

\bibitem{Lieb-Loss-2001}
E. H. Lieb and M. Loss,
{\it Analysis}, second edition,
Graduate Studies in Mathematics, 14, Amer. Math. Soc., Providence, RI, 2001.

\bibitem{Li-JEMS-2004}
Y. Y. Li, 
Remark on some conformally invariant integral equations: the method of moving spheres, 
J. Eur. Math. Soc. (JEMS) {\bf 6} (2004), no. 2, 153--180.

\bibitem{Laskin-2000}
N. Laskin, 
Fractals and quantum mechanics, 
Chaos {\bf 10} (2000), no. 4, 780--790.

\bibitem{Laskin-2002}
N. Laskin, 
Fractional Schr\"odinger equation,
Phys. Rev. E (3) {\bf 66} (2002), no. 5, 056108, 7 pp.

\bibitem{Mcleod-Serrin-ARMA-1987}
K. McLeod and J. B. Serrin Jr., 
Uniqueness of positive radial solutions of $\Delta u+f(u)=0$ in ${\bf R}^n$, 
Arch. Rational Mech. Anal. {\bf 99} (1987), no. 2, 115--145. 

\bibitem{Reed-Simon-1978-1}
M. C. Reed and B. Simon, 
{\it Methods of modern mathematical physics. I}, second edition, 
Academic Press, New York, 1980.

\bibitem{Reed-Simon-1978-4}
M. C. Reed and B. Simon, 
{\it Methods of modern mathematical physics. IV. Analysis of operators}, 
Academic Press, New York-London, 1978. 

\bibitem{RosOton-Serra-2016}
X. Ros-Oton and J. Serra,
Regularity theory for general stable operators,
J. Differential Equations {\bf 260} (2016), no. 12, 8675--8715.

\bibitem{Serrin-Tang-Indiana-2000}
J. B. Serrin Jr. and M. Tang, 
Uniqueness of ground states for quasilinear elliptic equations, 
Indiana Univ. Math. J. {\bf 49} (2000), no. 3, 897--923.

\bibitem{Tang-Invention-2026}
M. Tang, 
Uniqueness of bound states to $\Delta u-u+|u|^{p-1}u=0$ in $\mathbb{R}^{n}$, $n\geq3$, 
Invent. Math. {\bf 243} (2026), no. 2, 245--291.

\bibitem{Troy-ARMA-2016}
W. C. Troy, 
Uniqueness of positive ground state solutions of the logarithmic Schr\"odinger equation, 
Arch. Ration. Mech. Anal. {\bf 222} (2016), no. 3, 1581--1600.

\bibitem{Wang-Zhang-2019}
Z. Q. Wang and C. Zhang,
Convergence from power-law to logarithm-law in nonlinear scalar field equations,
Arch. Ration. Mech. Anal. {\bf 231} (2019), no. 1, 45--61.

\bibitem{Weinstein-1985}
M. I. Weinstein, 
Modulational stability of ground states of nonlinear Schr\"odinger equations, 
SIAM J. Math. Anal. {\bf 16} (1985), no. 3, 472--491.

\bibitem{Weinstein-1987}
M. I. Weinstein, 
Existence and dynamic stability of solitary wave solutions of equations arising in long wave propagation, 
Comm. Partial Differential Equations {\bf 12} (1987), no. 10, 1133--1173.

\bibitem{Zhang-Zhang-JFPT-2022}
C. Zhang and L. Zhang, 
Qualitative analysis on logarithmic Schr\"odinger equation with general potential, 
J. Fixed Point Theory Appl. {\bf 24} (2022), no. 4, Paper No. 74, 33 pp.


\end{thebibliography}
\end{document}